\documentclass[a4paper]{amsart}
\usepackage{amsfonts}
\usepackage{amsthm}
\usepackage{amsmath}
\usepackage{amssymb}
\usepackage{graphicx}
\usepackage{layout}
\usepackage{mathrsfs}
\usepackage{epsfig}
\usepackage{indentfirst,latexsym,bm}
\usepackage{pifont}
\usepackage{color}
\usepackage{txfonts}
\begin{document}
\def\dis{\displaystyle}
\renewcommand{\theequation}{\thesection.\arabic{equation}}
\newtheorem{thm}{Theorem}[section]
\newtheorem{lemma}[thm]{Lemma}
\newtheorem{cor}[thm]{Corollary}
\newtheorem{pro}[thm]{Proposition}
\newtheorem{remark}[thm]{Remark}
\newtheorem{defin}[thm]{Definition}

\title[]{Stabilities of background perturbations for compressible adiabatic flows through porous media in 3-D}
\author{Shijin Deng}
\address{School of Mathematical Sciences, CMA-Shanghai, Shanghai Jiao Tong University}
\author{Xiaochun Yang}
%\thanks{* Corresponding author.}
\thanks{Email addresses: matdengs@sjtu.edu.cn (Shijin Deng), yangxc0220@sjtu.edu.cn (Xiaochun Yang).}
\address{School of Mathematical Sciences, Shanghai Jiao Tong University}
\date{}

\maketitle

\noindent{\bf Abstract.} In this paper, we study the background stability problem for three-dimensional compressible adiabatic flows through porous media. For two given equilibrium states $(\bar{P}^k,\vec{0},\bar{S}^k)$ ($k=0,1$), under the assumption of small initial perturbations around each equilibrium state, pointwise estimates of solutions for individual perturbation problems are obtained. Based on these pointwise structures, a comparison between two individual solutions could be applied and it is proved that the difference between two solutions depends continuously on the equilibrium distance $|\bar{P}^0-\bar{P}^1|+|\bar{S}^0-\bar{S}^1|$ as well as the initial perturbation difference, with additional decay rates, thereby establishing a nonlinear comparison principle for this system. Compared with previous works on background stability problems, it is the first time to obtain stabilities in the pointwise sense. Our main tools are Green's functions, a priori energy estimates, and refined decay estimates for higher‑order derivatives.

\vskip .05in 


%\noindent {\bf 2000 MR Subject Classification}: 

\section{Introduction}
The compressible adiabatic flows through porous media could be described by the following equations:
\begin{equation} \label{pro1}
\begin{cases}
	\partial_t \rho+div\left(\rho\vec{u}\right)=0, \quad (\vec{x} , t) \in \mathbb{R}^3 \times \mathbb{R}_{+},\\
	\partial_t\left(\rho\vec u\right)+div \left(\rho\vec{u} \otimes \vec{u}\right)+\nabla P=-a\rho\vec u,\\
	\partial_t\left(\rho\mathcal{E}\right)+div \left(\rho\mathcal{E}\vec{u}+P\vec{u}\right)=-a\rho\left|\vec u\right|^2.
\end{cases}
\end{equation}
Here, $\rho, \vec{u}\equiv (u_1,u_2,u_3)^t, P$ and $e$ denote the density, the velocity, the pressure and 
the internal energy respectively, and the total energy $\mathcal{E}$ is given by $\mathcal{E}\equiv \frac 12\left|\vec u\right|^2+e$. The constant $a>0$ is the friction constant. In this paper, we only consider the polytropic fluids and thus the state equations read
$$P=R\rho\theta,\quad e=C_v \theta,$$
with the universal gas constant $R>0$ and the specific heat $C_v>0$ at constant volume.

For isentropic flows, the system \eqref{pro1} reduces to the well-known compressible Euler equations with damping,
\begin{equation}\nonumber
	\begin{cases}
		\partial_t \rho+div\left(\rho\vec{u}\right)=0,\\
		\partial_t\left(\rho\vec u\right)+div \left(\rho\vec{u} \otimes \vec{u}\right)+\nabla P=-a\rho\vec u,
	\end{cases}
\end{equation}
which has been widely investigated. Global existence of smooth solutions for small initial data in the one-dimensional setting was established by Nishida \cite{n,n1}. The existence theory and large-time behaviors for the 1-D case and the multi-D case were studied in \cite{hp2,hp,hp1,mm1,m,mm,mm2,mmr,m1,ny,pz,zh1,zh2} and \cite{d3,fx,hyz,k,lwy,lw1,pz1,svs,tw,wy,wly,xz,zsl} respectively (see also references therein). 

For $s \neq \text{constant}$, it is the system \eqref{pro1} for the adiabatic flows which is totally different from the classical compressible Euler equations with damping. The system \eqref{pro1} contains zero eigenvalues, which correspond to non-decaying components at the linear level and may cause troubles for the existence of the nonlinear solution. Fortunately, for this system, the non-decaying components could always be decoupled.  In the one-dimensional setting, its Lagrangian formulation decouples the non-decaying part, and there are huge results about 1-D case. For the 1-D Cauchy problem, results include the existence theory and also large-time behaviors of solutions \cite{gw,hs,hl,hs1,mp,nn,p1,zh}. For the half-space problems, decay rates of asymptotic profiles were obtained in \cite{gw1,p}. 

The three-dimensional case was first studied in \cite{wth}. The authors draw on the experience of Duan and Ma’s work \cite{dm} on the compressible Navier–Stokes equations, and treat the thermodynamic quantity $\rho$ as a function of $p$ and $s$ to convert system \eqref{pro1} into the $(p,\vec{u},s)$ system and achieve the linear decoupling of decaying and non-decaying components. In \cite{qw}, this strategy is extended to a class of partially dissipative quasilinear hyperbolic systems. The authors introduce a (weak) linear degeneracy condition under which a partially dissipative quasilinear hyperbolic system can be decomposed into a subsystem satisfying the Kawashima condition and a remaining subsystem that violates the Kawashima condition but is (weakly) linearly degenerate. This (weak) linear degeneracy condition applies to the damped full Euler system of adiabatic flows studied in \cite{wth}; however, it does not apply to the three-dimensional thermal nonequilibrium flow model \cite{z1}, which %admits a triple zero eigenvalue and 
retains a double zero eigenvalue even after linear decoupling from the entropy subsystem. This necessitates additional structural conditions such as the irrotational assumption introduced in \cite{hl1}. Related work on the 3-D adiabatic flows could be also found in \cite{d,wm,yz}. In \cite{wm}, the $L^1$ boundedness condition on initial data required in \cite{wth} is eliminated through a frequency-decomposed energy approach; in \cite{yz}, weighted energy estimates and interpolation arguments are applied for decay rates of the pressure and velocity in the weighted Lebesgue space $L_{\gamma}^2$; and the half-space problem for 3-D is studied in \cite {d} by the Green's function for the initial-boundary value problem which is constructed through a Laplace–Fourier transform converting the half-space problem into a whole-space one.

As shown in \cite{wth}, if one chooses dependent variables $P, \vec u$ and the entropy $S$ with the state equation for the gas given by
\begin{equation}\label{state}
\rho=\kappa P^{\frac{C_v}{C_v+R}}\exp{\left(-\frac{S}{C_v+R}\right)}
\end{equation}
with a positive constant $\kappa$, it leads to a decoupling of $(P,\vec u)$ and the non-decaying component $S$ in the linear level. In the rest of the paper, we follow this choice, and the equations \eqref{pro1} in terms of new unknowns $P, \vec{u}$ and $S$ read
\begin{equation}\label{proo}
\begin{cases}
\partial_t P+\frac{R+C_v}{C_v} P div\vec{u}+\vec{u} \cdot \nabla P=0,\\
\partial_t \vec{u}+\vec{u}\cdot \nabla \vec{u}+\frac{\nabla P}{\rho}=-a\vec{u},\\
\partial_t S+\vec{u}\cdot\nabla S=0,
\end{cases}
\end{equation}
with $\rho=\rho(P,S)$ given by \eqref{state}.

In this paper, we are interested in the Cauchy problems for \eqref{proo} with the following two initial conditions and denote the solutions as $\bigl(P^k,\vec u^k,S^k\bigr)(\vec{x},t)$:
\begin{equation}\label{ini12}
\left(P^k,\vec u^k,S^k\right) \! (\vec{x},0)=\left(p^k_0,\vec u^k_0,s^k_0\right) \! (\vec{x})+\left(\bar{P}^k,\vec 0,\bar{S}^k\right), \ \ k=0,1,
\end{equation}
and aim to establish a comparison principle between the two corresponding solutions. Here, $\bigl(p^k_0,\vec u^k_0,s^k_0\bigr)$ denote initial perturbation functions and $\bigl(\bar{P}^k,\vec 0,\bar{S}^k\bigr)$ ($k=0,1$) are two different constant solutions of \eqref{proo}.

In applications, background states such as $\bigl(\bar{P}^k,\vec 0,\bar{S}^k\bigr)$ ($k=0,1$) are usually obtained from measurements and thus inherent errors cannot be avoided, which naturally leads to background stability problems. In \cite{lww}, the background stability problem for the Boltzmann equation is investigated, and a refined Green's function method suitable for such problems is proposed. In \cite{dy}, both the refined Green's function method and a refined energy method are used to study the background perturbation problem for the compressible Navier-Stokes equations. There is also related work on potential stability problems \cite{ltw} and force stability problems \cite{lw}. Recently, we established the background stability for the three-dimensional thermal nonequilibrium flow model in \cite{dy2}, where a new mechanism concerning models with zero eigenvalues appears. In this paper, the adiabatic flow model also contains zero eigenvalues and the study here enriches the understanding of the background stability problems for partially dissipative quasilinear hyperbolic systems.

It requires very precise estimates for the individual perturbation problem to derive a comparison principle incorporating the spatial and temporal information. In \cite{yz}, the obtained decay rates $t^{-3/4+\gamma/2}$ for the pressure and velocity in the weighted Lebesgue space $L_{\gamma}^2$ and the ones $t^{-5/4+\gamma/2}$ for the $k$-th ($1\le k\le 3$) order spatial derivatives $\nabla^k(p,u)$ are not adequate for the comparison; and thus we seek more decaying structures for the individual perturbation problems before the comparison. In fact, in Section~\ref{sec:2}, the Green’s function method combined with energy estimates and refined decay rates of high-order derivatives is applied to derive pointwise space-time estimates of $(P^k-\bar{P}^k,\vec{u}^k,S^k-\bar{S}^k)$ and their higher-order derivatives, which provide sufficient information for the study of the comparison problem in Section~\ref{sec:3}.

The main result of this paper is as follows:
\begin{thm}
\label{main}
Suppose that the two constant solutions $\bigl(\bar{P}^k,\vec 0,\bar{S}^k\bigr)$ $(k=0,1)$ satisfy $\bigl| \bar{P}^0-\bar{P}^1\bigr|, \bigl|\bar{S}^0-\bar{S}^1\bigr|\le \delta$ for a positive constant $\delta\ll 1$ and the initial perturbation functions $\bigl(p_0^k,\vec u_0^k,s_0^k\bigr)(\vec{x})$ in \eqref{ini12} satisfy
\begin{equation}\label{inicon}
\begin{cases}
\left\|\left(p_0^k, \vec{u}_0^k\right) \right\|_{H^8} + \sum_{|\alpha|=0}^5 \left\| e^{|\vec x|} \partial_{\vec x}^{\alpha} \left(p^k_0,\vec u^k_0\right) \! (\vec x)\right\|_{L^\infty} +\sum_{|\alpha|=0}^{7} \left\| e^{|\vec{x}|} \partial_{\vec x}^{\alpha} s_0^k (\vec x) \right\|_{L^2} \equiv \epsilon_k \ll 1\\
\\
\left\|\left(p_0^0-p_0^1, \vec{u}_0^0-\vec{u}_0^1, s_0^0-s_0^1\right) \right\|_{H^6}+\sum_{|\alpha|=0}^2 \left\| e^{|\vec x|} \partial_{\vec x}^{\alpha}\left(p^0_0-p^1_0,\vec u^0_0-\vec u^1_0\right) \! (\vec x)\right\|_{L^\infty}+\sum_{|\alpha|=0}^{2} \left\| e^{|\vec{x}|} \partial_{\vec x}^{\alpha} \left( s_0^0-s_0^1\right) \! (\vec x)\right\|_{L^2} \le \epsilon_{01} \delta_0,
\end{cases}
\end{equation}
for positive constants $\delta_0$, and $\epsilon_{01}:=\epsilon_0+\epsilon_1 \ll 1$. Then there exists a positive constant $\epsilon_*\ll 1$ such that if $\epsilon_{01}, \delta_0< \epsilon_*$, there exists a constant $D_c>0$ such that the solutions $\bigl(P^k,\vec u^k,S^k\bigr)(\vec x,t)$ $(k=0,1)$ of the Cauchy problems \eqref{proo}--\eqref{ini12} satisfy
\begin{equation}\label{puscpoint}
	\begin{aligned}
		&\left|\left(P^0-P^1\right) \!(\vec x,t)-\left(\bar{P}^0-\bar{P}^1\right)\right|\le D_c\epsilon_{01}\left(\delta_0+\delta\right)\left((1+t)^{-3/2}e^{-\frac{|\vec{x}|^2}{D_c t}} + e^{-(|\vec{x}|+t)/D_c} \right),\\
		&\left|\left(\vec u^0-\vec u^1\right) \!(\vec x,t)\right|\le D_c\epsilon_{01}\left(\delta_0+\delta\right)\left((1+t)^{-2}\log(2+t)e^{-\frac{|\vec{x}|^2}{D_c t}} + e^{-(|\vec{x}|+t)/D_c} \right),\\
		&\left|\left(S^0-S^1\right) \!(\vec x,t)-\left(\bar{S}^0-\bar{S}^1\right)\right|\le D_c\epsilon_{01}\left(\delta_0+\delta\right)e^{-|\vec{x}|}.
	\end{aligned}
\end{equation}
\end{thm}

\begin{remark}
Under the assumption $\bar{\vec u}^0=\bar{\vec u}^1(=\vec 0)$, there is almost no loss in decay rates compared with the original ones obtained in Theorem~\ref{pointwise} for individual problems. As shown in \cite{lww}, this is the optimal result in the sense that it matches the Taylor expansion of heat kernels with respect to the diffusion constant. It is the first time to obtain these optimal decaying structures for a complex model; the previous results for both the Boltzmann equation \cite{lww} and Navier-Stokes equations \cite{dy} show a growing factor $t^{1/2}$ to compensate for a dependence on the difference of background states.

The last inequality of \eqref{puscpoint} reveals a uniform dependence with respect to $t$ on the difference of background states for the non-decaying component $S^0-S^1$. This is different from the result obtained in \cite{dy2} for the three-dimensional thermal nonequilibrium flow model, which is another type of partially dissipative system.
\end{remark}

The rest of the paper is arranged as follows. In Section~\ref{sec:2}, we quote the basic properties of the Green's functions for individual perturbation problems and establish pointwise estimates for solutions of the corresponding nonlinear problems. Proofs of the standard energy estimates and temporal decay estimates, and computational lemmas concerning nonlinear wave interactions are given in Appendix I and Appendix II, respectively.
In Section~\ref{sec:3}, we establish the comparison principle in the pointwise sense by the refined Green's function method and refined energy estimates, together with the required decay rates of higher‑order derivatives.

{\bf Notations.} For a multi-index $\alpha := (\alpha_1, \alpha_2, \alpha_3)$, one denotes
$$\partial_{\vec{x}}^\alpha := \partial_{x_1}^{\alpha_1}\partial_{x_2}^{\alpha_2}\partial_{x_3}^{\alpha_3},\quad \partial_{\vec{x}}:= (\partial_{x_1},\partial_{x_2},\partial_{x_3}),\quad |\alpha| := \alpha_1+\alpha_2+\alpha_3.$$
The symbol $\nabla_{\vec{x}}^m$ ($m\in\mathbb{N}^+$), abbreviated as $\nabla^m$, represents all $m$th-order derivatives with respect to the spatial variable $\vec{x}$. For a vector $\vec{x} =(x_1, x_2, x_3)^t \in \mathbb{R}^3$, its Euclidean norm is denoted as $|\vec{x}| \equiv \bigl(\sum_{i=1}^{3} (x_i)^2\bigr)^{1/2}$. The notations $\left\| \cdot \right\|_{L^p}$ and $\left\| \cdot \right\|_{H^m}$ denote the classical norms on $L^p_{\vec{x}} (\mathbb{R}^3)$ and Sobolev space $H^m_{\vec{x}} (\mathbb{R}^3)$, respectively. The Fourier transform of a function $f$ is denoted by $\mathscr{F}[f](\vec{\xi})$, defined as
$$\mathscr{F}[f](\vec{\xi}) := \int_{\mathbb{R}^3} e^{-i\vec{\xi} \cdot \vec{x}} f(\vec{x}) \, d\vec{x},$$
with the Fourier variable $\vec{\xi} = (\xi_1, \xi_2, \xi_3)^t \in \mathbb{R}^3$ and $\mathscr{F}^{-1}$ denotes the inverse Fourier transform.

The notation $f = \mathcal{O}(1) g$ indicates that there exists a positive constant $C $ such that $|f| \le C |g|$. Throughout this paper, $C > 0$ denotes a generic constant and its value may vary from line to line.

\setcounter{equation}{0}

\section{Green's Function and Pointwise Space-Time Estimates for Solutions of the Individual Perturbation Problems} \label{sec:2}
In this section, we consider the solutions $\bigl(P^k,\vec u^k,S^k\bigr)(\vec{x},t)$ of the Cauchy problems \eqref{proo}--\eqref{ini12} in the perturbation form: for $k=0,1$,
$$\left(p^k,\vec u^k,s^k\right) \! (\vec{x},t)\equiv\left(P^k,\vec u^k,S^k\right) \! (\vec{x},t)-\left(\bar{P}^k,\vec 0,\bar{S}^k\right),$$
and the equations for the perturbations $\bigl(p^k,\vec u^k,s^k\bigr) (k=0,1)$ are as follows:
\begin{equation}\label{pk}
	\begin{cases}
		\partial_t p^k+\frac{\bar{P}^k\left(R+C_v\right)}{C_v} div\vec u^k=-\frac{R+C_v}{C_v} p^k div\vec u^k-\vec u^k\cdot\nabla p^k:=R_p^k,\\
		\partial_t \vec{u}^k+\frac{1}{\bar{\rho}^k}\nabla p^k+a\vec{u}^k=-\vec{u}^k\cdot\nabla \vec u^k+\frac{\rho^k-\bar{\rho}^k}{\bar{\rho}^k\rho^k}\nabla p^k:=	R_{\vec{u}}^k,\\
		\partial_t s^k=-\vec u^k\cdot\nabla s^k,\\
		\left(p^k,\vec u^k,s^k\right) \! (\vec x,0)=\left(p^k_0,\vec u^k_0,s^k_0\right) \!(\vec x).
	\end{cases}
\end{equation}
Here, the relations among $\rho^k, \bar{\rho}^k$ and the corresponding $\bigl(P^k,S^k\bigr), \bigl(\bar{P}^k,\bar{S}^k\bigr)$ are given by the identity \eqref{state}. The pointwise estimates for the individual solutions $\bigl(p^k,\vec{u}^k,s^k\bigr)$ are as follows:
\begin{thm}[\textbf{Pointwise Space-Time Structures}] \label{pointwise}
	Suppose that the initial data $\bigl(p_0^k,\vec{u}_0^k,s_0^k\bigr)(\vec{x})$ in \eqref{pk} satisfy
	\begin{equation} \label{initial}
		\left\|\left(p_0^k, \vec{u}_0^k\right) \right\|_{H^8} + \sum_{|\alpha|=0}^5 \left\| e^{|\vec x|} \partial_{\vec x}^{\alpha} \left(p^k_0,\vec u^k_0\right) \!(\vec x)\right\|_{L^\infty} +\sum_{|\alpha|=0}^{7} \left\| e^{|\vec{x}|} \partial_{\vec x}^{\alpha} s_0^k (\vec x)\right\|_{L^2} \equiv \epsilon_k \ll 1.
	\end{equation}
	Then there exist positive constants $D_k$ such that the solutions $\bigl(p^k,\vec{u}^k,s^k\bigr)$ $(k=0,1)$ of the Cauchy problem \eqref{pk} satisfy
	\begin{equation}\label{poinconclu}
		\begin{aligned}
			&\left| p^k(\vec{x},t)\right| \le D_k \epsilon_k \left( (1+t)^{-3/2} e^{-\frac{|\vec{x}|^2}{D_k t}} + e^{-(|\vec{x}+t)/D_k}\right),\\
			& \left|\vec{u}^k(\vec{x},t)\right|, \left|\nabla^i (p^k,\vec{u}^k)(\vec{x},t)\right| \le D_k \epsilon_k \left( (1+t)^{-2} e^{-\frac{|\vec{x}|^2}{D_k t}} + e^{-(|\vec{x}+t)/D_k}\right), \quad \text{for } 1 \le i \le 3,\\
			& \left|\nabla^j s^k (\vec{x})\right| \le D_k\epsilon_k e^{-|\vec{x}|}, \quad \text{for } 0 \le j \le 5.
		\end{aligned}
	\end{equation}
\end{thm}

\subsection{Green's Functions}
Define $\mathbb{F}^k$ ($k=0,1$) to be the solution of the following Cauchy problem
\begin{equation}\label{f}
\begin{cases}
\partial_t\mathbb{F}^k+\begin{pmatrix}
0&\frac{\bar{P}^k\left(R+C_v\right)}{C_v} \partial_{\vec{x}}\\
\frac{1}{\bar{\rho}^k}\partial_{\vec{x}}^t &aI_n
\end{pmatrix}\mathbb{F}^k=0,\\
\mathbb{F}^k(\vec x,0)=\delta(\vec x)I_{n+1},
\end{cases}
\end{equation}
where $\delta(\vec x)$ is the 3-D Dirac delta function and $I_{n+1}$ denotes the $(n+1)\times (n+1)$ identity matrix. Then one could define $\mathbb{G}^k(\vec x,t)$ as follows: 
\begin{equation}\nonumber
\mathbb{G}^k(\vec x,t)\equiv \begin{pmatrix}
\mathbb{F}^k(\vec x,t)&0\\
0&\delta(\vec x)
\end{pmatrix},
\end{equation}
and $\mathbb{G}^k(\vec x,t)$ is the Green's function of \eqref{pk} for  $k=0,1$.  

One could apply Fourier transform to \eqref{f} and follow the steps in \cite{d} to construct the Green's functions $\mathbb{F}^k(\vec x,t)$ for $k=0,1$. Denote
\begin{equation}\label{fks}
\mathbb{F}^k_S(\vec{x},t)\equiv 
\begin{pmatrix}
	\mathbb{j}_0^{k*}(\vec{x},t)&-\frac{\bar{P}^k (R+C_v)}{C_v}\partial_{\vec{x}}\mathbb{j}_1^{k*}(\vec{x},t)\\
\\
	-\frac{1}{\bar{\rho}^k}\partial^t_{\vec{x}}\mathbb{j}_1^{k*}(\vec{x},t)&e^{-at}\delta(\vec{x})I_n+\mathbb{j}_{ml}^{k*}(\vec{x},t)
\end{pmatrix},\end{equation}
and
$$\begin{cases}
\mathbb{j}_0^{k*}(\vec{x},t)\equiv e^{-at/2}\left(a_{0S}^k(t)\mathbb{S}^k(\vec{x},t;c_k)+a_{0C}^{k}(t)\mathbb{C}^k(\vec{x},t;c_k)+ \left(a_{0SY}^k(t)\mathbb{S}^k(\vec{x},t;c_k)+a_{0CY}^{k}(t)\mathbb{C}^k(\vec{x},t;c_k)\right)\underset{\vec{x}}{\ast} \mathbb{Y}(\vec{x})\right),\\
\\
\mathbb{j}_1^{k*}(\vec{x},t)\equiv e^{-at/2}\left(a_{1S}^k(t)\mathbb{S}^k(\vec{x},t;c_k)+a_{1C}^{k}(t)\mathbb{C}^k(\vec{x},t;c_k)+ 
\left(a_{1SY}^k(t)\mathbb{S}^k(\vec{x},t;c_k)+a_{1CY}^{k}(t)\mathbb{C}^k(\vec{x},t;c_k)\right) \underset{\vec{x}}{\ast} \mathbb{Y}(\vec{x})\right),\\
\\
\mathbb{j}_{ml}^{k*}(\vec{x},t)\equiv \left(\mathbb{j}_0^{k*}-a\mathbb{j}_1^{k*}\right) \! (\vec{x},t) \underset{\vec{x}}{\ast} \partial_{x_m x_l}\mathbb{Y}(\vec{x}),
\end{cases}$$
\begin{equation}\label{skck}
	\begin{cases}
		\mathbb{S}^k(\vec{x},t;c_k)\equiv \mathscr{F}^{-1}\left[\sin \left(c_k|\vec{\xi}|t\right)/|\vec{\xi}|\right],\\
		\mathbb{C}^k(\vec{x},t;c_k)\equiv \mathscr{F}^{-1}\left[\cos \left(c_k|\vec{\xi}|t\right)\right],\\
		c_k \equiv \sqrt{\frac{\bar{P}^k(R+C_v)}{C_v\bar{\rho}^k}},\\
		\mathbb{Y}(\vec{x})\equiv \mathscr{F}^{-1}\left[\left(1+|\vec{\xi}|^2\right)^{-1}\right]=\frac{e^{-|\vec{x}|}}{4\pi|\vec{x}|}.
	\end{cases}
\end{equation}
Here, the coefficients $a_{0(1)S}^k(t), a_{0(1)C}^k(t), a_{0(1)SY}^k(t), a_{0(1)CY}^k(t)$ are polynomials in $t$.
The pointwise structures of Green's functions $\mathbb{F}^k(\vec x,t)$ are given as follows:
\begin{thm}[Theorem~2.3, \cite{d}]\label{thm2.1}
There exists a positive constant $D>1$ such that $\mathbb{F}^k(\vec x,t)$ given by \eqref{f} satisfy that
for $0\le|\alpha|\le 5$,
\begin{equation}\nonumber\left|\partial_{\vec{x}}^\alpha \left(\mathbb{F}^k-\mathbb{F}^{k*}\right)_{ij}(\vec{x},t)\right|
\le D\left(1+t\right)^{-(3+|\alpha|)/2}e^{-\frac{|\vec{x}|^2}{Dt}} \cdot
\begin{cases}
	1,\ \ i=j=1,\\
	(1+t)^{-1/2},\ \ i=1, j\neq 1, \ \text{ or } j=1, i\neq 1,\\
	(1+t)^{-1},\ \ i\neq 1, j\neq 1,
\end{cases}
\end{equation}
\begin{equation}\nonumber\left|\mathbb{F}^{k*}(\vec{x},t)-\mathbb{F}_S^k(\vec{x},t)\right|\le De^{-(|\vec{x}|+t)/D},\end{equation}
where $A_{ij}$ denotes the $(i,j)$-entry of a matrix $A$ and $\mathbb{F}^k_S(\vec{x},t)$ given by \eqref{fks} is the singular part 
of $\mathbb{F}^k(\vec{x},t)$.
\end{thm}

Based on the pointwise structure of the Green's function $\mathbb{F}^k(\vec{x},t)$ given above, one could obtain its $L^p$ estimates as follows:
\begin{thm}\label{greenlp} 
	For a vector function $\vec{V}(\vec{x}) \equiv (p_0, \vec{u}_0)(\vec{x})=(p_0, {u}_0^1, {u}_0^2, {u}_0^3)(\vec{x})$,
	there exist positive constants $\tilde{D}$ and $R_0$ such that for $|\alpha| \ge 0$ and $q \ge 1$,
	\begin{equation}\label{gfkl00}
		\left\|\partial_{\vec{x}}^\alpha \left(\mathbb{F}_i^k(\vec{x},t) \underset{\vec{x}}{\ast} \vec{V}(\vec{x})\right)\right\|_{L^\infty} \le C(1+t)^{-\frac{3}{2q}-\frac{|\alpha|}{2}-d_i} \left(\left\|p_0 \right\|_{L^q}+(1+t)^{-1/2} \left\|\vec{u}_0 \right\|_{L^q}\right) + Ce^{-t/\tilde{D}} 
		\begin{cases}
			\left\|\partial_{\vec{x}}^\alpha p_0\right\|_{W^{1,\infty}}+\left\|\partial_{\vec{x}}^\alpha \vec{u}_0\right\|_{W^{2,\infty}}, \  \text{for } i=1,\\
			\left\|\partial_{\vec{x}}^\alpha \vec{V} \right\|_{W^{2,\infty}}, \ \ \ \ \ \ \ \ \ \ \  \ \ \ \ \ \ \ \ \ \ \ \  \text{for } i=2,
		\end{cases}
	\end{equation}
		\begin{equation}\label{gf2kl2}
		\left\|\partial_{\vec{x}}^\alpha \left(\mathbb{F}_i^k(\vec{x},t) \underset{\vec{x}}{\ast} \vec{V}(\vec{x})\right)\right\|_{L^2} \le Ce^{-R_0t} \left\|\nabla^{|\alpha|} \vec{V}\right\|_{L^2} +C
		\begin{cases}
			(1+t)^{-\frac{|\alpha|}{2}}\left(\left\|p_0\right\|_{L^2} + (1+t)^{-\frac{1}{2}}\left\|\vec{u}_0\right\|_{L^2} \right), & \text{for } i=1,\\
			(1+t)^{-\frac{1+|\alpha|}{2}}\left\|\vec{V}\right\|_{L^2}, & \text{for } i=2.
		\end{cases}
	\end{equation}
	Moreover, the $L^2$ estimates could be improved under the $L^1$ boundedness assumption on $V$\textup{:}
	\begin{equation}\label{gf1kl2}
		\left\|\partial_{\vec{x}}^\alpha \left(\mathbb{F}_i^k(\vec{x},t) \underset{\vec{x}}{\ast} \vec{V}(\vec{x})\right)\right\|_{L^2} \le Ce^{-R_0t} \left\|\nabla^{|\alpha|} \vec{V}\right\|_{L^2} +C
		\begin{cases}
			(1+t)^{-\frac{3}{4}-\frac{|\alpha|}{2}}\left(\left\|p_0\right\|_{L^1} + (1+t)^{-\frac{1}{2}}\left\|\vec{u}_0\right\|_{L^1} \right), & \text{for } i=1,\\
			(1+t)^{-\frac{5}{4}-\frac{|\alpha|}{2}} \left\|\vec{V}\right\|_{L^1}, & \text{for } i=2.
		\end{cases}
	\end{equation}
	Here, $\mathbb{F}_i^k$ denotes the $i$-th row of the second‑order block matrix $\mathbb{F}^k$. The parameter $d_i$, which distinguishes the decay rates of $\mathbb{F}_1^k \ast \vec{V}$ and $\mathbb{F}_2^k \ast \vec{V}$, is defined by
	\begin{equation}\label{di}
		d_1=0, \quad d_2=1/2.
	\end{equation}
\end{thm}

\begin{proof}
	By virtue of Theorem~\ref{thm2.1}, \eqref{fks}, and Kirchhoff's formula \eqref{sg}--\eqref{cg}, together with the following estimates obtained via polar coordinate transformation and the residue theorem, it follows that
	\begin{equation} \nonumber %\label{y3sc}
		\begin{aligned}
			&\mathbb{Y}(\vec{x}) \underset{\vec{x}}{\ast} \mathbb{S}^k (\vec{x},t;c_k) = \frac{1}{8 \pi |\vec{x}|} \left( e^{-\bigl| |\vec{x}| - c_k t \bigr|} - e^{-(|\vec{x}| + c_k t)} \right),\\ &\mathbb{Y}(\vec{x}) \underset{\vec{x}}{\ast} \mathbb{C}^k(\vec{x},t;c_k) = \frac{1}{8 \pi |\vec{x}|} \left( e^{-(|\vec{x}| + c_k t)} + \operatorname{sgn}\bigl(|\vec{x}| - c_k t\bigr) e^{-\bigl| |\vec{x}| - c_k t \bigr|} \right),\\
			&\mathbb{Y}(\vec{x}) \underset{\vec{x}}{\ast} \mathbb{Y}(\vec{x}) \underset{\vec{x}}{\ast} \mathbb{S}^k(\vec{x},t;c_k) = \frac{1}{16 \pi |\vec{x}|} \left[ \left( \bigl| |\vec{x}| - c_k t \bigr| + 1 \right) e^{-\bigl| |\vec{x}| - c_k t \bigr|} - (|\vec{x}| + c_k t + 1)  e^{-(|\vec{x}| + c_k t)} \right],\\
			&\mathbb{Y}(\vec{x}) \underset{\vec{x}}{\ast} \mathbb{Y}(\vec{x}) \underset{\vec{x}}{\ast} \mathbb{C}^k (\vec{x},t;c_k) = \frac{1}{16 \pi |\vec{x}|} \left[ (|\vec{x}| + c_k t) e^{-(|\vec{x}| + c_k t)} + (|\vec{x}| - c_k t)  e^{-\bigl| |\vec{x}| - c_k t \bigr|} \right],		
		\end{aligned}
	\end{equation}
	and thus
	\begin{equation}\label{fs1k}
		\begin{aligned}
			\left|\left( \mathbb{F}_{S,1}^k (\vec{x},t) \underset{\vec{x}}{\ast} \partial_{\vec{x}}^\alpha \vec{V}(\vec{x})\right)\right|
			& \le C(1+t)P_k(t) e^{-at/2} \int_{|\vec{y}|=1} \left(\left|\partial_{\vec{x}}^\alpha p_0\right|+\left|\nabla \partial_{\vec{x}}^\alpha p_0\right|+\left|div\left(\partial_{\vec{x}}^\alpha \vec{u}_0\right) \right|+ \left|\nabla div \left(\partial_{\vec{x}}^\alpha \vec{u}_0 \right)\right| \right)(\vec{x}+c_k t\vec{y}) \, dS(\vec{y})\\
			&\quad + C\left(1+|\vec{x}|^{-1}\right)e^{-at/8}\left( e^{-|\vec{x}|/2}+e^{-a|\vec{x}|/(16c_k)}\right) \underset{\vec{x}}{\ast} \left( \left|\partial_{\vec{x}}^\alpha p_0(\vec{x})\right|+\left|div \left(\partial_{\vec{x}}^\alpha \vec{u}_0(\vec{x}) \right)\right|\right),
		\end{aligned}
	\end{equation}
	\begin{equation}\label{fs2k}
		\begin{aligned}
			\left|\left( \mathbb{F}_{S,2}^k (\vec{x},t) \underset{\vec{x}}{\ast} \partial_{\vec{x}}^\alpha \vec{V}(\vec{x})\right)\right|
			& \le C(1+t)\left(\left|a_{1S}^k(t)\right|+\left|a_{1C}^k(t)\right| \right) e^{-at/2} \int_{|\vec{y}|=1} \left(\left|\nabla \partial_{\vec{x}}^\alpha p_0\right|+\left|\nabla^2 \partial_{\vec{x}}^\alpha p_0\right|\right)(\vec{x}+c_k t\vec{y}) \, dS(\vec{y})\\
			&\quad + C\left(1+|\vec{x}|^{-1}\right)e^{-at/8}\left( e^{-|\vec{x}|/4}+e^{-a|\vec{x}|/(16c_k)}\right) \underset{\vec{x}}{\ast} \left( \left|\nabla \partial_{\vec{x}}^\alpha p_0(\vec{x})\right|+\left|\nabla div \left(\partial_{\vec{x}}^\alpha \vec{u}_0(\vec{x}) \right)\right|\right)+\left| e^{-at} \partial_{\vec{x}}^\alpha \vec{u}_0(\vec{x})\right|,
		\end{aligned}
	\end{equation}
	where $c_k$ is defined in \eqref{skck}, and $P_k(t)=\bigl|a_{0S}^k(t)\bigr|+\bigl|a_{0C}^k(t)\bigr|+\bar{P}^k C_v^{-1} (R+C_v) \left(\bigl|a_{1S}^k(t)\bigr|+\bigl|a_{1C}^k(t)\bigr|\right)$. The convolutions of the regular parts of Green's function and the given function $\vec{V}(\vec x)$ could be estimated as follows:
	\begin{equation}\label{greenfi}
		\begin{aligned}
			\left|\partial_{\vec{x}}^\alpha \left( \left( \mathbb{F}_i^k-\mathbb{F}_{S,i}^k\right) \!(\vec{x},t) \underset{\vec{x}}{\ast} \vec{V}(\vec{x})\right)\right| &\le \left|\partial_{\vec{x}}^\alpha \left( \mathbb{F}_i^k-\mathbb{F}_i^{k*} \right) \! (\vec{x},t) \underset{\vec{x}}{\ast} \vec{V}(\vec{x})\right| +	\left|\left( \mathbb{F}_i^{k*}-\mathbb{F}_{S,i}^k\right) \! (\vec{x},t) \underset{\vec{x}}{\ast} \partial_{\vec{x}}^\alpha \vec{V}(\vec{x}) \right|\\
			&\le C\left[(1+t)^{-(3+|\alpha|)/2-d_i} e^{-\frac{|\vec{x}|^2}{Dt}} \underset{\vec{x}}{\ast} \left(\left|p_0(\vec{x})\right|+(1+t)^{-1/2}\left| \vec{u}_0(\vec{x})\right|\right)\right.\\
			& \left.\quad \quad \quad+ e^{-(|\vec{x}|+t)/D} \underset{\vec{x}}{\ast} \left|\partial_{\vec{x}}^\alpha \vec{V}(\vec{x})\right|\right].
		\end{aligned}
	\end{equation}
	 Combining \eqref{fs1k}--\eqref{greenfi} with Young's inequality yields
	\begin{equation}\nonumber
		\begin{aligned}
			&\left\|\partial_{\vec{x}}^\alpha \left(\mathbb{F}_i^k(\vec{x},t) \underset{\vec{x}}{\ast} \vec{V}(\vec{x})\right)\right\|_{L^\infty} \le \left\|\partial_{\vec{x}}^\alpha \left( \left( \mathbb{F}_i^k-\mathbb{F}_{S,i}^k\right) \!(\vec{x},t) \underset{\vec{x}}{\ast} \vec{V}(\vec{x})\right)\right\|_{L^\infty} +\left\|\left( \mathbb{F}_{S,i}^k (\vec{x},t) \underset{\vec{x}}{\ast} \partial_{\vec{x}}^\alpha \vec{V}(\vec{x})\right)\right\|_{L^\infty}\\
			&\le C\left[\left\|(1+t)^{-(3+|\alpha|)/2-d_i} e^{-\frac{|\vec{x}|^2}{Dt}} \right\|_{L^r} \left(\left\|p_0 \right\|_{L^q}+(1+t)^{-1/2} \left\|\vec{u}_0 \right\|_{L^q}\right) + \left\|e^{-(|\vec{x}|+t)/D}\right\|_{L^1} \cdot \left\|\partial_{\vec{x}}^\alpha \vec{V}\right\|_{L^\infty} \right]\\
			&\quad + C
			\begin{cases}
				e^{-at/8}\left(1+\left\|\left(1+|\vec{x}|^{-1}\right)\left(e^{-|\vec{x}|/2}+e^{-a|\vec{x}|/(16c_k)}\right)\right\|_{L^1}\right) \left(\left\|\partial_{\vec{x}}^\alpha p_0\right\|_{W^{1,\infty}}+\left\|div\left(\partial_{\vec{x}}^\alpha \vec{u}_0\right) \right\|_{W^{1,\infty}}\right) & \text{for } i=1,\\
				e^{-at/8}\left(1+\left\|\left(1+|\vec{x}|^{-1}\right)\left(e^{-|\vec{x}|/4}+e^{-a|\vec{x}|/(16c_k)}\right)\right\|_{L^1}\right) \left(\left\|\nabla \partial_{\vec{x}}^\alpha p_0\right\|_{W^{1,\infty}}+\left\|\nabla div \partial_{\vec{x}}^\alpha \vec{u}_0\right\|_{L^\infty}+\left\|\partial_{\vec{x}}^\alpha \vec{u}_0\right\|_{L^\infty}\right) & \text{for } i=2
			\end{cases}\\
			&\le C(1+t)^{-\frac{3}{2q}-\frac{|\alpha|}{2}-d_i} \left(\left\|p_0 \right\|_{L^q}+(1+t)^{-1/2} \left\|\vec{u}_0 \right\|_{L^q}\right) + Ce^{-t/\tilde{D}} 
				\begin{cases}
				\left\|\partial_{\vec{x}}^\alpha p_0\right\|_{W^{1,\infty}}+\left\|\partial_{\vec{x}}^\alpha \vec{u}_0\right\|_{W^{2,\infty}} & \text{for } i=1,\\
				\left\|\partial_{\vec{x}}^\alpha p_0\right\|_{W^{2,\infty}}+\left\|\partial_{\vec{x}}^\alpha \vec{u}_0\right\|_{W^{2,\infty}} & \text{for } i=2,
			\end{cases}
		\end{aligned}
	\end{equation}
where $1/r+1/q=1$ and $\tilde{D}:=\max\{D,8/a\}$. This estimate verifies \eqref{gfkl00}.

As given in \cite{d}, the Fourier transform of $\mathbb{F}^k$ is as follows:
\begin{equation}\nonumber
\mathscr{F}\left[\mathbb{F}^k\right]=\begin{pmatrix}
\frac{\lambda_+^ke^{\lambda_-^kt}-\lambda_-^ke^{\lambda_+^kt}}{\lambda_+^k-\lambda_-^k}&\frac{i\bar{P}^k(R+C_v)\vec\xi^t\left(e^{\lambda_-^kt}-e^{\lambda_+^kt}\right)}{C_v\left(\lambda_+^k-\lambda_-^k\right)}\\
\frac{i\vec\xi\left(e^{\lambda_-^kt}-e^{\lambda_+^kt}\right)}{\bar{\rho}^k\left(\lambda_+^k-\lambda_-^k\right)}&
\frac{\lambda_+^ke^{\lambda_+^kt}-\lambda_-^ke^{\lambda_-^kt}}{\lambda_+^k-\lambda_-^k}\frac{\vec\xi\vec\xi^t}{|\vec \xi|^2}+
\left(I_n-\frac{\vec\xi\vec\xi^t}{|\vec \xi|^2}\right)e^{-at}
\end{pmatrix}
\end{equation}
with $\lambda_{\pm}^k\equiv -\frac a2\pm\sqrt{\frac{a^2}4-\frac{\bar{P}^k(R+C_v)}{\bar{\rho}^kC_v}|\vec\xi|^2}$, and thus together with Parseval's identity one has that
\begin{equation}\nonumber
	\begin{aligned}
		\left\|\partial_{\vec{x}}^\alpha \left(\mathbb{F}_1^k(\vec{x},t) \underset{\vec{x}}{\ast} \vec{V}(\vec{x})\right)\right\|_{L^2}^2 &\le C\int_{|\vec{\xi}|\le \eta}  e^{-c_k^2|\vec{\xi}|^2 t/a}  |\vec{\xi}|^{2|\alpha|}\left( \bigl|\mathscr{F}[p_0]\bigr|^2 + |\vec{\xi}|^2 \left||\vec{\xi}|^{-1} \mathscr{F}[div \vec{u}_0]\right|^2 \right) d\vec{\xi}\\
		&\quad +C\int_{|\vec{\xi}| > \eta} e^{-2R_0 t} |\vec{\xi}|^{2|\alpha|}\left( \bigl|\mathscr{F}[p_0]\bigr|^2 + \left||\vec{\xi}|^{-1} \mathscr{F}[div \vec{u}_0]\right|^2 \right) d\vec{\xi}\\
		&\le C(1+t)^{-|\alpha|-\frac{3}{2}}\left(\bigl\|\mathscr{F}[p_0]\bigr\|_{L_{\vec{\xi}}^\infty}^2 + (1+t)^{-1}\left\|\mathscr{F}[\vec{u}_0]\right\|_{L_{\vec{\xi}}^\infty}^2 \right) + Ce^{-2R_0t} \left\|\nabla^{|\alpha|} \vec{V}\right\|_{L^2}^2\\
		&\le C(1+t)^{-|\alpha|-\frac{3}{2}}\left(\left\|p_0\right\|_{L^1}^2 + (1+t)^{-1}\left\|\vec{u}_0\right\|_{L^1}^2 \right) + Ce^{-2R_0t} \left\|\nabla^{|\alpha|} \vec{V}\right\|_{L^2}^2,
	\end{aligned}
\end{equation}
for a positive constant $R_0$, which also admits the following alternative form 
\begin{equation}\nonumber
	\begin{aligned}
		\left\|\partial_{\vec{x}}^\alpha \left(\mathbb{F}_1^k(\vec{x},t) \underset{\vec{x}}{\ast} \vec{V}(\vec{x})\right)\right\|_{L^2}^2 
		&\le C(1+t)^{-|\alpha|}\left(\left\|p_0\right\|_{L^2}^2 + (1+t)^{-1}\left\|\vec{u}_0\right\|_{L^2}^2 \right) +Ce^{-2R_0t} \left\|\nabla^{|\alpha|} \vec{V}\right\|_{L^2}^2.
	\end{aligned}
\end{equation}
Similar estimates hold for $\partial_{\vec{x}}^\alpha \bigl(\mathbb{F}_2^k(\vec{x},t)\underset{\vec{x}}{\ast}\vec{V}(\vec{x})\bigr)$:
\begin{equation}\nonumber
	\begin{aligned}
		\left\|\partial_{\vec{x}}^\alpha \left(\mathbb{F}_2^k(\vec{x},t) \underset{\vec{x}}{\ast} \vec{V}(\vec{x})\right)\right\|_{L^2}^2 &\le C\int_{|\vec{\xi}|\le \eta}  e^{-c_k^2|\vec{\xi}|^2 t/a}  |\vec{\xi}|^{2|\alpha|+2} \bigl|\mathscr{F}[p_0]\bigr|^2  d\vec{\xi} + C\int_{|\vec{\xi}|\le \eta} \left(e^{-at}+|\vec{\xi}|^2 e^{-c_k^2|\vec{\xi}|^2 t/a} \right)|\vec{\xi}|^{2|\alpha|}  \left| \mathscr{F}[\vec{u}_0]\right|^2 d\vec{\xi}\\
		&\quad +C\int_{|\vec{\xi}| > \eta} e^{-2R_0 t} |\vec{\xi}|^{2|\alpha|}\left( \bigl|\mathscr{F}[p_0]\bigr|^2 + \left|\mathscr{F}[\vec{u}_0]\right|^2 \right) d\vec{\xi}\\
		&\le C\left[(1+t)^{-|\alpha|-\frac{5}{2}} \left\|\vec V\right\|_{L^1}^2  + e^{-2R_0t} \left\|\nabla^{|\alpha|} \vec{V}\right\|_{L^2}^2\right],
	\end{aligned}
\end{equation}

\begin{equation}\nonumber
	\begin{aligned}
		\left\|\partial_{\vec{x}}^\alpha \left(\mathbb{F}_2^k(\vec{x},t) \underset{\vec{x}}{\ast} \vec{V}(\vec{x})\right)\right\|_{L^2}^2 
		&\le C\left[(1+t)^{-|\alpha|-1} \left\|\vec V\right\|_{L^2}^2 + e^{-2R_0t} \left\|\nabla^{|\alpha|} \vec{V}\right\|_{L^2}^2\right].
	\end{aligned}
\end{equation}
Thus, \eqref{gf2kl2} and \eqref{gf1kl2} follow directly.
\end{proof}

In the above proof of Theorem \ref{greenlp}, Kirchhoff's formula is applied to deal with the convolutions between singular parts of $\mathbb{F}^k$ and the initial functions $\vec{V}(\vec{x})$; and it is quoted below for the convenience of the readers:
\begin{cor}[\textbf{Kirchhoff's Formula}, Theorem~4.1, \cite{z1}]\label{kirch}
	Let $\mathbb{S}^k = \mathbb{S}^k(\vec{x}, t;c_k)$ and $\mathbb{C}^k = \mathbb{C}^k(\vec{x}, t;c_k)$ be as defined in \eqref{skck}. The inverse Fourier transforms of $\mathscr{F}[\mathbb{S}^k] \cdot \mathscr{F}[g]$ and $\mathscr{F}[\mathbb{C}^k] \cdot \mathscr{F}[g]$ for a given function $g(\vec{x})$ are
	\begin{equation} \label{sg}
		\bigl(\mathbb{S}^k \underset{\vec{x}}{\ast} g\bigr)(\vec{x}, t) = \frac{c_kt}{4 \pi} \int_{|\vec{y}|=1} g(\vec{x} + c_kt\vec{y}) \, dS(\vec{y}),
	\end{equation}
	\begin{equation} \label{cg}
		\bigl(\mathbb{C}^k \underset{\vec{x}}{\ast} g \bigr)(\vec{x}, t) = \frac{1}{4 \pi} \int_{|\vec{y}|=1} g(\vec{x} + c_k t\vec{y}) \, dS(\vec{y}) + \frac{c_k t}{4 \pi} \int_{|\vec{y}|=1} \nabla g(\vec{x} + c_k t\vec{y}) \cdot \vec{y} \, dS(\vec{y}),
	\end{equation}
	respectively.
\end{cor}

\subsection{Pointwise Estimates for Solutions of the Individual Perturbation Problems: Proof of Theorem~\ref{pointwise}}
By Duhamel's principle combined with the Green's function constructed in Theorem~\ref{thm2.1}, one has the following representation
for the solution $(p^k,\vec{u}^k)$ of \eqref{pk}:
$$\vec{V}^k(\vec{x},t) = \mathbb{F}^k(\vec{x},t) \underset{\vec{x}}{\ast} \vec{V}^k(\vec{x},0) + \int_0^t \mathbb{F}^k(\vec{x},t-\tau) \underset{\vec{x}}{\ast} 
\begin{pmatrix}
	R_p^k \vspace{1.5mm}\\
	R_{\vec{u}}^k
\end{pmatrix}(\vec{x},\tau) \, d\tau,$$
and thus
\begin{equation} \label{duhpu}
	\begin{aligned}
		\partial_{\vec{x}}^\alpha \vec{V}^k(\vec{x},t) &= \partial_{\vec{x}}^\alpha \mathbb{F}^k (\vec{x},t) \underset{\vec{x}}{\ast} \vec{V}^k(\vec{x},0) + \int_0^t \partial_{\vec{x}}^\alpha \mathbb{F}^k (\vec{x},t-\tau) \underset{\vec{x}}{\ast} 
		\begin{pmatrix}
			R_p^k \vspace{1.5mm}\\
			R_{\vec{u}}^k
		\end{pmatrix}(\vec{x},\tau) \, d\tau\\
		&= \partial_{\vec{x}}^\alpha \left( \mathbb{F}^k-\mathbb{F}^{k*}\right) \!(\vec{x},t) \underset{\vec{x}}{\ast} \vec{V}^k(\vec{x},0) + \left(\mathbb{F}^{k*}-\mathbb{F}_S^k + \mathbb{F}_S^k \right) \! (\vec{x},t) \underset{\vec{x}}{\ast} \partial_{\vec{x}}^\alpha \vec{V}^k(\vec{x},0) + \partial_{\vec{x}}^\alpha \vec{S}_R^k (\vec{x},t).
	\end{aligned}
\end{equation}
Here, one denotes
\begin{equation} \label{v}
	\vec{V}^k:=\left(\vec{V}_1^k,\vec{V}_2^k\right)^t:=\bigl(p^k,\vec{u}^k\bigr)^t,\ \ \ \  \vec{S}_R^k:=\left( \vec{S}_{R,1}^k, \vec{S}_{R,2}^k\right)^t:=\int_0^t  \mathbb{F}^k(\vec{x},t-\tau) \underset{\vec{x}}{\ast} 
	\begin{pmatrix}
		R_p^k \vspace{1.5mm}\\
		R_{\vec{u}}^k
	\end{pmatrix}(\vec{x},\tau) \, d\tau,
\end{equation}
where the source terms $R_p^k$ and $R_{\vec{u}}^k$ are defined in \eqref{pk}.

Since the system \eqref{pk} is quasi-linear and thus there are higher-order derivatives in the solution representation \eqref{duhpu}, usually it should be estimated by the Green's function method combined with a priori bounded estimates for high order derivatives. Here, it is a real combination of two methods, the Green's function method and the energy method, in the sense that the boundedness of the solutions and the derivatives cannot be concluded only from the energy method. The reason is that for the system \eqref{pk}, there is a non-decaying component in the linear level and one could not decouple the nonlinear system \eqref{pk}. To overcome this trouble, one thus introduces the following time-dependent functional $\mathcal{N}^k(t)$. The desired pointwise space–time structures are characterized via the weight functions $w^k_i(\vec x,t)$ ($i=0,1,2,3,k=0,1$):
\begin{equation} \label{anasn}
	\begin{aligned}
		\mathcal{N}^k(t)&:= \underset{0 \le \tau \le t}{\sup} \Bigg\lbrace \left\|\left(\left(w^k_{0}\right)^{-1} (1+\tau)^{-1/2}p^k\right) \!(\cdot, \tau)\right\|_{L^{\infty}} + \left\|\left(\left(w^k_{0}\right)^{-1} \left|\vec{u}^k\right| \right) \!(\cdot, \tau)\right\|_{L^{\infty}}+ \sum_{i=1}^3 \left\|\left(\left(w^k_i\right)^{-1}\left( \left|\nabla^i p^k\right| + \left|\nabla^i \vec{u}^k\right| \right)\right) \! (\cdot, \tau)\right\|_{L^{\infty}}\Bigg\rbrace
	\end{aligned}
\end{equation}
with
\begin{equation}\label{w0123}
		w_i^k (\vec{x}, t):=w^k(\vec{x},t;C_{i,k},D_{i,k}):=(1+t)^{-2} e^{-\frac{|\vec{x}|^2}{C_{i,k} t}} + e^{-(|\vec{x}|+t)/D_{i,k}},
\end{equation}
where $C_{i,k}$ and $D_{i,k}$ for $i=0,1,2,3$ are positive constants satisfying $C_{i,k}=\mu^{2i} C_{0,k}$ and $D_{i,k}=\mu^i D_{0,k}$ for $i=1,2,3$, with $\mu=(3\sqrt{6}/7)^{1/3}>1$, and
\begin{equation}\label{cd0k}
	C_{0,k} \ge 6D_{0,k}, \quad D_{0,k} :=4\max\big\lbrace 4/(\mu-1), D/(\mu-1),2/(a(\mu-1)),16c_k/(a(\mu-1))\big\rbrace.
\end{equation}
Here, the choice of $\mu$ and the relations for $C_{i,k}$ and $D_{i,k}$ are not essential and are imposed only to simplify computations.\\

\textbf{Energy Estimates and Temporal Decay Estimates.}
First, one could replace $H^3(\mathbb{R}^3)$ in Theorem~1.1 of \cite{wth} with $H^l(\mathbb{R}^3)$ for $l\ge 3$ and obtain the following generalized global existence result. The detailed proof is given in Appendix I.
\begin{thm}[\textbf{Basic Energy Estimates}: a modification of Theorem~1.1, \cite{wth}] \label{genes} 
	Let $\bar{P}^k$, $\bar{S}^k$ be positive constants and the initial functions $\bigl(p_0^k, \vec{u}_0^k, s_0^k\bigr)$ satisfy that $\bigl\|\bigl(p_0^k, \vec{u}_0^k\bigr)\bigr\|_{L^1}$ is bounded and $\bigl\|\bigl(p_0^k, \vec{u}_0^k, s_0^k\bigr)\bigr\|_{H^l}$ is sufficiently small for some integer $l \ge 3$.
	Then the Cauchy problem \eqref{pk} admits a unique global solution $\bigl(p^k, \vec{u}^k, s^k\bigr)$ satisfying $p^k+\bar{P}^k >0$ and
	\begin{equation} \nonumber
		\left(p^k, \vec{u}^k, s^k\right) \in C^0([ 0, \infty); H^l(\mathbb{R}^3)) \cap C^1( [0, \infty); H^{l-1}(\mathbb{R}^3)).
	\end{equation}
	Moreover, there exists a constant $C_0>0$ such that
	\begin{equation} \label{genenes}
		\begin{aligned}
			& \left\| \left(p^k, \vec{u}^k\right) \!(\cdot,t) \right\|_{H^l}^2+ \int_0^t \left( \left\| \nabla p^k(\cdot,\tau) \right\|^2_{H^{l-1}} + \left\| \vec{u}^k(\cdot,\tau) \right\|^2_{H^l} \right) d\tau \le C_0 \left\|\left(p_0^k, \vec{u}_0^k\right)\right\|_{H^l}^2\\
			&\left\|s^k(\cdot,t) \right\|_{H^l}^2 \le C_0 \exp\left\lbrace C_0 \left\|\left(p_0^k, \vec{u}_0^k\right)\right\|_{L^1 \cap H^l}\right\rbrace \left\|s_0^k\right\|_{H^l}^2.
		\end{aligned}
	\end{equation}
	And there exists a constant $C_1>0$ such that the solution $\bigl(p^k, \vec{u}^k\bigr)$ satisfies the following decay property\textup{:}
	\begin{equation} \label{pude}
		\left\| \vec{u}^k(\cdot, t) \right\|_{H^l} + \left\| \nabla p^k(\cdot, t) \right\|_{H^{l-1}} \le C_1 (1 + t)^{-\frac{5}{4}} \left\|\left(p_0^k, \vec{u}_0^k\right)\right\|_{L^1 \cap H^l}, \quad \forall t \ge 0.
	\end{equation}
\end{thm}

Based on Theorem \ref{genes}, temporal decay estimates for higher-order derivatives are then established as follows and it is an improvement of \eqref{pude}:
\begin{lemma}[\textbf{Decay Properties}] \label{decay}
	Under the assumptions of Theorem~\ref{pointwise}, the following temporal decay estimates are valid for the solution $\bigl(p^k, \vec{u}^k\bigr)$ and its derivatives: for all $t \ge 0$,
	\begin{equation} \label{decaygen2}
		\left\| \nabla \vec{u}^k(\cdot, t) \right\|_{H^{l_2-1}} + \left\| \nabla^2 p^k(\cdot, t) \right\|_{H^{l_2 - 2}} \le C(1 + t)^{-\frac{7}{4}} \left\| \left( p_0^k, \vec{u}_0^k\right)  \right\|_{L^1 \cap H^{l_2}},
	\end{equation}
	and
	\begin{equation} \label{decaygen3}
		\left\| \nabla^2 \vec{u}^k(\cdot, t) \right\|_{H^{l_3 - 2}} + \left\| \nabla^3 p^k(\cdot, t) \right\|_{H^{l_3 - 3}} \le C(1 + t)^{-2} \left( \left\| \left( p_0^k, \vec{u}_0^k\right)  \right\|_{L^1 \cap H^{l_3}} + \mathcal{N}^k(t)\right).
	\end{equation}
	Here, $l_i$ denote regularity indices of the Sobolev spaces, with $l_i-i \ge 1$ for $i=2,3$.
\end{lemma}

\textbf{Pointwise Estimates for the Entropy Function $s^k(\vec{x},t)$ and Its Derivatives.} Multiply the entropy equation 
$$\left( \partial_{\vec{x}}^{\alpha} s^k \right)_t=-\vec u^k\cdot\nabla \partial_{\vec{x}}^{\alpha} s^k - \left[ \partial_{\vec{x}}^{\alpha}, \vec u^k\cdot\nabla\right]s^k$$
by $e^{2|\vec{x}|} \partial_{\vec{x}}^\alpha s^k$ and integrate the product with respect to the spatial variable $\vec{x}$ over $\mathbb{R}^3$ to obtain
\begin{equation}\label{weightsk}
	\begin{aligned}
		\frac{1}{2} \frac{d}{dt} \left\| e^{|\vec{x}|}\partial_{\vec{x}}^\alpha s^k \right\|_{L^2}^2&=- \int_{\mathbb{R}^3} e^{2|\vec{x}|} \left( \frac{1}{2}\vec{u}^k \cdot \nabla \left(\partial_{\vec{x}}^\alpha s^k \right)^2 + \partial_{\vec{x}}^\alpha s^k \left[ \partial_{\vec{x}}^{\alpha}, \vec u^k\cdot\nabla\right]s^k \right) d\vec{x}\\
		& \le \frac{1}{2} \int_{\mathbb{R}^3} div\left( e^{2|\vec{x}|} \vec{u}^k \right) \left(\partial_{\vec{x}}^\alpha s^k \right)^2 d\vec{x} + \left\| e^{|\vec{x}|}\partial_{\vec{x}}^\alpha s^k \right\|_{L^2} \cdot  \left\| e^{|\vec{x}|} \left[ \partial_{\vec{x}}^{\alpha}, \vec u^k\cdot\nabla\right]s^k \right\|_{L^2}.
	\end{aligned}
\end{equation}
By virtue of the following Sobolev embedding inequality
\begin{equation}\label{weightsob}
	\left\| e^{|\vec{x}|}\partial_{\vec{x}}^\alpha s^k \right\|_{L^\infty} \le C \sum_{i=0}^2 \left\| e^{|\vec{x}|} \nabla^i \partial_{\vec{x}}^\alpha s^k \right\|_{L^2},
\end{equation}
the term $\bigl\| e^{|\vec{x}|} [ \partial_{\vec{x}}^{\alpha}, \vec u^k\cdot\nabla]s^k \bigr\|_{L^2}$ can be estimated as follows:
\begin{equation}\nonumber
	\begin{aligned}
		\sum_{|\alpha|=1}^l \left\| e^{|\vec{x}|} \left[ \partial_{\vec{x}}^{\alpha}, \vec u^k \cdot \nabla \right]s^k \right\|_{L^2} &\le C \sum_{|\alpha|=1}^l \Bigg( \sum_{|\beta| \le \max\{|\alpha|-3,0\}} \left\| e^{|\vec{x}|} \nabla \partial_{\vec{x}}^\beta s^k \right\|_{L^\infty} \cdot \left\| \partial_{\vec{x}}^{\alpha-\beta} \vec u^k \right\|_{L^2}\\
		 &\quad \quad \quad \quad +\sum_{\max\{|\alpha|-2,1\} \le |\beta| \le |\alpha|-1} \left\| \partial_{\vec{x}}^{\alpha-\beta} \vec u^k \right\|_{L^\infty} \cdot \left\| e^{|\vec{x}|} \nabla \partial_{\vec{x}}^\beta s^k \right\|_{L^2} \Bigg)\\
		 & \le C\sum_{|\alpha|=1}^l \left( \sum_{i=1}^{\max\{|\alpha|,3\}} \left\| e^{|\vec{x}|} \nabla^i s^k \right\|_{L^2} \cdot \left\|\nabla \vec{u}^k \right\|_{H^{l-1}} + \sum_{i=\max\{|\alpha|-1,2\}}^{|\alpha|} \left\| e^{|\vec{x}|} \nabla^i s^k \right\|_{L^2} \cdot \left\|\nabla \vec{u}^k \right\|_{W^{1,\infty}} \right)\\
		 & \le C\left(\left\|\nabla \vec{u}^k \right\|_{H^{l-1}}+ \left\|\nabla \vec{u}^k \right\|_{W^{1,\infty}}\right) \sum_{i=1}^l \left\| e^{|\vec{x}|} \nabla^i s^k \right\|_{L^2}
	\end{aligned}
\end{equation}
for $l\ge 3$. 

Combining this result with \eqref{weightsk}, one obtains that
\begin{equation}\nonumber
	\begin{aligned}
		\frac{1}{2} \frac{d}{dt} \sum_{j=0}^l \left\| e^{|\vec{x}|} \nabla^j s^k \right\|_{L^2}^2 \le C \left( \left\| \vec{u}^k\right\|_{W^{2,\infty}} + \left\|\nabla \vec{u}^k \right\|_{H^{l-1}}\right) \sum_{j=0}^l \left\| e^{|\vec{x}|} \nabla^j s^k \right\|_{L^2}^2 \le C \left\|\nabla \vec{u}^k \right\|_{H^{\max\{l-1,3\}}} \sum_{j=0}^l \left\| e^{|\vec{x}|} \nabla^j s^k \right\|_{L^2}^2,
	\end{aligned}
\end{equation}
which together with Gronwall's inequality and \eqref{decaygen2} yields that
\begin{equation}\label{expsk}
	\sum_{j=0}^l \left\| e^{|\vec{x}|} \nabla^j s^k \right\|_{L^2}^2  \le \exp\left\lbrace C\int_0^t \left\|\nabla \vec{u}^k(\cdot,\tau) \right\|_{H^{\max\{l-1,3\}}} d\tau \right\rbrace 	\sum_{j=0}^l \left\| e^{|\vec{x}|} \nabla^j s_0^k \right\|_{L^2}^2 \le \exp\left\lbrace C\left\|\left(p_0^k, \vec{u}_0^k\right)\right\|_{L^1 \cap H^{\max\{l,4\}}} \right\rbrace \sum_{j=0}^l \left\| e^{|\vec{x}|} \nabla^j s_0^k \right\|_{L^2}^2.
\end{equation}
Furthermore, from \eqref{weightsob}, for $|\alpha| \in [0,5]$, it holds that
\begin{equation}\label{pointsk}
	\begin{aligned}
		\left| \partial_{\vec{x}}^\alpha s^k(\vec{x},t) \right| \le e^{-|\vec{x}|} \left\| e^{|\vec{x}|}\partial_{\vec{x}}^\alpha s^k \right\|_{L^\infty} \le C e^{-|\vec{x}|} \sum_{i=|\alpha|}^{|\alpha|+2} \left\| e^{|\vec{x}|} \nabla^i s^k \right\|_{L^2} \le Ce^{-|\vec{x}|+C\left\|\left(p_0^k, \vec{u}_0^k\right)\right\|_{L^1 \cap H^{l_\alpha}}} \sum_{j=0}^{\max\{|\alpha|+2,3\}} \left\| e^{|\vec{x}|} \nabla^j s_0^k \right\|_{L^2} \le C\epsilon_k e^{-|\vec{x}|},
	\end{aligned}
\end{equation}
where $l_\alpha:=\max\{|\alpha|+2,4\}$.\\

\textbf{Pointwise Estimates for $(p^k, \vec{u}^k)(\vec{x},t)$.} Under the following initial assumption
\begin{equation}\nonumber
	\bigl|\vec{V}^k\bigr|(\vec{x},0), \, \bigl|\nabla^i \vec{V}^k\bigr|(\vec{x},0) \le C_{in} \epsilon_k e^{-|\vec{x}|}, \ \ \text{for } i=1,\ldots,5,
\end{equation}
together with \eqref{fs1k}--\eqref{greenfi}, one obtains that for $0 \le |\alpha| \le 3$,
\begin{equation}\label{fiinitial}
	\begin{aligned}
		\left|\partial_{\vec{x}}^\alpha \left( \left( \mathbb{F}_i^k-\mathbb{F}_{S,i}^k\right)(\vec{x},t) \underset{\vec{x}}{\ast} \vec{V}^k(\vec{x},0)\right)\right| &\le C\epsilon_k \left( (1+t)^{-(3+|\alpha|)/2-d_i} e^{-\frac{|\vec{x}|^2}{Dt}} + e^{-(|\vec{x}|+t)/D}\right) \underset{\vec{x}}{\ast} e^{-|\vec{x}|}\\
		& \le C\epsilon_k  \left((1+t)^{-(3+|\alpha|)/2-d_i} e^{-\frac{|\vec{x}|^2}{4Dt}}+e^{-(|\vec{x}|+t)/(8D)}\right),
	\end{aligned}
\end{equation}
\begin{equation}\label{fsikini}
	\begin{aligned}
		\left|\left( \mathbb{F}_{S,i}^k (\vec{x},t) \underset{\vec{x}}{\ast} \partial_{\vec{x}}^\alpha \vec{V}^k(\vec{x},0)\right)\right| & \le C\epsilon_k\left[e^{-at-|\vec{x}|}+ e^{-at/4} \int_{|\vec{y}|=1} e^{-|\vec{x}+c_k t\vec{y}|} \, dS(\vec{y})+ \left( \sum_{i=0}^1 |\vec{x}|^{-i}\right) e^{-at/8}\left( e^{-|\vec{x}|/4}+e^{-a|\vec{x}|/(16c_k)}\right) \underset{\vec{x}}{\ast} e^{-|\vec{x}|}\right]\\
		&\le C\epsilon_k e^{-at/8} \left(e^{-|\vec{x}|/16}+e^{-a|\vec{x}|/(64c_k)} \right),
	\end{aligned}
\end{equation} 
and thus it holds that
\begin{equation}\label{initialcon}
	\begin{aligned}
		\left|\partial_{\vec{x}}^\alpha \left( \mathbb{F}_i^k(\vec{x},t) \underset{\vec{x}}{\ast} \vec{V}^k(\vec{x},0)\right)\right| &\le \left|\partial_{\vec{x}}^\alpha \left( \left( \mathbb{F}_i^k-\mathbb{F}_{S,i}^k\right) \! (\vec{x},t) \underset{\vec{x}}{\ast} \vec{V}^k(\vec{x},0)\right)\right| + \left|\left( \mathbb{F}_{S,i}^k (\vec{x},t) \underset{\vec{x}}{\ast} \partial_{\vec{x}}^\alpha \vec{V}^k(\vec{x},0)\right)\right|\\
		&\le C\epsilon_k \left((1+t)^{-(3+|\alpha|)/2-d_i} e^{-\frac{|\vec{x}|^2}{4Dt}} + e^{-(|\vec{x}|+t)/K} \right), 
	\end{aligned}
\end{equation}
where $K:=8\max\{ 2,D,1/a,8c_k/a\}$.

To estimate the nonlinear wave propagations, one needs to establish the nonlinear estimates which require 
the followings Sobolev inequalities:
\begin{lemma}[Lemma~2.1, \cite{hl1}] \label{derivaineq} For $\omega_1\in H^l(\mathbb{R}^3), \nabla\omega_1\in L^\infty(\mathbb{R}^3)$, $\omega_2\in H^{l-1}(\mathbb{R}^3)\cap L^\infty(\mathbb{R}^3)$ and a multi-index $\alpha$,
	\begin{equation} \nonumber
		\sum_{1 \le  |\alpha|\le l} \left\|\bigl[\partial_{\vec{x}}^\alpha, \omega_1\bigr]\omega_2 \right\|_{L^2} \le C_{l} \left( \left\| \nabla \omega_1\right\|_{L^\infty} \cdot\left\|\nabla^{l-1} \omega_2 \right\|_{L^2} + \left\|\omega_2 \right\|_{L^\infty}\cdot \left\|\nabla^{l} \omega_1 \right\|_{L^2} \right),
	\end{equation}
	where the commutator $[\cdot,\cdot]$ is defined by $$\bigl[\partial_{\vec{x}}^\alpha, \omega_1\bigr]\omega_2:= \partial_{\vec{x}}^\alpha (\omega_1 \omega_2) - \omega_1 \partial_{\vec{x}}^\alpha \omega_2.$$ 
	
	Particularly, for $\omega_1, \omega_2\in H^{l}(\mathbb{R}^3)\cap L^\infty(\mathbb{R}^3)$,
	\begin{equation} \nonumber
		\sum_{|\alpha| \le l} \left\|\partial_{\vec{x}}^\alpha (\omega_1 \omega_2)\right\|_{L^2} \le C_{l} \left(\left\|\omega_1\right\|_{L^\infty} \cdot\left\|\nabla^{l} \omega_2 \right\|_{L^2} + \left\|\omega_2 \right\|_{L^\infty}\cdot \left\|\nabla^{l} \omega_1 \right\|_{L^2} \right),
	\end{equation}
	where 
	\begin{equation}\label{wemb}
		\left\|\omega_1\right\|_{L^\infty}\le C\left\|\nabla\omega_1\right\|^{1/2}_{L^2}\cdot \left\|\nabla^2\omega_1\right\|_{L^2}^{1/2}\le C\left\|\nabla\omega_1\right\|_{H^1},\ \ \forall \omega_1\in H^2(\mathbb{R}^3).
	\end{equation}
\end{lemma}

Now, one could estimate the nonlinear terms $\bigl(R_p^k, R_{\vec{u}}^k\bigr)$ defined in \eqref{pk} by $\mathcal{N}^k(t)$ defined in \eqref{anasn} as follows:
\begin{lemma}\label{lerpu}
	The nonlinear terms $\bigl(R_p^k, R_{\vec{u}}^k\bigr)$ defined in \eqref{pk} satisfy the following estimates:
	\begin{equation} \label{rpubound1}
		 	\left|\nabla^i \left( R_p^k,R_{\vec{u}}^k\right)(\vec{x},t)\right| \le 
			\begin{cases}C\left( \mathcal{N}^k(t)\right)^2 \left( (1+t)^{1/2} \left(w_3^k(\vec{x},t)\right)^2 
			+ (1+t)^{-2}e^{-|\vec{x}|}\right), \quad \text{for } i=0,1,2,\\
			\\
			 C \mathcal{N}^k(t) \left( \epsilon_k + \mathcal{N}^k(t)\right) \left( \left(w_3^k(\vec{x},t)\right)^2+ (1+t)^{-3/2}  w_{i-3}^k(\vec{x},t) +(1+t)^{-2} e^{-|\vec{x}|}\right),\quad \text{for } i=3,4,5,
			 \end{cases}
	\end{equation}
	where $\epsilon_k$ and $\mathcal{N}^k(t)$ are given in \eqref{initial} and \eqref{anasn} respectively, and $w^k_{j}(\vec{x},t)$ $(j=0,1,2,3)$ are defined in \eqref{w0123}.
\end{lemma}
\begin{proof}
From \eqref{pk} and \eqref{anasn}, one has that
\begin{equation}\label{rpu0}
	\begin{aligned}
		\left|R_p^k(\vec{x},t)\right| &\le C \left( \left| p^k\right|  \cdot \left| div \vec{u}^k\right| + \left| \vec{u}^k\right| \cdot \left|\nabla p^k \right| \right) \! (\vec{x},t) \le C\left( \mathcal{N}^k(t)\right)^2(1+t)^{1/2} w_{0}^k(\vec{x},t) \cdot w_1^k(\vec{x},t),\\
		\left|R_{\vec{u}}^k(\vec{x},t)\right| &\le C \left( \left| \vec{u}^k\right|  \cdot \left|\nabla \vec{u}^k\right| + \left| \rho^k-\bar{\rho}^k \right| \cdot \left|\nabla p^k \right| \right) \! (\vec{x},t) \le C\left( \mathcal{N}^k(t)\right)^2 \left( (1+t)^{1/2} w_{0}^k(\vec{x},t) +e^{-|\vec{x}|}\right) w_1^k(\vec{x},t),
	\end{aligned}
\end{equation}
\begin{equation}\label{rpu1}
	\begin{aligned}
		\left|\nabla R_p^k(\vec{x},t)\right| &\le C \left( \left| \nabla p^k\right|  \cdot \left| \nabla \vec{u}^k\right| + \left| p^k\right|  \cdot \left| \nabla div \vec{u}^k\right|+ \left| \vec{u}^k\right| \cdot \left|\nabla^2 p^k \right| \right) \!(\vec{x},t)\\
		&\le C\left( \mathcal{N}^k(t)\right)^2 \Bigl( w_1^k(\vec{x},t) \cdot w_1^k(\vec{x},t) + (1+t)^{1/2} w_{0}^k(\vec{x},t) \cdot w_2^k(\vec{x},t)\Bigr),\\
		\left|\nabla R_{\vec{u}}^k(\vec{x},t)\right| &\le C \left( \left|\nabla \vec{u}^k\right|^2+\left| \vec{u}^k\right|  \cdot \left|\nabla^2 \vec{u}^k\right| + \left| \rho^k-\bar{\rho}^k \right| \cdot \left|\nabla^2 p^k \right|+ \left| \nabla \frac{1}{\rho^k} \right| \cdot \left|\nabla p^k \right| \right) \!(\vec{x},t)\\ 
		&\le C\left( \mathcal{N}^k(t)\right)^2 \Bigl[ w_1^k(\vec{x},t) \cdot w_1^k(\vec{x},t) + \left( (1+t)^{1/2} w_{0}^k(\vec{x},t) +e^{-|\vec{x}|}\right) w_2^k(\vec{x},t)\Bigr],
	\end{aligned}
\end{equation}
\begin{equation}\label{rpu2}
	\begin{aligned}
		\left|\nabla^2 R_p^k(\vec{x},t) \right|&\le C \left( \left| \nabla p^k\right|  \cdot \left| \nabla^2 \vec{u}^k\right|+\left| \nabla^2 p^k\right|  \cdot \left| \nabla \vec{u}^k\right| + \left| p^k\right|  \cdot \left| \nabla^2 div \vec{u}^k\right|+ \left| \vec{u}^k\right| \cdot \left|\nabla^3 p^k \right| \right) \!(\vec{x},t)\\
		&\le C\left( \mathcal{N}^k(t)\right)^2 \Bigl( w_1^k(\vec{x},t) \cdot w_2^k(\vec{x},t) + (1+t)^{1/2} w_{0}^k(\vec{x},t) \cdot w_3^k(\vec{x},t)\Bigr),\\
		\left|\nabla^2 R_{\vec{u}}^k(\vec{x},t) \right|&\le C \left( \left|\nabla \vec{u}^k\right| \cdot \left|\nabla^2 \vec{u}^k\right| +\left| \vec{u}^k\right|  \cdot \left|\nabla^3 \vec{u}^k\right| + \left| \rho^k-\bar{\rho}^k \right| \cdot \left|\nabla^3 p^k \right|+  \sum_{i=1}^{2} \left| \nabla^i \frac{1}{\rho^k} \right| \cdot \left|\nabla^{3-i} p^k \right| \right) \!(\vec{x},t)\\ 
		&\le C\left( \mathcal{N}^k(t)\right)^2 \left[ w_1^k(\vec{x},t) \cdot w_2^k(\vec{x},t) + \left( (1+t)^{1/2} w_{0}^k(\vec{x},t) +e^{-|\vec{x}|}\right) w_3^k(\vec{x},t)\right].
	\end{aligned}
\end{equation}
In the above estimates, one applies the following inequalities
\begin{equation}\label{rhok}
	\begin{aligned}
		\left| \rho^k-\bar{\rho}^k \right| \le C\left(\left|p^k\right| +\left|s^k \right| \right), \quad \left| \nabla^i \rho^k \right|, \left| \nabla^i \frac{1}{\rho^k} \right| \le C\sum_{k_1+2k_2+ \cdots+ik_i=i} \prod_{j=1}^i \left( \left|\nabla^j p^k\right| +\left|\nabla^j s^k \right|\right)^{k_j},
	\end{aligned}
\end{equation}
for $i \ge 1 $ and $k_1, k_2,\dots,k_i \in \mathbb{N}$, together with the pointwise estimates \eqref{pointsk} for the entropy function $s^k$ and its derivatives.

Moreover, combining the decay property in \eqref{decaygen3}  and  the Sobolev embedding \eqref{wemb}, one can obtain higher‑order derivative estimates for the nonlinear terms $\bigl(R_p^k, R_{\vec{u}}^k\bigr)$ as follows:
\begin{equation}\label{rpu3}
	\begin{aligned}
		\left|\nabla^3 R_p^k(\vec{x},t)\right| &\le C \sum_{i+j=4, \, i,j \in \mathbb{N}} \left( \left| \nabla^i p^k\right|  \cdot \left| \nabla^j \vec{u}^k\right| \right) \!(\vec{x},t)\\
		&\le C \mathcal{N}^k(t) \left( \epsilon_k + \mathcal{N}^k(t)\right)\left( w_1^k(\vec{x},t) \cdot w_3^k(\vec{x},t) + w_2^k(\vec{x},t) \cdot w_2^k(\vec{x},t) + (1+t)^{-3/2} w_{0}^k(\vec{x},t)\right),\\
		\left|\nabla^3 R_{\vec{u}}^k(\vec{x},t)\right| &\le C \sum_{i+j=4, \, i,j \in \mathbb{N}} \left( \left| \nabla^i \vec{u}^k\right|  \cdot \left| \nabla^j \vec{u}^k\right| \right) \!(\vec{x},t) + \sum_{i+j=3, \, i,j \in \mathbb{N}} \left( \left| \nabla^i \left(\frac{1}{\bar{\rho}^k}-\frac{1}{\rho^k}\right)\right|  \cdot \left| \nabla^{j+1} p^k \right| \right) \!(\vec{x},t)\\
		&\le C \mathcal{N}^k(t) \left( \epsilon_k + \mathcal{N}^k(t)\right)\left( w_1^k(\vec{x},t) \cdot w_3^k(\vec{x},t) + w_2^k(\vec{x},t) \cdot w_2^k(\vec{x},t) + (1+t)^{-3/2}  w_{0}^k(\vec{x},t)+(1+t)^{-2} e^{-|\vec{x}|} \right).
	\end{aligned}
\end{equation}

From \eqref{rpu0}--\eqref{rpu3} and the following inequality
\begin{equation}\nonumber
	w_{0}^k(\vec{x},t)<w_1^k(\vec{x},t)<w_2^k(\vec{x},t)<w_3^k(\vec{x},t)\le C(1+t)^{-2},
\end{equation}
one can directly deduce \eqref{rpubound1} for $i=0,1,2,3$. The cases for $i=4,5$ could be verified similarly. 
\end{proof}

Combining 
\begin{equation}\nonumber
	\begin{aligned}
		 \left| \partial_{\vec{x}}^\alpha \vec{S}_{R,i}^k(\vec{x},t)\right|  &\le \left| \partial_{\vec{x}}^\alpha \left( \mathbb{F}_i^k-\mathbb{F}_i^{k*} \right) \! (\vec{x},t) \underset{(\vec{x},t)}{\ast}\begin{pmatrix}
		 	R_p^k \vspace{1.5mm}\\
		 	R_{\vec{u}}^k
		 \end{pmatrix}(\vec{x},t)\right| +	\left|\left( \mathbb{F}_i^{k*}-\mathbb{F}_{S,i}^k\right) \!(\vec{x},t) \underset{(\vec{x},t)}{\ast} \partial_{\vec{x}}^\alpha
	 \begin{pmatrix}
		 R_p^k \vspace{1.5mm}\\
		 R_{\vec{u}}^k
	 \end{pmatrix}(\vec{x},t) \right|\\
 		&\quad + \left|\mathbb{F}_{S,i}^k(\vec{x},t) \underset{(\vec{x},t)}{\ast} \partial_{\vec{x}}^\alpha
 		\begin{pmatrix}
 			R_p^k \vspace{1.5mm}\\
 			R_{\vec{u}}^k
 		\end{pmatrix}(\vec{x},t) \right|, \ \  i=1,2,
	\end{aligned}
\end{equation}
with \eqref{fs1k}--\eqref{greenfi}, \eqref{rpubound1}, \eqref{kercon1}--\eqref{kercon2} for $m=2d_i$ (with $d_i$ defined by \eqref{di}), \eqref{expw3}--\eqref{expst} with $L_1=\max\{D,4,16c_k/a\}$ and $L_2=\max\{D,8/a\}$, and \eqref{skckw3}--\eqref{skckst} with $c=c_k=\sqrt{\bar{P}^k(R+C_v)/\left(C_v \bar{\rho}^k\right)}$, one obtains the following estimate for $\vec{S}_R^k$:
\begin{equation}\nonumber
	\begin{aligned}
		\left| \vec{S}_{R,i}^k(\vec{x},t)\right|  & \le C\left( \mathcal{N}^k(t)\right)^2 \Bigg[(1+t)^{-3/2-d_i} e^{-\frac{|\vec{x}|^2}{C_{0,k}t}}+e^{-|\vec{x}|/D_{0,k}} \left( e^{-2t/(7L_2)} + e^{-at/28}+e^{-t/D_{0,k}} \right) \\
		&\quad \quad \quad \quad \quad \quad +(1+t)^{-2} \left( e^{-|\vec{x}|/(14L_1)} +e^{-a|\vec{x}|/(56c_k)}+ e^{-|\vec{x}|/2} \right) + (1+t)^{1/2} e^{-at} w_3^k(\vec{x},t) \cdot w_3^k(\vec{x},t) \Bigg]\\
		&\le C\left( \mathcal{N}^k(t)\right)^2 \left((1+t)^{-3/2-d_i} e^{-\frac{|\vec{x}|^2}{C_{0,k}t}}+e^{-(|\vec{x}|+t)/D_{0,k}}\right),
	\end{aligned}
\end{equation}
which together with \eqref{v} and \eqref{initialcon} yields that
\begin{equation}\label{pukpoint}
	\begin{aligned}
		\left| \vec{V}_i^k(\vec{x},t)\right| &=\left| \mathbb{F}_i^k(\vec{x},t) \underset{\vec{x}}{\ast} \vec{V}^k(\vec{x},0)+ \vec{S}_{R,i}^k(\vec{x},t)\right|\\
		&\le C\left( \epsilon_k+\left( \mathcal{N}^k(t)\right)^2\right)\left((1+t)^{-3/2-d_i} e^{-\frac{|\vec{x}|^2}{C_{0,k}t}}+e^{-(|\vec{x}|+t)/D_{0,k}}\right).
	\end{aligned}
\end{equation}

Similarly, combining \eqref{rpubound1}, \eqref{expwi} and \eqref{skckwi}, for $1\le |\alpha| \le 3$, $m=|\alpha|+2d_i$, and $\mu=(3\sqrt{6}/7)^{1/3}$, one has that
\begin{equation}\nonumber
	\begin{aligned}
		\left| \partial_{\vec{x}}^\alpha \vec{S}_{R,i}^k(\vec{x},t)\right| &\le C \mathcal{N}^k(t) \left( \epsilon_k + \mathcal{N}^k(t)\right) \Bigg[(1+t)^{-2} e^{-\frac{|\vec{x}|^2}{C_{0,k}t}} +e^{-|\vec{x}|/D_{0,k}} \left( e^{-2t/(7L_2)} +e^{-at/28} + e^{-t/D_{0,k}} \right)\\
		&\quad \quad +(1+t)^{-2} \left( e^{-|\vec{x}|/(14L_1)}+ e^{-a|\vec{x}|/(56c_k)}+ e^{-|\vec{x}|/2} \right)\\
		&\quad \quad+\sum_{j=0}^{|\alpha|-1} (1+t)^{-7/2}\left( e^{-(1-\mu^{-1})|\vec{x}|/(2L_1)}+e^{-a(1-\mu^{-1})|\vec{x}|/(8c_k)}+e^{-\frac{|\vec{x}|^2}{C_{j+1,k}t}}\right)\\
		&\quad \quad+\sum_{j=0}^{|\alpha|-1} e^{-|\vec{x}|/D_{j+1,k}} \left(e^{-(1-\mu^{-1}) t/L_2}+e^{-a(1-\mu^{-1}) t/8}+e^{-t/D_{j+1,k}}\right)\\
		&\quad \quad+ e^{-at}\left((1+t)^{1/2} w_3^k(\vec{x},t) \cdot w_3^k(\vec{x},t) + (1+t)^{-3/2} w_{0}^k(\vec{x},t)\right)\Bigg]\\
		&\le C \mathcal{N}^k(t) \left( \epsilon_k + \mathcal{N}^k(t)\right)\left( (1+t)^{-2} e^{-\frac{|\vec{x}|^2}{C_{|\alpha|,k}t}}+e^{-(|\vec{x}|+t)/D_{|\alpha|,k}}\right),
	\end{aligned}
\end{equation}
which together with \eqref{initialcon} results in that
\begin{equation}\label{pukalpha}
	\begin{aligned}
		\left| \partial_{\vec{x}}^\alpha \vec{V}_i^k(\vec{x},t)\right| &=\left|\partial_{\vec{x}}^\alpha \left( \mathbb{F}_i^k(\vec{x},t) \underset{\vec{x}}{\ast} \vec{V}^k(\vec{x},0)\right)+ \partial_{\vec{x}}^\alpha \vec{S}_{R,i}^k(\vec{x},t)\right|\\
		&\le C\left[\epsilon_k+\mathcal{N}^k(t) \left( \epsilon_k + \mathcal{N}^k(t)\right)\right]\left( (1+t)^{-2} e^{-\frac{|\vec{x}|^2}{C_{|\alpha|,k}t}}+e^{-(|\vec{x}|+t)/D_{|\alpha|,k}}\right),\quad \text{for } |\alpha|=1,2,3.
	\end{aligned}
\end{equation}

Finally form the definition of $\mathcal{N}^k(t)$ in \eqref{anasn} and the estimates \eqref{pukpoint}--\eqref{pukalpha}, one obtains that
\begin{equation}\nonumber
	\mathcal{N}^k(t) \le C\left[ \epsilon_k+\mathcal{N}^k(t) \left( \epsilon_k + \mathcal{N}^k(t)\right)\right].
\end{equation}
By virtue of the smallness of $\epsilon_k$, there exists a positive constant $C_*$ such that $\mathcal{N}^k(t) \le C_* \epsilon_k$, and we finish the proof of Theorem~\ref{pointwise}.

%%%%%%%%%%%%%%%%%%%%%%%%%%%%%%%%%%%%%%%%%%%%%%%%%%%%%%%%%%%%%%%%%%%%%%%%%%%%%%%%%5
%
%
%                        Section 3 
%
%
%%%%%%%%%%%%%%%%%%%%%%%%%%%%%%%%%%%%%%%%%%%%%%%%%%%%%%%%%%%%%%%%%%%%%%%%%%%%%%%%%

\setcounter{equation}{0}

\section{Comparison Principle: Proof of Theorem~\ref{main}} \label{sec:3}
Denote the difference of the solutions $\bigl(p^k,\vec{u}^k,s^k\bigr)$ ($k=0,1$) for two individual nonlinear problems \eqref{pk} to be $(p,\vec{u},s)$, that is,
\begin{equation}\nonumber
	\left(p,\vec{u},s\right) \!(\vec x,t)\equiv \left(p^0-p^1,\vec{u}^0-\vec{u}^1,s^0-s^1\right) \!(\vec{x},t),
\end{equation}
Then $(p,\vec{u},s)$ satisfies
\begin{equation}\label{pc}
	\begin{cases}
		\partial_t p+\frac{\bar{P}^0\left(R+C_v\right)}{C_v} div\vec{u}+\left(\bar{P}^0-\bar{P}^1\right)\frac{\left(R+C_v\right)}{C_v}div\vec{u}^1\\
		\ \ \ \ \ \ \ \ =-\frac{R+C_v}{C_v} p div\vec{u}^1-\frac{R+C_v}{C_v} p^0 div\vec{u}
		-\vec u\cdot \nabla p^1-\vec u^0\cdot \nabla p:=H_p,\\
		\partial_t \vec{u}+\frac{1}{\bar{\rho}^0}\nabla p+a\vec{u}\\
		\ \ \ \ \ \ \ \ =-\vec{u}\cdot\nabla \vec{u}^1-\vec{u}^0\cdot\nabla \vec{u}+\left(\frac{1}{\bar{\rho}^0}-\frac{1}{\rho^0}\right)\nabla p
		-\left(\frac{1}{\rho^0}-\frac{1}{\rho^1}\right)\nabla p^1:=H_{\vec{u}},\\
		\partial_t s=-\vec{u}\cdot\nabla {s}^1-\vec{u}^0 \cdot\nabla s:=H_s,\\
		\left(p,\vec{u},s\right) \!(\vec{x},0)=\left(p_0,\vec{u}_0,s_0\right) \!(\vec{x})=\left(p^0_0-p^1_0,\vec u^0_0-\vec u^1_0,s^0_0-s^1_0\right) \!(\vec{x}).
	\end{cases}
\end{equation}
In this section, we derive the pointwise space–time structures of $p$, $\vec{u}$, and $s$ using estimates of the Green’s function given in Theorem~\ref{thm2.1} and Duhamel’s principle, and also pointwise estimates of the individual problems obtained in Theorem \ref{pointwise}.

\subsection{Energy Estimates for $p$, $\vec{u}$ and $s$}
One first establishes the dependence of higher-order derivatives of $(p,\vec{u},s)$ on the initial difference and the background perturbations:
\begin{thm}\label{enlec}
	 Let $(p,\vec{u},s)(\vec{x},t)$ be the solution of \eqref{pc}, with $\|(p,\vec{u},s)(\cdot,0)\|_{H^l} \ll 1$ for $l \ge 3$ and $\bigl|\bar{P}^0-\bar{P}^1\bigr|+\bigl|\bar{S}^0-\bar{S}^1\bigr| = \mathcal{O}(1)\delta$. Then $(p,\vec{u})(\vec{x},t)$ satisfies that
	\begin{equation} \label{enpuc}
		\left\|(p,\vec{u})(\cdot, t)\right\|_{H^l}^2 + \int_0^t \left( \left\| \nabla p (\cdot, \tau)\right\|_{H^{l-1}}^2 + \left\|\vec{u} (\cdot, \tau)\right\|_{H^l}^2 \right) d\tau \le C \left(\delta^2+\left\|(p_0,\vec{u}_0,s_0)\right\|_{H^l}^2\right).
	\end{equation}
	Furthermore, if the assumptions of Theorem~\ref{pointwise} also hold, one has that
	\begin{equation}\label{ensc}
		\bigl\|s(\cdot, t)\bigr\|_{H^l} \le C \left\|s_0\right\|_{H^l} + C\left\|s_0^1\right\|_{H^{l+1}} \left(\left\|(p_0,\vec{u}_0)\right\|_{L^1 \cap H^l} + \epsilon_{01} \delta\right),
	\end{equation}
	and
	\begin{equation}\label{pucdecay1}
		\left\|\vec{u}(\cdot,t)\right\|_{H^l} +\left\|\nabla p(\cdot,t)\right\|_{H^{l-1}} \le C(1+t)^{-\frac{5}{4}}\log(2+t) \left(\left\|(p_0,\vec{u}_0)\right\|_{L^1 \cap H^l} + \epsilon_{01} \left(\delta +\left\|s_0\right\|_{H^l}\right)\right),
	\end{equation}
	where $C > 0$ is a universal constant independent of $t$ and $\delta$, and $\epsilon_{01}:=\epsilon_{0}+\epsilon_{1}$.
\end{thm}
\begin{proof}
	For a multi-index $\alpha $, taking the corresponding derivatives $\partial_{\vec{x}}^\alpha$ on both sides of \eqref{pc}, one obtains that
	\begin{align}
		&\left( \partial_{\vec{x}}^\alpha p\right)_t+\frac{\bar{P}^0\left(R+C_v\right)}{C_v}  div \partial_{\vec{x}}^\alpha \vec{u}+\left(\bar{P}^0-\bar{P}^1\right)\frac{\left(R+C_v\right)}{C_v} div \partial_{\vec{x}}^\alpha \vec{u}^1 \label{palpha}\\
		&\quad = \partial_{\vec{x}}^\alpha \left( -\frac{R+C_v}{C_v} p div\vec{u}^1 -\vec{u} \cdot \nabla p^1-\vec{u}^0\cdot \nabla p\right)-\frac{R+C_v}{C_v} \left[\partial_{\vec{x}}^\alpha, p^0\right]div\vec{u} -\frac{R+C_v}{C_v} p^0 div \partial_{\vec{x}}^\alpha \vec{u} := \partial_{\vec{x}}^\alpha H_p,\nonumber\\
		& \left( \partial_{\vec{x}}^\alpha \vec{u}\right)_t+\frac{1}{\bar{\rho}^0}\nabla  \partial_{\vec{x}}^\alpha p+a  \partial_{\vec{x}}^\alpha \vec{u}= \partial_{\vec{x}}^\alpha\left( -\vec{u}\cdot\nabla \vec{u}^1-\vec{u}^0\cdot\nabla \vec{u}-\left(\frac{1}{\rho^0}-\frac{1}{\rho^1}\right)\nabla p^1\right) -\left[\partial_{\vec{x}}^\alpha, \frac{1}{\rho^0}\right]\nabla p +\left(\frac{1}{\bar{\rho}^0}-\frac{1}{\rho^0}\right) \nabla \partial_{\vec{x}}^\alpha p := \partial_{\vec{x}}^\alpha H_{\vec{u}}, \label{ualpha}
	\end{align}
and
\begin{equation}\label{salpha}
	\left( \partial_{\vec{x}}^\alpha s\right)_t =-\partial_{\vec{x}}^\alpha \left( \vec{u}\cdot\nabla s^1\right) -\left[\partial_{\vec{x}}^\alpha, \vec{u}^0 \cdot \nabla\right] s -\vec{u}^0 \cdot \nabla \partial_{\vec{x}}^\alpha s:=\partial_{\vec{x}}^\alpha H_s.
\end{equation}

Integrating $(\rho^0)^{-1} \partial_{\vec{x}}^\alpha p$ \eqref{palpha} $+P^0 C_v^{-1} (R+C_v) \partial_{\vec{x}}^\alpha \vec{u}\cdot$ \eqref{ualpha} over $\mathbb{R}^3$ and using the following identities: 
\begin{equation}\label{rho}
		\partial_t \rho^k=-div\left( \rho^k \vec{u}^k\right),
	\end{equation}
	\begin{equation}\label{pp}
		\partial_t p^k = -\frac{R+C_v}{C_v} \left( p^k + \bar{P}^k\right) div \vec{u}^k-\vec{u}^k \cdot \nabla p^k,
	\end{equation}
one has that
\begin{equation}\label{puenesc}
	\begin{aligned}
		&\frac{1}{2} \frac{d}{dt} \int_{\mathbb{R}^3} \left( \frac{1}{\rho^0} \left(\partial_{\vec{x}}^\alpha p \right)^2 + \frac{P^0 \left(R+C_v\right)}{C_v}  \left|\partial_{\vec{x}}^\alpha \vec{u} \right|^2 \right) d\vec{x} 
		+ \frac{a\left( R+C_v\right)}{C_v} \int_{\mathbb{R}^3} P^0 \left|\partial_{\vec{x}}^\alpha \vec{u} \right|^2 d\vec{x}\\
		&=	\left(\bar{P}^1-\bar{P}^0\right)\frac{\left(R+C_v\right)}{C_v} \int_{\mathbb{R}^3} \frac{1}{\rho^0} \partial_{\vec{x}}^\alpha p \, div \partial_{\vec{x}}^\alpha \vec{u}^1 d\vec{x} + \frac{1}{2} \int_{\mathbb{R}^3} \left( \frac{div\left( \rho^0 \vec{u}^0\right)}{\left(\rho^0\right)^2} + div\left(\frac{\vec{u}^0}{\rho^0} \right)\right)  \left(\partial_{\vec{x}}^\alpha p \right)^2 d\vec{x}\\
		&\quad + \frac{R+C_v}{2C_v} \int_{\mathbb{R}^3} \left( div \left(P^0 \vec{u}^0 \right)-\frac{R+C_v}{C_v} P^0 div \vec{u}^0 - \vec{u}^0 \cdot \nabla p^0\right) \left|\partial_{\vec{x}}^\alpha \vec{u} \right|^2 d\vec{x} - \frac{R+C_v}{C_v} \int_{\mathbb{R}^3} \frac{P^0}{\rho^0} div\left( \partial_{\vec{x}}^\alpha p \, \partial_{\vec{x}}^\alpha \vec{u} \right) d\vec{x}\\
		&\quad -\int_{\mathbb{R}^3} \frac{1}{\rho^0} \partial_{\vec{x}}^\alpha p \left[ \partial_{\vec{x}}^\alpha \left(\frac{R+C_v}{C_v} p div \vec{u}^1 + \vec{u} \cdot \nabla p^1 \right) + \frac{R+C_v}{C_v}\left[\partial_{\vec{x}}^\alpha, p^0\right]div\vec{u} + \left[\partial_{\vec{x}}^\alpha, \vec{u}^0\right] \nabla p\right]d\vec{x}\\
		& \quad -\frac{R+C_v}{C_v} \int_{\mathbb{R}^3} P^0 \partial_{\vec{x}}^\alpha \vec{u} \cdot \left[  \partial_{\vec{x}}^\alpha \left(\vec{u} \cdot \nabla \vec{u}^1 + \left(\frac{1}{\rho^0}-\frac{1}{\rho^1} \right)\nabla p^1 \right) + \left[\partial_{\vec{x}}^\alpha, \vec{u}^0 \cdot \nabla\right]\vec{u} + \left[\partial_{\vec{x}}^\alpha, \frac{1}{\rho^0}\right]\nabla p \right] d\vec{x}.
	\end{aligned}
\end{equation}
Summing \eqref{puenesc} over multi‑indices $\alpha$ with $|\alpha|\in[1,l]$ for $l\ge 3$ together with Lemma~\ref{derivaineq}, \eqref{rhok}, and the following inequalities for $|\alpha|\ge 1$:
\begin{equation}\label{rhocl2}
	\begin{aligned}
		&\left\| \nabla^{|\alpha|} \frac{1}{\rho^k}\right\|_{L^2} \le C \left\| \nabla \left(p^k, s^k\right)\right\|_{H^{|\alpha|-1}}, \quad  \left| \frac{1}{\rho^0}-\frac{1}{\rho^1}\right| \le C\left(\left|\bar{P}^0-\bar{P}^1\right|+\left|\bar{S}^0-\bar{S}^1\right|+|p|+|s|\right),\\
		&\left\|\nabla^{|\alpha|} \left( \frac{1}{\rho^0}-\frac{1}{\rho^1}\right)\right\|_{L^2} \le C\left[\left( \left|\bar{P}^0-\bar{P}^1\right|+\left|\bar{S}^0-\bar{S}^1\right|+\left\|p\right\|_{L^\infty}+\left\|s\right\|_{L^\infty}\right) \left\|\nabla \left(p^1,s^1\right)\right\|_{H^{|\alpha|-1}}+\left\|\nabla(p,s)\right\|_{H^{|\alpha|-1}}\right],
	\end{aligned}
\end{equation}
one deduces that
\begin{equation}\label{puc1l}
	\begin{aligned}
		&\frac{1}{2} \sum_{|\alpha|=1}^l \frac{d}{dt} \int_{\mathbb{R}^3} \left( \frac{1}{\rho^0} \left(\partial_{\vec{x}}^\alpha p \right)^2 + \frac{P^0 \left(R+C_v\right)}{C_v}  \left|\partial_{\vec{x}}^\alpha \vec{u} \right|^2 \right) d\vec{x} 
		+ \frac{a\left( R+C_v\right)}{C_v} \sum_{|\alpha|=1}^l \int_{\mathbb{R}^3} P^0 \left|\partial_{\vec{x}}^\alpha \vec{u} \right|^2 d\vec{x}\\
		&\le C\left|\bar{P}^0-\bar{P}^1\right| \cdot \left\|\nabla p\right\|_{H^{l-1}} \cdot \left\|\nabla div \vec{u}^1 \right\|_{H^{l-1}} + \frac{C}{\tilde{\varepsilon}_0} \left\|\nabla^2 p^1\right\|_{H^{l-1}}^2 \left(\left|\bar{P}^0-\bar{P}^1\right|^2+\left|\bar{S}^0-\bar{S}^1\right|^2 + \left\| \nabla s\right\|_{H^{l-1}}^2 \right)\\
		& \quad + C\left( \left\| \nabla p\right\|_{H^{l-1}}^2 + \left\| \nabla \vec{u}\right\|_{H^{l-1}}^2 \right) \left(\left\| \nabla \left(p^0, \vec{u}^0, s^0\right)\right\|_{H^{l-1}} +  \left\| \nabla^2 \left(p^1, \vec{u}^1\right)\right\|_{H^{l-1}} + \left\| \nabla \rho^0 \right\|_{L^\infty} \right) + C\tilde{\varepsilon}_0 \left\| \nabla \vec{u}\right\|_{H^{l-1}}^2,
	\end{aligned}
\end{equation}
where $\tilde{\varepsilon}_0\in(0,1)$ is taken sufficiently small. For $|\alpha|=0$, it holds that
\begin{equation}\nonumber
	\begin{aligned}
		&\frac{1}{2} \frac{d}{dt} \int_{\mathbb{R}^3} \left( \frac{1}{\rho^0} p^2 + \frac{P^0 \left(R+C_v\right)}{C_v}  \left|\vec{u} \right|^2 \right) d\vec{x} 
		+ \frac{a\left( R+C_v\right)}{C_v}  \int_{\mathbb{R}^3} P^0 \left|\vec{u} \right|^2 d\vec{x}\\
		&\le C\left|\bar{P}^0-\bar{P}^1\right| \cdot \left\|\nabla p\right\|_{H^1} \cdot \left\|\vec{u}^1 \right\|_{L^2} + \frac{C}{\tilde{\varepsilon}_0} \left\|\nabla p^1\right\|_{L^2}^2 \left(\left|\bar{P}^0-\bar{P}^1\right|^2+\left|\bar{S}^0-\bar{S}^1\right|^2 + \left\| \nabla s\right\|_{H^1}^2 \right)\\
		& \quad + C\left\| \nabla p\right\|_{H^1}^2 \left(\left\| \left(\vec{u}^0, \vec{u}^1\right)\right\|_{L^2} +  \left\| \nabla \left(p^0,p^1,\rho^0\right)\right\|_{L^2}\right) + C\left\|\vec{u}\right\|_{L^2}^2 \left(\tilde{\varepsilon}_0 + \left\| div \vec{u}^0 \right\|_{L^\infty}+\left\|\nabla \vec{u}^1 \right\|_{L^\infty} +  \left\| \nabla \left(p^0,p^1,\rho^0\right)\right\|_{L^2}\right),
	\end{aligned}
\end{equation}
which together with \eqref{puc1l} and the smallness of
$\tilde{\varepsilon}_0 + \bigl\| \nabla \bigl(p^0, \vec{u}^0, s^0\bigr)\bigr\|_{H^{l-1}}
+ \bigl\| \nabla^2 \vec{u}^1 \bigr\|_{H^{l-1}}
+ \bigl\| \nabla p^1 \bigr\|_{H^l}
+ \bigl\| \nabla \rho^0 \bigr\|_{H^2}$
results in that
\begin{equation}\label{puc0l}
	\begin{aligned}
		&\frac{1}{2} \sum_{|\alpha|=0}^l \frac{d}{dt} \int_{\mathbb{R}^3} \left( \frac{1}{\rho^0} \left(\partial_{\vec{x}}^\alpha p \right)^2 + \frac{P^0 \left(R+C_v\right)}{C_v}  \left|\partial_{\vec{x}}^\alpha \vec{u} \right|^2 \right) d\vec{x} 
		+ \frac{a\left( R+C_v\right)}{2C_v} \sum_{|\alpha|=0}^l \int_{\mathbb{R}^3} P^0 \left|\partial_{\vec{x}}^\alpha \vec{u} \right|^2 d\vec{x}\\
		&\le C\left|\bar{P}^0-\bar{P}^1\right| \cdot \left\|\nabla p\right\|_{H^{l-1}} \cdot \left\|\vec{u}^1 \right\|_{H^{l+1}} + \frac{C}{\tilde{\varepsilon}_0} \left\|\nabla p^1\right\|_{H^l}^2 \left(\left|\bar{P}^0-\bar{P}^1\right|^2+\left|\bar{S}^0-\bar{S}^1\right|^2 + \left\| \nabla s\right\|_{H^{l-1}}^2 \right)\\
		& \quad + C \left\| \nabla p\right\|_{H^{l-1}}^2 \left(\left\| \nabla \left(p^0, s^0\right)\right\|_{H^{l-1}} +\left\|\vec{u}^0\right\|_{H^l}+\left\|\nabla p^1\right\|_{H^l}+\left\|\vec{u}^1\right\|_{H^{l+1}}+\left\| \nabla \rho^0 \right\|_{H^2} \right).
	\end{aligned}
\end{equation}

Proceeding as in \eqref{pterm}--\eqref{pmlenes} and invoking Lemma~\ref{derivaineq}, \eqref{palpha}--\eqref{ualpha} and \eqref{rhocl2}, one obtains that
\begin{equation}\nonumber
	\begin{aligned}
		\frac{d}{dt} \int_{\mathbb{R}^3} \nabla \partial_{\vec{x}}^\alpha p \cdot \partial_{\vec{x}}^\alpha \vec{u} \, d\vec{x} + \frac{1}{\bar{\rho}^0} \left\| \nabla \partial_{\vec{x}}^\alpha p \right\|_{L^2}^2 &= \frac{\bar{P}^0 \left(R+C_v \right)}{C_v} \left\|div \partial_{\vec{x}}^\alpha \vec{u} \right\|_{L^2}^2 + \frac{\left(\bar{P}^0-\bar{P}^1 \right)\left( R+C_v\right) }{C_v} \int_{\mathbb{R}^3} div \partial_{\vec{x}}^\alpha \vec{u}^1 \, div \partial_{\vec{x}}^\alpha \vec{u} \, d\vec{x}\\
		&\quad - a\int_{\mathbb{R}^3} \nabla \partial_{\vec{x}}^\alpha p \cdot \partial_{\vec{x}}^\alpha \vec{u} \, d\vec{x} +\int_{\mathbb{R}^3} \left( -div \partial_{\vec{x}}^\alpha \vec{u} \, \partial_{\vec{x}}^\alpha H_p + \nabla \partial_{\vec{x}}^\alpha p \cdot \partial_{\vec{x}}^\alpha H_{\vec{u}} \right) d\vec{x}
	\end{aligned}
\end{equation}
and
\begin{equation}\nonumber
	\begin{aligned}
		&\sum_{|\alpha|=0}^{l-1} \left\|\partial_{\vec{x}}^\alpha H_p\right\|_{L^2} \le C\left\|\nabla p\right\|_{H^{l-1}} \left(\left\|\nabla \vec{u}^0\right\|_{H^{l-2}} + \left\|div \vec{u}^1\right\|_{H^{l-1}} \right) + C \left\|\nabla \vec{u}\right\|_{H^{l-1}} \left( \left\|\nabla p^0\right\|_{H^{l-2}}+\left\|\nabla p^1\right\|_{H^{l-1}}\right),\\
		&\sum_{|\alpha|=0}^{l-1} \left\|\partial_{\vec{x}}^\alpha H_{\vec{u}}\right\|_{L^2} \le C\left\|\nabla p^1\right\|_{H^{l-1}} \left(\left|\bar{P}^0-\bar{P}^1\right|+\left|\bar{S}^0-\bar{S}^1\right|+ \left\|\nabla s\right\|_{H^{l-2}}\right)+C\left\|\nabla p\right\|_{H^{l-1}} \left(\left\|\nabla \left(p^0,s^0\right)\right\|_{H^{l-2}}+ \left\|\nabla p^1\right\|_{H^{l-1}} \right) \\
		&\quad \quad \quad \quad \quad \quad \quad 
		 + C \left\|\nabla \vec{u}\right\|_{H^{l-1}} \left( \left\|\nabla \vec{u}^0\right\|_{H^{l-2}}+\left\|\nabla \vec{u}^1\right\|_{H^{l-1}}\right).
	\end{aligned}
\end{equation}
Thus for a sufficiently small positive constant $\tilde{\varepsilon}_1$, combined with the smallness of $\epsilon_0$ and $\epsilon_1$, one has that
\begin{equation}\label{pcgood}
	\begin{aligned}
		\sum_{|\alpha|=0}^{l-1} \frac{d}{dt} \int_{\mathbb{R}^3} \nabla \partial_{\vec{x}}^\alpha p \cdot \partial_{\vec{x}}^\alpha \vec{u} \, d\vec{x} + \frac{1}{2\bar{\rho}^0} \left\| \nabla p \right\|_{H^{l-1}}^2 &\le C \left|\bar{P}^0-\bar{P}^1\right|^2 \left(\frac{1}{\tilde{\varepsilon}_1} \left\| \nabla p^1 \right\|_{H^{l-1}}^2 + \left\|div \vec{u}^1 \right\|_{H^{l-1}}^2 \right)\\
		&\quad + \frac{C}{\tilde{\varepsilon}_1}\left\| \nabla p^1 \right\|_{H^{l-1}}^2  \left(\left|\bar{S}^0-\bar{S}^1\right|^2+ \left\|\nabla s\right\|_{H^{l-2}}^2\right) +C\left\|\vec{u}\right\|_{H^l}^2,
	\end{aligned}
\end{equation}
which together with \eqref{puc0l} and suitably small constants $\tilde{\varepsilon}_2, \tilde{\varepsilon}_3 \in (0,1)$ yields that
\begin{equation}\label{puhlt}
	\begin{aligned}
		&\frac{1}{2} \sum_{|\alpha|=0}^l \frac{d}{dt} \int_{\mathbb{R}^3} \left( \frac{1}{\rho^0} \left(\partial_{\vec{x}}^\alpha p \right)^2 + \frac{P^0 \left(R+C_v\right)}{C_v}  \left|\partial_{\vec{x}}^\alpha \vec{u} \right|^2 \right) d\vec{x} 
		+\tilde{\varepsilon}_2 \sum_{|\alpha|=0}^{l-1} \frac{d}{dt} \int_{\mathbb{R}^3} \nabla \partial_{\vec{x}}^\alpha p \cdot \partial_{\vec{x}}^\alpha \vec{u} \, d\vec{x}\\
		&\quad + \frac{a\left( R+C_v\right)}{4C_v} \sum_{|\alpha|=0}^l \int_{\mathbb{R}^3} P^0 \left|\partial_{\vec{x}}^\alpha \vec{u} \right|^2 d\vec{x} + \frac{\tilde{\varepsilon}_2}{4\bar{\rho}^0} \left\| \nabla p \right\|_{H^{l-1}}^2\\
		&\le C\left(\frac{1}{\tilde{\varepsilon}_0}+\frac{1}{\tilde{\varepsilon}_1}+\frac{1}{\tilde{\varepsilon}_3} \right)\left|\bar{P}^0-\bar{P}^1\right|^2 \left( \left\|\nabla p^1\right\|_{H^l}^2 +  \left\|\vec{u}^1 \right\|_{H^{l+1}}^2\right) + C \left( \frac{1}{\tilde{\varepsilon}_0}+\frac{1}{\tilde{\varepsilon}_1}\right)  \left\|\nabla p^1\right\|_{H^l}^2 \left(\left|\bar{S}^0-\bar{S}^1\right|^2 + \left\| \nabla s\right\|_{H^{l-1}}^2 \right).
	\end{aligned}
\end{equation}

Next, by virtue of \eqref{salpha} and Lemma~\ref{derivaineq}, one could apply the standard energy estimates for $\partial_{\vec{x}}^\alpha s$ to obtain that
\begin{equation}\nonumber
	\begin{aligned}
		\frac{1}{2} \frac{d}{dt}\left\|s\right\|_{H^l}^2 &=\sum_{|\alpha|=0}^l \left( -\int_{\mathbb{R}^3} \partial_{\vec{x}}^\alpha s \, \partial_{\vec{x}}^\alpha \left( \vec{u}\cdot\nabla {s}^1\right) d\vec{x} + \frac{1}{2} \int_{\mathbb{R}^3} div \vec{u}^0 \left(\partial_{\vec{x}}^\alpha s\right)^2 d\vec{x} \right)  -\sum_{|\alpha|=1}^l \int_{\mathbb{R}^3} \partial_{\vec{x}}^\alpha s \left[\partial_{\vec{x}}^\alpha, \vec{u}^0 \cdot \nabla\right] s \, d\vec{x}\\
		& \le C \left( \left\|s\right\|_{H^l} \cdot \left\|\nabla \vec{u}\right\|_{H^{l-1}} \cdot \left\|\nabla  s^1\right\|_{H^l} + \left\|s\right\|_{H^l}^2 \cdot \left\|\nabla \vec{u}^0\right\|_{H^{l-1}}\right),
	\end{aligned}
\end{equation}
which combined with Gronwall's inequality and \eqref{genenes}--\eqref{pude} yields that
\begin{equation}\label{scenes}
	\begin{aligned}
		\left\|s(\cdot, t)\right\|_{H^l} &\le e^{C\int_0^t \left\|\nabla \vec{u}^0(\cdot,\tau)\right\|_{H^{l-1}} d\tau} \left\|s_0\right\|_{H^l} + C\int_0^t e^{C\int_{\tau}^t \left\|\nabla \vec{u}^0(\cdot,t_1)\right\|_{H^{l-1}} dt_1}  \left\|\nabla \vec{u}(\cdot,\tau)\right\|_{H^{l-1}} \cdot \left\|\nabla s^1(\cdot,\tau)\right\|_{H^l} d\tau\\
		&\le C \left\|s_0\right\|_{H^l} + C\left\|s_0^1\right\|_{H^{l+1}} \int_0^t \left\|\nabla \vec{u}(\cdot,\tau)\right\|_{H^{l-1}} d\tau\\
		&\le C \left\|s_0\right\|_{H^l} + C\left\|s_0^1\right\|_{H^{l+1}} (1+t)^{1/2} \left( \int_0^t \left\|\nabla \vec{u}(\cdot,\tau)\right\|_{H^{l-1}}^2 d\tau\right)^{1/2}.
	\end{aligned}
\end{equation}

Integrating \eqref{puhlt} with respect to the time variable over $[0,t]$ and using the decay property in \eqref{pude} as well as the estimate for $\|s\|_{H^l}$ given by \eqref{scenes}, one finally has that
\begin{equation}\nonumber
	\begin{aligned}
		&\frac{1}{2} \sum_{|\alpha|=0}^l \int_{\mathbb{R}^3} \left( \frac{1}{\rho^0} \left(\partial_{\vec{x}}^\alpha p \right)^2 + \frac{P^0 \left(R+C_v\right)}{C_v}  \left|\partial_{\vec{x}}^\alpha \vec{u} \right|^2 \right) \! (\vec{x},t) \, d\vec{x} 
		+\tilde{\varepsilon}_2 \sum_{|\alpha|=0}^{l-1} \int_{\mathbb{R}^3} \left( \nabla \partial_{\vec{x}}^\alpha p \cdot \partial_{\vec{x}}^\alpha \vec{u}\right) \! (\vec{x},t) \, d\vec{x}\\
		&\quad + \frac{a\left( R+C_v\right)}{4C_v} \sum_{|\alpha|=0}^l \int_0^t \int_{\mathbb{R}^3} P^0 \left|\partial_{\vec{x}}^\alpha \vec{u} \right|^2(\vec{x},\tau) \, d\vec{x} d\tau + \frac{\tilde{\varepsilon}_2}{4\bar{\rho}^0} \int_0^t \left\| \nabla p(\cdot,\tau)\right\|_{H^{l-1}}^2 d\tau\\
		&\le C_{\tilde{\varepsilon}} \left[\left|\bar{P}^0-\bar{P}^1\right|^2+\left|\bar{S}^0-\bar{S}^1\right|^2 + \left\|\left(p_0,\vec{u}_0,s_0\right)\right\|_{H^l}^2 + \left\|s_0^1\right\|_{H^{l+1}}^2 \int_0^t (1+\tau)^{-\frac{5}{2}} (1+\tau) \int_0^{\tau} \left\|\nabla \vec{u}(\cdot,t_1)\right\|_{H^{l-1}}^2 dt_1 d\tau \right]\\
		&\le C_{\tilde{\varepsilon}} \left(\delta^2 + \left\|\left(p_0,\vec{u}_0,s_0\right)\right\|_{H^l}^2 + \epsilon_1^2 \int_0^t \left\|\nabla \vec{u}(\cdot,\tau)\right\|_{H^{l-1}}^2 d\tau\right),
	\end{aligned}
\end{equation}
where $C_{\tilde{\varepsilon}}$ is a constant depending on $\tilde{\varepsilon}_i^{-1}$ for $i=0,1,3$, and verifies \eqref{enpuc} since $\epsilon_1\ll 1$.

The bound for $\|s(\cdot,t)\|_{H^l}$ relies on the decay rate of $\|\nabla\vec{u}(\cdot,t)\|_{H^{l-1}}$ and thus we introduce two time‑dependent functionals:
\begin{equation}\label{E1}
	E_1(t):=\frac{1}{2} \sum_{|\alpha|=1}^l  \int_{\mathbb{R}^3} \frac{1}{\rho^0} \left(\partial_{\vec{x}}^\alpha p \right)^2 d\vec{x}+ \frac{R+C_v}{2C_v} \sum_{|\alpha|=0}^l  \int_{\mathbb{R}^3} P^0 \left|\partial_{\vec{x}}^\alpha \vec{u} \right|^2 d\vec{x} 
	+\tilde{\varepsilon}_2 \sum_{|\alpha|=0}^{l-1} \int_{\mathbb{R}^3} \nabla \partial_{\vec{x}}^\alpha p \cdot \partial_{\vec{x}}^\alpha \vec{u} \, d\vec{x},
\end{equation}
and
\begin{equation}\label{N1}
	N_1(t):= \underset{0 \le \tau \le t}{\sup} m^{-2}(\tau)(1+\tau)^{\frac{5}{2}} E_1(\tau),
\end{equation}
with $m(\tau):=\max\{\log(1+\tau),1\}$.

Invoking \eqref{ualpha}, \eqref{pp} and \eqref{rhocl2}, one obtains the following energy estimate for $\vec{u}$:
\begin{equation}\nonumber
	\begin{aligned}
		&\frac{1}{2} \frac{d}{dt} \int_{\mathbb{R}^3} P^0 \left|\vec{u} \right|^2 d\vec{x} + a \int_{\mathbb{R}^3} P^0 \left|\vec{u} \right|^2 d\vec{x}= \frac{1}{2} \int_{\mathbb{R}^3} \partial_t P^0 \left|\vec{u} \right|^2 d\vec{x} + \int_{\mathbb{R}^3} P^0 \vec{u} \cdot \left(H_{\vec{u}}-\frac{1}{\bar{\rho}^0} \nabla p \right) d\vec{x}\\
		&\le \frac{C}{\tilde{\varepsilon}_4}\left\|\nabla p^1\right\|_{L^2}^2 \left( \left|\bar{P}^0-\bar{P}^1\right|^2 + \left|\bar{S}^0-\bar{S}^1\right|^2+ \left\|s\right\|_{L^\infty}^2 \right) + C\left\|\nabla p\right\|_{L^2}^2 \left(1+\left\|\left(\rho^0-\bar{\rho}^0\right) \right\|_{L^\infty} \right) + C \left\|p\right\|_{L^\infty}^2 \cdot \left\|\nabla p^1\right\|_{L^2}\\
		& \quad +\frac{a}{4} \int_{\mathbb{R}^3} P^0 \left|\vec{u} \right|^2 d\vec{x}+ C\left[ \left\|\vec{u}\right\|_{L^2}^2 \left(\tilde{\varepsilon}_4 + \left\|\vec{u}^0\right\|_{W^{1,\infty}} + \left\|\nabla \vec{u}^1\right\|_{L^\infty}+ \left\|\left(\rho^0-\bar{\rho}^0\right) \right\|_{L^\infty} +\left\|\nabla p^1 \right\|_{L^2}\right) + \left\|\nabla \vec{u} \right\|_{L^2}^2 \cdot \left\|\vec{u}^0\right\|_{L^\infty}\right],
	\end{aligned}
\end{equation}
with $\tilde{\varepsilon}_4>0$ chosen sufficiently small. Since
$\bigl\|\vec{u}^0\bigr\|_{W^{1,\infty}} + \bigl\|\nabla \vec{u}^1\bigr\|_{L^\infty}+ \bigl\|\bigl(\rho^0-\bar{\rho}^0\bigr)\bigr\|_{L^\infty} +\bigl\|\nabla p^1\bigr\|_{L^2} \le C(\epsilon_0+\epsilon_1) \ll 1$, the above estimate yields that
\begin{equation}\label{ucenes0}
	\begin{aligned}
		\frac{1}{2} \frac{d}{dt} \int_{\mathbb{R}^3} P^0 \left|\vec{u} \right|^2 d\vec{x} + \frac{a}{2} \int_{\mathbb{R}^3} P^0 \left|\vec{u} \right|^2 d\vec{x} &\le \frac{C}{\tilde{\varepsilon}_4}\left\|\nabla p^1\right\|_{L^2}^2 \left( \left|\bar{P}^0-\bar{P}^1\right|^2 + \left|\bar{S}^0-\bar{S}^1\right|^2+ \left\|s\right\|_{L^\infty}^2 \right)\\
		&\quad + C\left( \left\|\nabla p\right\|_{L^2}^2 +  \left\|\nabla p\right\|_{H^1}^2 \cdot \left\|\nabla p^1\right\|_{L^2} + \left\|\nabla \vec{u} \right\|_{L^2}^2 \cdot \left\|\nabla \vec{u}^0\right\|_{H^1}\right),
	\end{aligned}
\end{equation}
which together with \eqref{puc1l} and \eqref{pcgood} results in that
\begin{equation}\label{pugood01l}
	\begin{aligned}
		&\frac{1}{2} \sum_{|\alpha|=1}^l \frac{d}{dt} \int_{\mathbb{R}^3} \frac{1}{\rho^0} \left(\partial_{\vec{x}}^\alpha p \right)^2  d\vec{x} +\frac{R+C_v}{2C_v} \sum_{|\alpha|=0}^l \frac{d}{dt} \int_{\mathbb{R}^3} P^0 \left|\partial_{\vec{x}}^\alpha \vec{u} \right|^2 d\vec{x} + \tilde{\varepsilon}_2 \sum_{|\alpha|=0}^{l-1} \frac{d}{dt} \int_{\mathbb{R}^3} \nabla \partial_{\vec{x}}^\alpha p \cdot \partial_{\vec{x}}^\alpha \vec{u} \, d\vec{x}\\
		&+ \frac{a\left( R+C_v\right)}{4C_v} \sum_{|\alpha|=0}^l \int_{\mathbb{R}^3} P^0 \left|\partial_{\vec{x}}^\alpha \vec{u} \right|^2 d\vec{x} +\frac{\tilde{\varepsilon}_2}{4\bar{\rho}^0} \left\| \nabla p \right\|_{H^{l-1}}^2\\
		&\le C_{\tilde{\tilde{\varepsilon}}} \left[ \left|\bar{P}^0-\bar{P}^1\right|^2 \left( \left\|div \vec{u}^1 \right\|_{H^l}^2 +  \left\|\nabla p^1\right\|_{H^l}^2\right) +  \left\|\nabla p^1\right\|_{H^l}^2 \left(\left|\bar{S}^0-\bar{S}^1\right|^2 + \left\| \nabla s\right\|_{H^{l-1}}^2 \right)\right]+C\left\|\nabla p\right\|_{L^2}^2.
	\end{aligned}
\end{equation}
Here, $C_{\tilde{\tilde{\varepsilon}}}$ is a positive constant depending on $C_{\tilde{\varepsilon}}$ and $\tilde{\varepsilon}_4$.

By virtue of the decay property \eqref{pude}, the estimate for $\|s\|_{H^l}$ given by \eqref{scenes}, the definitions \eqref{E1}--\eqref{N1} of $E_1(t)$ and $N_1(t)$, \eqref{pugood01l} and Gronwall's inequality, one has that
\begin{equation}\label{E1gron}
	\begin{aligned}
		E_1(t) &\le e^{-C_* t} E_1(0) +  C_{\tilde{\tilde{\varepsilon}}}\left(\left|\bar{P}^0-\bar{P}^1\right|^2+\left|\bar{S}^0-\bar{S}^1\right|^2 \right) \int_0^t e^{-C_*(t-\tau)} \left( \left\|div \vec{u}^1 (\cdot, \tau)\right\|_{H^l}^2 +  \left\|\nabla p^1(\cdot, \tau)\right\|_{H^l}^2\right) d\tau\\
		&\quad + C_{\tilde{\tilde{\varepsilon}}} \int_0^t e^{-C_*(t-\tau)} \left\|\nabla p^1(\cdot, \tau)\right\|_{H^l}^2  \left( \left\|s_0\right\|_{H^l}^2 + \left\|s_0^1\right\|_{H^{l+1}}^2 \left( \int_0^\tau \left\|\nabla \vec{u}(\cdot,t_1)\right\|_{H^{l-1}} dt_1 \right)^2 \right) d\tau\\
		&\quad + C\int_0^t e^{-C_*(t-\tau)} \left\|\nabla p(\cdot, \tau)\right\|_{L^2}^2 d\tau\\
		&\le e^{-C_* t}  E_1(0) +   C_{\tilde{\tilde{\varepsilon}}}(1+t)^{-\frac{5}{2}}\left\|\left(p_0^1,\vec{u}_0^1 \right) \right\|_{L^1 \cap H^{l+1}}^2 \left(\delta^2+ \left\|s_0\right\|_{H^l}^2 + \left\|s_0^1\right\|_{H^{l+1}}^2 N_1(t) \right)\\
		&\quad + C\int_0^t e^{-C_*(t-\tau)} \left\|\nabla p(\cdot, \tau)\right\|_{L^2}^2 d\tau
	\end{aligned}
\end{equation}
for a positive constant $C_*$. One can obtain the decay rate of $\left\|\nabla p(\cdot,t)\right\|_{L^2}$ by the Green's function with estimates given in \eqref{gf2kl2}--\eqref{gf1kl2} in Theorem~\ref{greenlp}:
\begin{equation}\nonumber
	\begin{aligned}
		\left\|\nabla p(\cdot,t)\right\|_{L^2} &\le C(1+t)^{-\frac{5}{4}} \left( \left\|p_0\right\|_{L^1 \cap H^1}+ (1+t)^{-\frac{1}{2}} \left\|\vec{u}_0\right\|_{L^1 \cap H^1}\right)\\
		&\quad + C\int_0^t e^{-R_0(t-\tau)} \left\|\nabla \left(H_p+\left(\bar{P}^1-\bar{P}^0\right)\frac{\left(R+C_v\right)}{C_v} div \vec{u}^1, H_{\vec{u}}\right) \!(\cdot,\tau)\right\|_{L^2} d\tau\\
		&\quad + C\int_0^t (1+t-\tau)^{-\frac{5}{4}} \left(\left\|H_p(\cdot,\tau)\right\|_{L^1} + (1+t-\tau)^{-\frac{1}{2}} \left\|\left(H_{\vec{u}}+\left(\frac{1}{\rho^0}-\frac{1}{\rho^1}\right)\nabla p^1\right) \!(\cdot,\tau)\right\|_{L^1} \right) d\tau\\
		&\quad + C\int_0^{t/2} (1+t-\tau)^{-\frac{7}{4}} \left(\left\|\left(\bar{P}^1-\bar{P}^0\right)\frac{\left(R+C_v\right)}{C_v} \vec{u}^1(\cdot,\tau)\right\|_{L^1} +  \left\|\left(\frac{1}{\rho^0}-\frac{1}{\rho^1}\right)\nabla p^1(\cdot,\tau)\right\|_{L^1} \right) d\tau\\
		&\quad + C\int_{t/2}^t (1+t-\tau)^{-\frac{1}{2}} \left(\left\|\left(\bar{P}^1-\bar{P}^0\right)\frac{\left(R+C_v\right)}{C_v} div \vec{u}^1(\cdot,\tau)\right\|_{L^2} + (1+t-\tau)^{-\frac{1}{2}} \left\|\left(\frac{1}{\rho^0}-\frac{1}{\rho^1}\right)\nabla p^1(\cdot,\tau)\right\|_{L^2} \right) d\tau.
	\end{aligned}
\end{equation}
Combining Theorem~\ref{pointwise}, Lemma~\ref{decay}, \eqref{rhocl2} and \eqref{scenes}, one has that
\begin{equation}\nonumber
	\begin{aligned}
		\left\|\nabla p(\cdot,t)\right\|_{L^2} &\le C(1+t)^{-\frac{5}{4}} \left( \left\|p_0\right\|_{L^1 \cap H^1}+ (1+t)^{-\frac{1}{2}} \left\|\vec{u}_0\right\|_{L^1 \cap H^1}\right)\\
		&\quad + C\epsilon_{01} \int_0^t e^{-R_0(t-\tau)}\left[(1+\tau)^{-\frac{7}{4}}  \left(\delta+\left\|s_0\right\|_{H^3}\right)+(1+\tau)^{-\frac{5}{4}} m(\tau) \sqrt{N_1(\tau)}\right] d\tau\\
		&\quad + C\epsilon_{01} \int_0^t (1+t-\tau)^{-\frac{5}{4}} m(\tau) \sqrt{N_1(\tau)}  \left[(1+\tau)^{-\frac{7}{4}} + (1+t-\tau)^{-\frac{1}{2}}(1+\tau)^{-\frac{5}{4}} \right] d\tau\\
		&\quad + C\epsilon_1 \int_0^{t/2} (1+t-\tau)^{-\frac{7}{4}} (1+\tau)^{-\frac{1}{2}} \left(\delta +\left\|s_0\right\|_{H^3}+  \sqrt{N_1(\tau)}\right) d\tau\\
		&\quad + C\epsilon_1 \int_{t/2}^t (1+t-\tau)^{-\frac{1}{2}} \left[(1+\tau)^{-\frac{7}{4}}\delta + (1+t-\tau)^{-\frac{1}{2}} (1+\tau)^{-\frac{5}{4}} \left(\delta +\left\|s_0\right\|_{H^3} +  \sqrt{N_1(\tau)}\right) \right] d\tau\\
		&\le C(1+t)^{-\frac{5}{4}} \left( \left\|p_0\right\|_{L^1 \cap H^1}+ (1+t)^{-\frac{1}{2}} \left\|\vec{u}_0\right\|_{L^1 \cap H^1}\right)+C\epsilon_{01} (1+t)^{-\frac{5}{4}}m(t)  \left(\delta +\left\|s_0\right\|_{H^3} + \sqrt{N_1(t)}\right),
	\end{aligned}
\end{equation}
where the following estimates are used:
\begin{equation}\nonumber
	\begin{aligned}
		&\left\|\nabla H_p(\cdot,\tau)\right\|_{L^2} \le C\left(\left\|\nabla p(\cdot,\tau)\right\|_{H^1}+\left\|\nabla \vec{u}(\cdot,\tau)\right\|_{H^1}\right)\left( \left\| \nabla^2 (p^1,\vec{u}^1)(\cdot,\tau)\right\|_{H^1} + \left\| \nabla(p^0, \vec{u}^0)(\cdot,\tau)\right\|_{H^2} \right)\\
		& \quad \quad \quad \quad \quad \ \ \le C\epsilon_{01}  (1+\tau)^{-\frac{5}{2}} m(\tau) \sqrt{N_1(\tau)},\\
		&\left\|\nabla H_{\vec{u}}(\cdot,\tau)\right\|_{L^2} \le C\left\|\nabla^2 p^1(\cdot,\tau)\right\|_{H^1} \left(\left|\bar{P}^0-\bar{P}^1\right| + \left|\bar{S}^0-\bar{S}^1\right| + \left\|\nabla p(\cdot,\tau)\right\|_{H^1} +\left\|\nabla s(\cdot,\tau)\right\|_{H^1}\right)\\
		&\quad \quad \quad \quad \quad \quad \quad + C \left\|\nabla \vec{u}(\cdot,\tau)\right\|_{H^1} \left(\left\|\nabla \vec{u}^0(\cdot,\tau)\right\|_{H^2}+\left\|\nabla^2 \vec{u}^1(\cdot,\tau)\right\|_{H^1} \right) + C \left\|\nabla p(\cdot,\tau)\right\|_{H^1} \cdot \left\|\nabla \left(\rho^0-\bar{\rho}^0\right) \! (\cdot,\tau)\right\|_{H^2}\\
		&\quad \quad \quad \quad \quad \quad \le C\epsilon_{01} \left[ (1+\tau)^{-\frac{7}{4}} \left(\delta+\left\|s_0\right\|_{H^3}+ \sqrt{N_1(\tau)}\right)+(1+\tau)^{-\frac{5}{4}} m(\tau) \sqrt{N_1(\tau)}\right],
	\end{aligned}
	\end{equation}
	\begin{equation}\label{hpul1}
	\begin{aligned}
		&\left\|H_p(\cdot,\tau)\right\|_{L^1} \le C\left( \left\|p(\cdot,\tau)\right\|_{L^\infty} \cdot \left\|div \vec{u}^1(\cdot,\tau)\right\|_{L^1} +\left\|\nabla p(\cdot,\tau)\right\|_{L^2} \cdot \left\|\vec{u}^0(\cdot,\tau)\right\|_{L^2} +  \left\|\vec{u}(\cdot,\tau)\right\|_{H^1} \cdot \left\|\left( p^0,\nabla p^1\right)(\cdot,\tau)\right\|_{L^2}\right)\\
		&\quad \quad \quad \quad \quad \le C\epsilon_{01} (1+\tau)^{-\frac{1}{2}} \left(\left\|\vec{u}(\cdot,\tau)\right\|_{H^1} +\left\|\nabla p(\cdot,\tau)\right\|_{H^1}\right) \le C\epsilon_{01} (1+\tau)^{-\frac{7}{4}} m(\tau) \sqrt{N_1(\tau)},
	\end{aligned}
\end{equation}
\begin{equation}\label{hpul1u}
	\begin{aligned}
		&\left\|H_{\vec{u}}+\left(\frac{1}{\rho^0}-\frac{1}{\rho^1}\right)\nabla p^1 (\cdot,\tau)\right\|_{L^1} \le C  \left\|\vec{u}(\cdot,\tau)\right\|_{H^1} \left(\left\|\vec{u}^0(\cdot,\tau)\right\|_{L^2}+\left\|\nabla \vec{u}^1(\cdot,\tau)\right\|_{L^2} \right) + C \left\|\nabla p(\cdot,\tau)\right\|_{L^2} \cdot \left\|\left(\rho^0-\bar{\rho}^0 \right) (\cdot,\tau)\right\|_{L^2}\\
		&\quad \quad \quad \quad \quad \quad \quad \quad  \quad \quad \quad \quad \le C\epsilon_{01} (1+\tau)^{-\frac{5}{4}} m(\tau) \sqrt{N_1(\tau)}, 
	\end{aligned}
\end{equation}
and
\begin{equation}\label{hpul12}
	\begin{aligned}
		 &\left\|\left(\frac{1}{\rho^0}-\frac{1}{\rho^1}\right)\nabla p^1 (\cdot,\tau)\right\|_{L^1} \le C\epsilon_1 (1+\tau)^{-\frac{1}{2}} \left(\left|\bar{P}^0-\bar{P}^1\right| + \left|\bar{S}^0-\bar{S}^1\right| + \left\|\nabla p(\cdot,\tau)\right\|_{H^1} +\left\|s_0\right\|_{H^3} + \left\|s_0^1\right\|_{H^4} \sqrt{N_1(\tau)}\right),\\
		 &\quad \quad \quad \quad \quad \quad \quad \quad \quad \quad  \le C\epsilon_1 (1+\tau)^{-\frac{1}{2}} \left(\delta +\left\|s_0\right\|_{H^3} +  \sqrt{N_1(\tau)}\right),\\
		 &\left\|\left(\frac{1}{\rho^0}-\frac{1}{\rho^1}\right)\nabla p^1 (\cdot,\tau)\right\|_{L^2} \le C\epsilon_1 (1+\tau)^{-\frac{5}{4}} \left(\delta +\left\|s_0\right\|_{H^3}+  \sqrt{N_1(\tau)}\right).
	\end{aligned}
\end{equation}

Substituting the above bound for $\|\nabla p\|_{L^2}$ into \eqref{E1gron} yields that
\begin{equation}\nonumber
		E_1(t) \le C(1+t)^{-\frac{5}{2}} \left( \left\|p_0\right\|_{L^1 \cap H^l}^2+ (1+t)^{-1} \left\|\vec{u}_0\right\|_{L^1 \cap H^l}^2\right)+C\epsilon_{01}^2 (1+t)^{-\frac{5}{2}} m^2(t) \left(\delta^2 +\left\|s_0\right\|_{H^l}^2 + N_1(t)\right),
\end{equation}
and thus
\begin{equation}\nonumber
	N_1(t) \le C\left\|(p_0,\vec{u}_0)\right\|_{L^1 \cap H^l}^2 + C\epsilon_{01}^2 \left(\delta^2 +\left\|s_0\right\|_{H^l}^2 + N_1(t)\right).
\end{equation}
By virtue of the smallness of $\epsilon_{01}$, one has that
\begin{equation}\nonumber
	N_1(t) \le C\left\|(p_0,\vec{u}_0)\right\|_{L^1 \cap H^l}^2 + C\epsilon_{01}^2 \left(\delta^2 +\left\|s_0\right\|_{H^l}^2\right),
\end{equation}
which yields that
\begin{equation}\nonumber
	\left\|\vec{u}(\cdot,t)\right\|_{H^l} +\left\|\nabla p(\cdot,t)\right\|_{H^{l-1}} \le C(1+t)^{-\frac{5}{4}}m(t) \left(\left\|(p_0,\vec{u}_0)\right\|_{L^1 \cap H^l} + \epsilon_{01} \left(\delta +\left\|s_0\right\|_{H^l}\right)\right).
\end{equation}
It together with \eqref{scenes} results in that
\begin{equation}\nonumber
		\left\|s(\cdot, t)\right\|_{H^l}\le C \left\|s_0\right\|_{H^l} + C\left\|s_0^1\right\|_{H^{l+1}} \left(\left\|(p_0,\vec{u}_0)\right\|_{L^1 \cap H^l} + \epsilon_{01} \delta\right).
\end{equation}
Thus one verifies \eqref{ensc} and \eqref{pucdecay1}, and finishes the proof.
\end{proof}

\subsection{Pointwise Estimates for the Entropy Function $s(\vec{x},t)$}
Following the arguments in \eqref{weightsk} and invoking \eqref{salpha}, one obtains that
\begin{equation}\nonumber
	\begin{aligned}
		\frac{1}{2} \frac{d}{dt} \left\| e^{|\vec{x}|}\partial_{\vec{x}}^\alpha s \right\|_{L^2}^2&=- \int_{\mathbb{R}^3} e^{2|\vec{x}|} \left[ \frac{1}{2}\vec{u}^0 \cdot \nabla \left(\partial_{\vec{x}}^\alpha s \right)^2 + \partial_{\vec{x}}^\alpha s \left( \partial_{\vec{x}}^\alpha \left( \vec{u}\cdot\nabla s^1\right) +\left[\partial_{\vec{x}}^\alpha, \vec{u}^0 \cdot \nabla\right] s\right) \right] d\vec{x}\\
		& \le \frac{1}{2} \int_{\mathbb{R}^3} div\left( e^{2|\vec{x}|} \vec{u}^0 \right) \left(\partial_{\vec{x}}^\alpha s \right)^2 d\vec{x} + \left\| e^{|\vec{x}|}\partial_{\vec{x}}^\alpha s \right\|_{L^2} \left(\left\| e^{|\vec{x}|} \partial_{\vec{x}}^\alpha \left( \vec{u}\cdot\nabla {s}^1\right) \right\|_{L^2} + \left\| e^{|\vec{x}|} \left[ \partial_{\vec{x}}^{\alpha}, \vec{u}^0\cdot\nabla\right]s \right\|_{L^2}\right),
	\end{aligned}
\end{equation}
which together with
\begin{equation}\nonumber
	\begin{aligned}
		\sum_{|\alpha|=1}^l \left\| e^{|\vec{x}|} \left[ \partial_{\vec{x}}^{\alpha}, \vec u^0 \cdot \nabla \right]s \right\|_{L^2} \le C \left\|\nabla \vec{u}^0 \right\|_{W^{l-1,\infty}} \sum_{i=1}^l \left\| e^{|\vec{x}|} \nabla^i s\right\|_{L^2},
	\end{aligned}
\end{equation}
\begin{equation}\nonumber
	\begin{aligned}
		\sum_{|\alpha|=0}^l \left\| e^{|\vec{x}|} \partial_{\vec{x}}^\alpha \left( \vec{u}\cdot\nabla s^1\right) \right\|_{L^2} &\le \sum_{|\alpha|=0}^l \left\| e^{|\vec{x}|}  \left( \vec{u}\cdot\nabla \partial_{\vec{x}}^\alpha s^1\right)  \right\|_{L^2} + \sum_{|\alpha|=1}^l \left\| e^{|\vec{x}|} \left[ \partial_{\vec{x}}^\alpha, \vec{u}\cdot\nabla\right]s^1 \right\|_{L^2} \\
		&\le C\left(\left\|\vec{u} \right\|_{L^\infty}  \sum_{|\alpha|=0}^l \left\|e^{|\vec{x}|} \nabla \partial_{\vec{x}}^\alpha s^1 \right\|_{L^2}+ \left\|\nabla \vec{u} \right\|_{H^{l-1}} \sum_{i=1}^l \left\| e^{|\vec{x}|} \nabla^i s^1\right\|_{L^\infty}\right) \le C \left\|\nabla \vec{u} \right\|_{H^{l-1}} \sum_{i=1}^{l+2} \left\| e^{|\vec{x}|} \nabla^i s^1\right\|_{L^2},
	\end{aligned}
\end{equation}
for $l \ge 2$, yields that
\begin{equation}\nonumber
		\frac{1}{2} \sum_{j=0}^l \frac{d}{dt} \left\| e^{|\vec{x}|} \nabla^j s \right\|_{L^2}^2 \le C \sum_{j=0}^l \left\| e^{|\vec{x}|} \nabla^j s \right\|_{L^2}  \left( \left\|\vec{u}^0 \right\|_{W^{l,\infty}} \sum_{j=0}^l \left\| e^{|\vec{x}|} \nabla^j s \right\|_{L^2} + \left\|\nabla \vec{u} \right\|_{H^{l-1}} \sum_{j=1}^{l+2} \left\| e^{|\vec{x}|} \nabla^j s^1\right\|_{L^2}\right).
\end{equation}
By Gronwall's inequality, the Sobolev embedding \eqref{wemb}, the estimate for $\| e^{|\vec{x}|} \nabla^j s^1\|_{L^2}$ in \eqref{expsk}, and the decay properties \eqref{pude} and \eqref{pucdecay1}, one has that
\begin{equation}\nonumber
	\begin{aligned}
		\sum_{j=0}^l \left\| e^{|\vec{x}|} \nabla^j s \right\|_{L^2}& \le \exp\left\lbrace C\int_0^t \left\|\nabla \vec{u}^0(\cdot,\tau) \right\|_{H^{l+1}} d\tau \right\rbrace 	\sum_{j=0}^l \left\| e^{|\vec{x}|} \nabla^j s_0 \right\|_{L^2} \\
		&\quad +C \exp\left\lbrace C\left\|\left(p_0^1, \vec{u}_0^1\right)\right\|_{L^1 \cap H^{l+2}} \right\rbrace \sum_{j=0}^{l+2} \left\| e^{|\vec{x}|} \nabla^j s_0^1 \right\|_{L^2} \int_0^t \exp\left\lbrace C\int_\tau^t \left\|\nabla \vec{u}^0(\cdot,t_1) \right\|_{H^{l+1}} dt_1\right\rbrace  \left\|\nabla \vec{u}(\cdot,\tau) \right\|_{H^{l-1}} d\tau\\
		&\le C\left[\sum_{j=0}^l \left\| e^{|\vec{x}|} \nabla^j s_0 \right\|_{L^2} +\sum_{j=0}^{l+2} \left\| e^{|\vec{x}|} \nabla^j s_0^1 \right\|_{L^2}\left(\left\|(p_0,\vec{u}_0)\right\|_{L^1 \cap H^l} + \epsilon_{01}\left(\delta+\left\|s_0\right\|_{H^l} \right)\right)  \right],
	\end{aligned}
\end{equation}
which combined with \eqref{weightsob} results in that
\begin{equation}\label{pointsc}
		\left|s(\vec{x},t)\right| %\le e^{-|\vec{x}|} \left\| e^{|\vec{x}|} s \right\|_{L^\infty} 
		\le Ce^{-|\vec{x}|} \sum_{i=0}^2 \left\| e^{|\vec{x}|} \nabla^i s \right\|_{L^2}
		\le Ce^{-|\vec{x}|} \left(\sum_{j=0}^2 \left\| e^{|\vec{x}|} \nabla^j s_0\right\|_{L^2} +\left\|(p_0,\vec{u}_0)\right\|_{L^1 \cap H^2} + \epsilon_{01} \delta \right) \le C\epsilon_{01}(\delta+\delta_0)e^{-|\vec{x}|}.
\end{equation}

\subsection{Temporal Decay Estimates for Higher-Order Derivatives of $(p,\vec{u})(\vec{x},t)$}
To characterize the pointwise space-time behaviors of $p$ and $\vec{u}$, the lemma below about decay estimates for their higher‑order derivatives is required:
\begin{lemma}
\label{decaypuc}
	Under the assumptions of Theorem~\ref{main}, for $l\ge 6$, the following temporal decay estimates hold for derivatives of the solution $(p,\vec{u})$ of \eqref{pc}\textup{:}
	\begin{equation}\label{decaypu3}
		\left\| \nabla p(\cdot, t) \right\|_{L^\infty}+\left\| \nabla^2 \vec{u}(\cdot, t) \right\|_{H^{l-2}} + \left\| \nabla^3 p(\cdot, t) \right\|_{H^{l-3}} \le C(1 + t)^{-2} \log(2+t) \left[\left\| \left( p_0, \vec{u}_0\right)  \right\|_{L^1 \cap H^l}+\epsilon_{01}\left(\delta +\left\|s_0\right\|_{H^l}\right)\right].
	\end{equation}
\end{lemma}
\begin{proof}
	In analogy with the derivations in \eqref{puc1l}, \eqref{pcgood}, and \eqref{ucenes0}, together with the global existence result \eqref{genenes} and the smallness of $\epsilon_0$, $\epsilon_1$, $\tilde{\varepsilon}_0$, $\tilde{\varepsilon}_1$ and $\tilde{\varepsilon}_4$, where $ \epsilon_j \ll \tilde{\varepsilon}_i$ for $i=0,1,4$ and $j=0,1$, one obtains that
\begin{equation}\nonumber
	\begin{aligned}
		&\frac{1}{2} \sum_{|\alpha|=3}^l \frac{d}{dt} \int_{\mathbb{R}^3} \left( \frac{1}{\rho^0} \left(\partial_{\vec{x}}^\alpha p \right)^2 + \frac{P^0 \left(R+C_v\right)}{C_v}  \left|\partial_{\vec{x}}^\alpha \vec{u} \right|^2 \right) d\vec{x} 
		+ \frac{a\left( R+C_v\right)}{2C_v} \sum_{|\alpha|=3}^l \int_{\mathbb{R}^3} P^0 \left|\partial_{\vec{x}}^\alpha \vec{u} \right|^2 d\vec{x}\\
		&\le C\left|\bar{P}^0-\bar{P}^1\right| \cdot \left\|\nabla^3 p\right\|_{H^{l-3}} \cdot \left\|\nabla^3 div \vec{u}^1 \right\|_{H^{l-3}} + \frac{C}{\tilde{\varepsilon}_0}  \left(\left|\bar{P}^0-\bar{P}^1\right|^2+\left|\bar{S}^0-\bar{S}^1\right|^2 + \left\|\nabla (p,s)\right\|_{H^{l-1}}^2 \right)\left(\left\|\nabla^4 p^1\right\|_{H^{l-3}}^2+\left\|\nabla p^1\right\|_{L^\infty}^2 \right) \\
		& \quad  + \frac{C}{\tilde{\varepsilon}_0} \left\|\vec{u}\right\|_{W^{1,\infty}}^2 \left(\left\|\nabla^3 \vec{u}^0\right\|_{H^{l-3}}^2+\left\|\nabla^4 \vec{u}^1\right\|_{H^{l-3}}^2\right)+ C\left\| \nabla^3 p\right\|_{H^{l-3}}^2 \left( \left\|\vec{u}^0\right\|_{W^{1,\infty}} +\left\| \nabla \left(p^0, p^1,\vec{u}^1, \rho^0\right)\right\|_{L^\infty}\right)\\ 
		& \quad + C\left\| \nabla^3 p\right\|_{H^{l-3}} \left[ \left\| p\right\|_{W^{1,\infty}} \left(\left\| \nabla^3 \vec{u}^0 \right\|_{H^{l-3}}+\left\| \nabla^3 div \vec{u}^1 \right\|_{H^{l-3}}\right) + \left\| \vec{u} \right\|_{W^{1,\infty}}\left(\left\| \nabla^3 p^0 \right\|_{H^{l-3}}+\left\| \nabla^4 p^1 \right\|_{H^{l-3}}\right) \right]\\
		&\quad + C\left\| \nabla p\right\|_{L^\infty}^2 \cdot \left\|\nabla \left(p^0, s^0\right)\right\|_{H^{l-1}},
	\end{aligned}
\end{equation}
\begin{equation}\nonumber
	\begin{aligned}
		&\sum_{|\alpha|=2}^{l-1} \frac{d}{dt} \int_{\mathbb{R}^3} \nabla \partial_{\vec{x}}^\alpha p \cdot \partial_{\vec{x}}^\alpha \vec{u} \, d\vec{x} + \frac{1}{2\bar{\rho}^0} \left\| \nabla^3 p \right\|_{H^{l-3}}^2 \\
		&\le C\left|\bar{P}^0-\bar{P}^1\right|^2 \cdot \left\|\nabla^2 div \vec{u}^1 \right\|_{H^{l-3}}^2 +\frac{C}{\tilde{\varepsilon}_1} \left(\left|\bar{P}^0-\bar{P}^1\right|^2+\left|\bar{S}^0-\bar{S}^1\right|^2+ \left\|\nabla s\right\|_{H^{l-2}}^2\right) \left( \left\| \nabla^3 p^1 \right\|_{H^{l-3}}^2+\left\| \nabla p^1 \right\|_{L^\infty}^2\right)\\
		&\quad + \frac{C}{\tilde{\varepsilon}_1} \left[ \left\|\nabla p\right\|_{H^1}^2  \left(\left\| \nabla^2 \vec{u}^0 \right\|_{H^{l-3}}^2+\left\| \nabla^3 \left(p^1,\vec{u}^1\right) \right\|_{H^{l-3}}^2+ \left\| \nabla \left(p^1,\vec{u}^1\right) \right\|_{L^\infty}^2\right) + \left\|\vec{u}\right\|_{W^{1,\infty}}^2 \left(\left\| \nabla^2 \left( p^0, \vec{u}^0\right) \right\|_{H^{l-3}}^2+\left\| \nabla^3 \left(p^1,\vec{u}^1\right) \right\|_{H^{l-3}}^2\right) \right]\\
		&\quad +C\left\|\nabla^2 \vec{u}\right\|_{H^{l-2}}^2
		+ C\left\|\nabla p\right\|_{L^\infty}^2 \cdot \left\|\nabla \left(p^0,s^0\right)\right\|_{H^{l-2}},
	\end{aligned}
\end{equation}
and 
\begin{equation}\nonumber
	\begin{aligned}
		&\frac{1}{2} \sum_{|\alpha|=2} \frac{d}{dt} \int_{\mathbb{R}^3} P^0 \left|\partial_{\vec{x}}^\alpha \vec{u} \right|^2 d\vec{x} + \frac{a}{2} \sum_{|\alpha|=2} \int_{\mathbb{R}^3} P^0 \left|\partial_{\vec{x}}^\alpha \vec{u} \right|^2 d\vec{x}\\
		&\le \frac{C}{\tilde{\varepsilon}_4} \left[ \left( \left|\bar{P}^0-\bar{P}^1\right|^2 + \left|\bar{S}^0-\bar{S}^1\right|^2+ \left\|\nabla (p,s)\right\|_{H^1}^2 \right) \left( \left\| \nabla^3 p^1 \right\|_{L^2}^2+\left\| \nabla p^1 \right\|_{L^\infty}^2\right)+ \left\|\vec{u} \right\|_{W^{1,\infty}}^2 \left(\left\| \nabla^2 \vec{u}^0 \right\|_{L^2}^2+\left\| \nabla^3 \vec{u}^1 \right\|_{L^2}^2\right) \right]\\
		&\quad+C\left[ \left\| \nabla^3 p\right\|_{L^2}^2+\left\| \nabla^3 \vec{u}\right\|_{L^2}^2 \cdot \left\|\vec{u}^0\right\|_{L^\infty} +\left\|\nabla p\right\|_{L^\infty}^2 \cdot \left\|\nabla \left(p^0,s^0\right)\right\|_{H^1}\right],
	\end{aligned}
\end{equation}
which together with a suitably small constant $\tilde{\varepsilon}_2 \in (0,1)$ satisfying $\epsilon_0+\epsilon_1 \ll \tilde{\varepsilon}_2$, yields that
\begin{equation}\nonumber
	\begin{aligned}
		&\frac{1}{2} \sum_{|\alpha|=3}^l \frac{d}{dt} \int_{\mathbb{R}^3} \frac{1}{\rho^0} \left(\partial_{\vec{x}}^\alpha p \right)^2 d\vec{x}+\frac{R+C_v}{2C_v} \sum_{|\alpha|=2}^l \frac{d}{dt} \int_{\mathbb{R}^3}  P^0 \left|\partial_{\vec{x}}^\alpha \vec{u} \right|^2 d\vec{x}  +\tilde{\varepsilon}_2 \sum_{|\alpha|=2}^{l-1} \frac{d}{dt} \int_{\mathbb{R}^3} \nabla \partial_{\vec{x}}^\alpha p \cdot \partial_{\vec{x}}^\alpha \vec{u} \, d\vec{x}\\
		&\quad + \frac{a\left( R+C_v\right)}{8C_v} \sum_{|\alpha|=2}^l \int_{\mathbb{R}^3} P^0 \left|\partial_{\vec{x}}^\alpha \vec{u} \right|^2 d\vec{x}+ \frac{\tilde{\varepsilon}_2}{4\bar{\rho}^0} \left\| \nabla^3 p \right\|_{H^{l-3}}^2 \\
		&\le C'_{\tilde{\varepsilon}} \left|\bar{P}^0-\bar{P}^1\right|^2 \cdot \left\|\nabla^2 div \vec{u}^1 \right\|_{H^{l-2}}^2 + C'_{\tilde{\varepsilon}} \left(\left|\bar{P}^0-\bar{P}^1\right|^2+\left|\bar{S}^0-\bar{S}^1\right|^2 + \left\|\nabla s\right\|_{H^{l-1}}^2 \right) \left(\left\|\nabla^3 p^1\right\|_{H^{l-2}}^2+\left\|\nabla p^1\right\|_{L^\infty}^2 \right)\\ 
		&\quad + C'_{\tilde{\varepsilon}}\left( \left\|\nabla p\right\|_{H^1}^2+ \left\|\nabla \vec{u}\right\|_{L^2}^2\right)  \left(\left\| \nabla^2 \left(p^0,\vec{u}^0\right) \right\|_{H^{l-2}}^2+\left\| \nabla^3 \left(p^1,\vec{u}^1\right) \right\|_{H^{l-2}}^2+ \left\| \nabla \left(p^1,\vec{u}^1\right) \right\|_{L^\infty}^2\right)\\
		&\quad +C\left(\left\| \nabla^3 p\right\|_{L^2}^2 + \left\| \nabla p\right\|_{L^\infty}^2 \cdot \left\|\nabla (p^0, s^0)\right\|_{H^{l-1}}\right),
	\end{aligned}
\end{equation}
where $C'_{\tilde{\varepsilon}}$ denotes a positive constant depending on $\tilde{\varepsilon}_i^{-1}$ for $i=0,1,2,4$.

Similar to \eqref{E1}--\eqref{N1}, define the following time-dependent functionals:
\begin{equation}\nonumber
\begin{aligned}
	E_3(t)&:=\frac{1}{2} \sum_{|\alpha|=3}^l  \int_{\mathbb{R}^3} \frac{1}{\rho^0} \left(\partial_{\vec{x}}^\alpha p \right)^2 d\vec{x}+ \frac{R+C_v}{2C_v} \sum_{|\alpha|=2}^l  \int_{\mathbb{R}^3} P^0 \left|\partial_{\vec{x}}^\alpha \vec{u} \right|^2 d\vec{x} 
	+\tilde{\varepsilon}_2 \sum_{|\alpha|=2}^{l-1}  \int_{\mathbb{R}^3} \nabla \partial_{\vec{x}}^\alpha p \cdot \partial_{\vec{x}}^\alpha \vec{u} \, d\vec{x},\\
	N_3(t)&:= \underset{0 \le \tau \le t}{\sup} m^{-2}(\tau)(1+\tau)^{4} \left( E_3(\tau)+\left\| \nabla p(\cdot,\tau)\right\|_{L^\infty}^2\right).
	\end{aligned}
\end{equation}
By virtue of Gronwall's inequality, the decay property \eqref{decaygen3}, the pointwise estimate of $\nabla(p^1,\vec{u}^1)$ from Theorem~\ref{pointwise}, and the estimates for $p,\vec{u},s$ in \eqref{enpuc}--\eqref{pucdecay1}, there exists a positive constant $D_*$ such that
\begin{equation}\label{E3gron}
	\begin{aligned}
		E_3(t) &\le e^{-D_* t} E_3(0) +  C'_{\tilde{\varepsilon}} \epsilon_{01}^2 (1+t)^{-4} \left(\left|\bar{P}^0-\bar{P}^1\right|^2+\left|\bar{S}^0-\bar{S}^1\right|^2 +\left\|\left(p_0,\vec{u}_0\right)\right\|_{L^1 \cap H^l}^2 + \left\|s_0\right\|_{H^l}^2 \right)\\
		&\quad +C\int_0^t e^{-D_*(t-\tau)} \left(\left\| \nabla^3 p(\cdot,\tau)\right\|_{L^2}^2 + \left\| \nabla p(\cdot,\tau)\right\|_{L^\infty}^2 \cdot \left\|\nabla \left(p^0, s^0\right) \!(\cdot,\tau)\right\|_{H^{l-1}}\right) d\tau.
	\end{aligned}
\end{equation}

Based on the $L^p$ estimates of the Green's function in Theorem~\ref{greenlp}, one could obtain the decay rate of $\bigl\| \nabla^3 p(\cdot,t)\bigr\|_{L^2}$ as follows:
\begin{equation}\label{p3l2}
	\begin{aligned}
		\left\|\nabla^3 p(\cdot,t)\right\|_{L^2} &\le C(1+t)^{-\frac{9}{4}} \left( \left\|p_0\right\|_{L^1 \cap H^3}+ (1+t)^{-\frac{1}{2}} \left\|\vec{u}_0\right\|_{L^1 \cap H^3}\right)\\
		&\quad + C\int_0^t e^{-R_0(t-\tau)} \left\|\nabla^3 \left(H_p+\left(\bar{P}^1-\bar{P}^0\right)\frac{\left(R+C_v\right)}{C_v} div \vec{u}^1, H_{\vec{u}}\right) \! (\cdot,\tau)\right\|_{L^2} d\tau\\
		&\quad + C\int_{t/2}^t (1+t-\tau)^{-\frac{3}{2}} \left(\left\|H_p(\cdot,\tau)\right\|_{L^2} + (1+t-\tau)^{-\frac{1}{2}} \left\|\left(H_{\vec{u}}+\left(\frac{1}{\rho^0}-\frac{1}{\rho^1}\right)\nabla p^1\right) \! (\cdot,\tau)\right\|_{L^2} \right) d\tau\\
		&\quad + C\int_{t/2}^t (1+t-\tau)^{-1} \left(\left\|\left(\bar{P}^1-\bar{P}^0\right)\frac{\left(R+C_v\right)}{C_v} \left( \nabla div \vec{u}^1\right) \! (\cdot,\tau)\right\|_{L^2} + \left\|\nabla^2 \left( \left(\frac{1}{\rho^0}-\frac{1}{\rho^1}\right)\nabla p^1\right) \! (\cdot,\tau)\right\|_{L^2} \right) d\tau\\
		&\quad + C\int_0^{t/2} (1+t-\tau)^{-\frac{9}{4}} \Bigg[\left\|H_p(\cdot,\tau)\right\|_{L^1} + (1+t-\tau)^{-\frac{1}{2}} \Bigg( \left\|\left(\bar{P}^1-\bar{P}^0\right)\frac{\left(R+C_v\right)}{C_v} \vec{u}^1(\cdot,\tau)\right\|_{L^1} +\left\|H_{\vec{u}}(\cdot,\tau)\right\|_{L^1} \Bigg) \Bigg] d\tau\\
		&\le C(1+t)^{-2}m(t)\left[\left\|(p_0,\vec{u}_0)\right\|_{L^1 \cap H^4}+\epsilon_{01}\left(\delta +\left\|s_0\right\|_{H^4}+\sqrt{N_3(t)}\right)\right].
	\end{aligned}
\end{equation}
In the above inequality, one uses \eqref{hpul1}--\eqref{hpul12} and the following estimates for the low‑order derivatives of $H_p$ and $H_{\vec{u}}$ gained from Lemma~\ref{derivaineq}, the decay estimates for $\nabla^i p$ and $\nabla^i \vec{u}$ in \eqref{pucdecay1}, and inequalities \eqref{rhocl2}:
\begin{equation}\nonumber
	\begin{aligned}
		&\left\|H_p(\cdot,\tau)\right\|_{L^2} \le C\epsilon_{01}(1+\tau)^{-2}\left(\delta+\left\|(p_0,\vec{u}_0)\right\|_{L^1 \cap H^3}+\left\|s_0\right\|_{H^3}\right),\\
		&\left\|\left(H_{\vec{u}}+\left(\frac{1}{\rho^0}-\frac{1}{\rho^1}\right)\nabla p^1\right) \!(\cdot,\tau)\right\|_{L^2} \le C\epsilon_{01} (1+\tau)^{-2} \left(\delta+\left\|(p_0,\vec{u}_0,s_0)\right\|_{H^3}\right)+C\epsilon_0 \left\|\nabla p(\cdot,\tau)\right\|_{L^\infty},\\
		&\left\|\nabla^2 \left( \left(\frac{1}{\rho^0}-\frac{1}{\rho^1}\right)\nabla p^1\right) \! (\cdot,\tau)\right\|_{L^2} \le C \epsilon_1 (1+\tau)^{-2} \left(\delta+\left\|(p_0,\vec{u}_0)\right\|_{L^1 \cap H^3}+\left\|s_0\right\|_{H^3}\right),
	\end{aligned}
\end{equation}
and the estimates involving higher‑order derivatives are presented in detail below:

\begin{equation}\nonumber
	\begin{aligned}
		&\left\|\nabla^3 \left(H_p+\left(\bar{P}^1-\bar{P}^0\right)\frac{\left(R+C_v\right)}{C_v} div \vec{u}^1, H_{\vec{u}}\right) \!(\cdot,\tau)\right\|_{L^2}\\
		&\le C\Bigg[\left| \bar{P}^0-\bar{P}^1\right| \cdot \left\|\nabla^3 div \vec{u}^1(\cdot,\tau)\right\|_{L^2} +\left\|\nabla p(\cdot,\tau)\right\|_{H^3}\left(\left\| \vec{u}^0(\cdot,\tau)\right\|_{L^\infty} +\left\|\nabla^3 \vec{u}^0(\cdot,\tau)\right\|_{L^2}+\left\|\nabla div \vec{u}^1(\cdot,\tau)\right\|_{H^2}\right)\\
		&\quad +\left\|\nabla \vec{u}(\cdot,\tau)\right\|_{H^3}\left(\left\|\left( p^0,\vec{u}^0\right) \! (\cdot,\tau)\right\|_{L^\infty} +\left\|\nabla^3 \left(p^0,\vec{u}^0\right) \! (\cdot,\tau)\right\|_{L^2} +\left\|\nabla^4 \left(p^1,\vec{u}^1\right) \!(\cdot,\tau)\right\|_{L^2}+\left\|\nabla \left(p^1,\vec{u}^1\right)\! (\cdot,\tau)\right\|_{L^\infty}\right)\\
		&\quad +\left\|\nabla^4 p(\cdot,\tau)\right\|_{L^2} \cdot \left\|\left( \rho^0-\bar{\rho}^0\right) \!(\cdot,\tau)\right\|_{L^\infty} + \left\|\nabla p(\cdot,\tau)\right\|_{L^\infty} \cdot \left\|\nabla \left(p^0,s^0\right) \!(\cdot,\tau)\right\|_{H^2}\\
		&\quad +\left(\delta + \left\|\nabla (p,s)(\cdot,\tau)\right\|_{H^2}\right)\left(\left\|\nabla^4 p^1(\cdot,\tau)\right\|_{L^2} +\left\|\nabla p^1(\cdot,\tau)\right\|_{L^\infty}\right) \Bigg]\\
		&\le C\epsilon_{01}(1+\tau)^{-2}\left(\delta+\left\|(p_0,\vec{u}_0)\right\|_{L^1 \cap H^4}+\left\|s_0\right\|_{H^4}\right) +C\epsilon_0 \left(\left\|\nabla p (\cdot,\tau)\right\|_{L^\infty}+\left\|\nabla^4 p (\cdot,\tau)\right\|_{L^2}\right).
	\end{aligned}
\end{equation}

The decay rate of $\bigl\|\nabla p(\cdot,t)\bigr\|_{L^\infty}$ could be obtained similarly:
\begin{equation}\label{pinfty}
	\begin{aligned}
		\left\|\nabla p(\cdot,t)\right\|_{L^\infty} &\le C\left[ (1+t)^{-2} \left( \left\|p_0\right\|_{L^1}+ (1+t)^{-\frac{1}{2}} \left\|\vec{u}_0\right\|_{L^1}\right)+e^{-t/\tilde{D}} \left( \left\|\nabla p_0\right\|_{W^{1,\infty}}+ \left\|\nabla \vec{u}_0\right\|_{W^{2,\infty}}\right)\right]\\
		&\quad + C\int_0^t e^{-(t-\tau)/\tilde{D}} \left( \left\|\nabla \left(H_p+\left(\bar{P}^1-\bar{P}^0\right)\frac{\left(R+C_v\right)}{C_v} div \vec{u}^1\right) \!(\cdot,\tau)\right\|_{W^{1,\infty}}+\left\|\nabla H_{\vec{u}}(\cdot,\tau)\right\|_{W^{2,\infty}}\right) d\tau\\
		&\quad + C\int_{t/2}^t (1+t-\tau)^{-\frac{1}{2}} \Bigg[\left\|H_p(\cdot,\tau)\right\|_{L^\infty} + (1+t-\tau)^{-\frac{1}{2}} \Bigg(\left\|\left(\bar{P}^1-\bar{P}^0\right)\frac{\left(R+C_v\right)}{C_v} \vec{u}^1(\cdot,\tau)\right\|_{L^\infty}\\
		&\quad \quad \quad \quad \quad +\left\|\left(H_{\vec{u}}-\left(\frac{1}{\bar{\rho}^0}-\frac{1}{\rho^0}\right)\nabla p\right) \!(\cdot,\tau)\right\|_{L^\infty} \Bigg)\Bigg] d\tau\\
		&\quad + C\int_{t/2}^t (1+t-\tau)^{-\frac{7}{4}} \left\| \left( \left(\frac{1}{\bar{\rho}^0}-\frac{1}{\rho^0}\right)\nabla p\right) \!(\cdot,\tau)\right\|_{L^2} d\tau\\
		&\quad + C\int_0^{t/2} (1+t-\tau)^{-2} \Bigg[\left\|\left( H_p+\frac{R+C_v}{C_v} p^0 div \vec{u}\right) \! (\cdot,\tau)\right\|_{L^1} +\left\|\frac{R+C_v}{C_v}\left( \vec{u} \cdot \nabla p^0\right) \!(\cdot,\tau)\right\|_{L^1}\\
		&\quad \quad \quad \quad \quad +(1+t-\tau)^{-\frac{1}{2}} \Bigg( \left\|\left(\bar{P}^1-\bar{P}^0\right)\frac{\left(R+C_v\right)}{C_v} \vec{u}^1(\cdot,\tau)\right\|_{L^1}+\left\|\frac{R+C_v}{C_v} \left( p^0 \vec{u}\right) \!(\cdot,\tau)\right\|_{L^1} +\left\|H_{\vec{u}}(\cdot,\tau)\right\|_{L^1} \Bigg) \Bigg] d\tau\\
		&\le C\left[ (1+t)^{-2} \left( \left\|p_0\right\|_{L^1}+ (1+t)^{-\frac{1}{2}} \left\|\vec{u}_0\right\|_{L^1}\right)+e^{-t/\tilde{D}} \left( \left\|\nabla p_0\right\|_{W^{1,\infty}}+ \left\|\nabla \vec{u}_0\right\|_{W^{2,\infty}}\right)\right]\\
		&\quad +C\epsilon_{01}(1+t)^{-2}m(t) \left(\delta+\left\|(p_0,\vec{u}_0)\right\|_{L^1 \cap H^6}+\left\|s_0\right\|_{H^6}+\sqrt{N_3(t)}\right)\\
		&\le C(1+t)^{-2}m(t)\left[\left\|(p_0,\vec{u}_0)\right\|_{L^1 \cap H^6}+\epsilon_{01}\left(\delta +\left\|s_0\right\|_{H^6}+\sqrt{N_3(t)}\right)\right],
	\end{aligned}
\end{equation}
where one uses the following estimates:
\begin{equation}\nonumber
	\begin{aligned}
		&\left\|\nabla \left(H_p+\left(\bar{P}^1-\bar{P}^0\right)\frac{\left(R+C_v\right)}{C_v} div \vec{u}^1\right) \!(\cdot,\tau)\right\|_{W^{1,\infty}}+\left\|\nabla H_{\vec{u}}(\cdot,\tau)\right\|_{W^{2,\infty}}\\
		&\le C\Bigg[\left|\bar{P}^0-\bar{P}^1\right| \cdot \left\|\nabla div \vec{u}^1(\cdot,\tau)\right\|_{W^{1,\infty}} + \left\|p(\cdot,\tau)\right\|_{W^{3,\infty}}\left( \left\|\vec{u}^0 (\cdot,\tau)\right\|_{W^{2,\infty}}+ \left\|\nabla \left(p^1, \vec{u}^1\right) \!(\cdot,\tau)\right\|_{W^{2,\infty}}\right)\\
		& \quad + \left\|\vec{u} (\cdot,\tau)\right\|_{W^{4,\infty}} \left(\left\|p^0(\cdot,\tau)\right\|_{W^{2,\infty}}+\left\|\vec{u}^0 (\cdot,\tau)\right\|_{W^{3,\infty}}+\left\|\nabla p^1(\cdot,\tau)\right\|_{W^{2,\infty}} +\left\|\nabla \vec{u}^1 (\cdot,\tau)\right\|_{W^{3,\infty}}\right)\\
		&\quad  + \left\|\nabla p (\cdot,\tau)\right\|_{W^{3,\infty}} \cdot \left\|\left(\frac{1}{\bar{\rho}^0}-\frac{1}{\rho^0}\right) \! (\cdot,\tau)\right\|_{W^{3,\infty}} +\left(\delta+\left\|(p,s) (\cdot,\tau)\right\|_{W^{3,\infty}}\right) \left\|\nabla p^1 (\cdot,\tau)\right\|_{W^{3,\infty}}\Bigg]\\
		&\le C\epsilon_{01}(1+\tau)^{-2}\left(\delta+\left\|(p_0,\vec{u}_0)\right\|_{L^1 \cap H^6}+\left\|s_0\right\|_{H^6}\right) +C\epsilon_0 \left(\left\|\nabla p (\cdot,\tau)\right\|_{L^\infty}+\left\|\nabla^3 p (\cdot,\tau)\right\|_{H^3} \right),
	\end{aligned}
\end{equation}
\begin{equation}\nonumber
	\begin{aligned}
		&\left\|H_p(\cdot,\tau)\right\|_{L^\infty} \le C\epsilon_{01} (1+\tau)^{-\frac{5}{4}} \left\|\nabla(p,\vec{u})(\cdot,\tau)\right\|_{H^2} \le C\epsilon_{01} (1+\tau)^{-\frac{5}{2}}m(t)\left(\delta+\left\|(p_0,\vec{u}_0)\right\|_{L^1 \cap H^3}+\left\|s_0\right\|_{H^3}\right),\\
		&\left\|\left(H_{\vec{u}}-\left(\frac{1}{\bar{\rho}^0}-\frac{1}{\rho^0}\right)\nabla p\right) \!(\cdot,\tau)\right\|_{L^\infty} \le C\epsilon_{01} (1+\tau)^{-2} \left(\delta+\left\|(p_0,\vec{u}_0)\right\|_{L^1 \cap H^3}+\left\|s_0\right\|_{H^3}\right),\\
		&\left\| \left( \left(\frac{1}{\bar{\rho}^0}-\frac{1}{\rho^0}\right)\nabla p\right) \! (\cdot,\tau)\right\|_{L^2}\le C\epsilon_0 \left\|\nabla p(\cdot,\tau)\right\|_{L^\infty},
	\end{aligned}
\end{equation}
\begin{equation}\nonumber
	\begin{aligned}
		&\left\|\left( H_p+\frac{R+C_v}{C_v} p^0 div \vec{u}\right) \! (\cdot,\tau)\right\|_{L^1}+\left\|\frac{R+C_v}{C_v}\left( \vec{u} \cdot \nabla p^0\right) \!(\cdot,\tau)\right\|_{L^1} \le C\epsilon_{01} (1+\tau)^{-\frac{5}{4}} \left(\delta+\left\|(p_0,\vec{u}_0,s_0)\right\|_{H^3}\right),\\
		&\left\|H_{\vec{u}}(\cdot,\tau)\right\|_{L^1}
		\le C\epsilon_{01} (1+\tau)^{-\frac{1}{2}}\left(\delta+\left\|(p_0,\vec{u}_0)\right\|_{L^1 \cap H^3}+\left\|s_0\right\|_{H^3}\right).
		\end{aligned}
\end{equation}

Substituting \eqref{p3l2} and \eqref{pinfty} into \eqref{E3gron}, one obtains
\begin{equation}\nonumber
	\begin{aligned}
		E_3(t) &\le e^{-D_*t}E_3(0)+C'_{\tilde{\varepsilon}}\epsilon_{01}^2 (1+t)^{-4}\left(\delta^2+\left\|(p_0,\vec{u}_0)\right\|_{L^1 \cap H^l}^2+\left\|s_0\right\|_{H^l}^2\right)\\
		&\quad +C(1+t)^{-4}m^2(t)\left[\left\|(p_0,\vec{u}_0)\right\|_{L^1 \cap H^6}^2+\epsilon_{01}^2\left(\delta^2 +\left\|s_0\right\|_{H^6}^2+N_3(t)\right)\right] +C\epsilon_0 (1+t)^{-4}m^2(t) N_3(t)\\
		&\le C'_{\tilde{\varepsilon}}(1+t)^{-4}m^2(t)\left[\left\|(p_0,\vec{u}_0)\right\|_{L^1 \cap H^l}^2+\epsilon_{01}^2\left(\delta^2 +\left\|s_0\right\|_{H^l}^2\right)\right] +C\epsilon_{01} (1+t)^{-4}m^2(t) N_3(t),
	\end{aligned}
\end{equation}
which combined with \eqref{pinfty} and the smallness of $\epsilon_{01}$ yields that
\begin{equation}\nonumber
	N_3(t) \le C'_{\tilde{\varepsilon}}\left[\left\|(p_0,\vec{u}_0)\right\|_{L^1 \cap H^l}^2+\epsilon_{01}^2\left(\delta^2 +\left\|s_0\right\|_{H^l}^2\right)\right],
\end{equation}
and this verifies \eqref{decaypu3}.
\end{proof}

\subsection{Pointwise Estimates for $(p,\vec{u})(\vec{x},t)$}
By Duhamel's principle, the solution $(p,\vec{u})$ of \eqref{pc} could be represented as follows:
\begin{equation}\label{jdefine}
	\begin{aligned}
	\vec{J}(\vec{x},t)
	&= \mathbb{F}^0(\vec{x},t) \underset{\vec{x}}{\ast} \vec{J}(\vec{x},0) + \int_0^t \left( \mathbb{F}^0- \mathbb{F}^{0*}\right) \!(\vec{x},t-\tau) \underset{\vec{x}}{\ast} 
	\begin{pmatrix}
		\tilde{H}_p \vspace{1.5mm}\\
		H_{\vec{u}}
	\end{pmatrix}(\vec{x},\tau) \, d\tau
	+ \int_0^t \left( \mathbb{F}^{0*}- \mathbb{F}_S^0\right) \! (\vec{x},t-\tau) \underset{\vec{x}}{\ast} 
	\begin{pmatrix}
		\tilde{H}_p \vspace{1.5mm}\\
		H_{\vec{u}}
	\end{pmatrix}(\vec{x},\tau) \, d\tau\\
	&\quad  + \int_0^t \mathbb{F}_S^0(\vec{x},t-\tau) \underset{\vec{x}}{\ast} 
	\begin{pmatrix}
		\tilde{H}_p \vspace{1.5mm}\\
		H_{\vec{u}}
	\end{pmatrix}(\vec{x},\tau) \, d\tau\\
	&: = \mathbb{F}^0(\vec{x},t) \underset{\vec{x}}{\ast} \vec{J}(\vec{x},0)+\vec{S}_r(\vec{x},t)+\vec{S}_*(\vec{x},t)+\vec{S}_c(\vec{x},t).
\end{aligned}
\end{equation}
Here, one denotes $\vec{J}$ and $\tilde{H}_p$:
\begin{equation} \label{jh}
	\vec{J}:=\left({J}_1,\vec{J}_2\right)^t:=\bigl(p,\vec{u}\bigr)^t,\quad \tilde{H}_p:=H_p+\left(\bar{P}^1-\bar{P}^0\right)\frac{\left(R+C_v\right)}{C_v} div \vec{u}^1,
\end{equation}
with $H_p$ and $H_{\vec{u}}$ defined by \eqref{pc}.
	
Combining the initial condition \eqref{inicon} and the inequalities \eqref{fs1k}--\eqref{greenfi}, \eqref{fiinitial}--\eqref{fsikini} by setting $k=0$, one obtains the following estimate consistent with \eqref{initialcon}:
\begin{equation}\label{iniconpuc}
	\begin{aligned}
		\left|\mathbb{F}_i^0(\vec{x},t) \underset{\vec{x}}{\ast} \vec{J}(\vec{x},0)\right| &\le \left|\left( \mathbb{F}_i^0-\mathbb{F}_{S,i}^0\right) \!(\vec{x},t) \underset{\vec{x}}{\ast} \vec{J}(\vec{x},0)\right| + \left|\left( \mathbb{F}_{S,i}^0 (\vec{x},t) \underset{\vec{x}}{\ast} \vec{J}(\vec{x},0)\right)\right|\\
		&\le C\epsilon_{01} \delta_0 \left((1+t)^{-3/2-d_i} e^{-\frac{|\vec{x}|^2}{4Dt}} + e^{-(|\vec{x}|+t)/K_0} \right), 
	\end{aligned}
\end{equation}
with a positive constant given by $K_0:=8\max\{ 2,D,1/a,8c_0/a\}$.

From \eqref{fs1k}--\eqref{greenfi},  $\vec{S}_r(\vec{x},t)$, $\vec{S}_*(\vec{x},t)$ and $\vec{S}_c(\vec{x},t)$ can be estimated as follows:
\begin{equation}\label{sri}
	\begin{aligned}
		\left|\vec{S}_{r,i}(\vec{x},t)\right| &= \left|\left( \mathbb{F}_i^0-\mathbb{F}_i^{0*}\right) \! (\vec{x},t)\underset{(\vec{x},t)}{\ast} 
		\begin{pmatrix}
			\tilde{H}_p \vspace{1.5mm}\\
			H_{\vec{u}}
		\end{pmatrix}(\vec{x},t) \right|\\
		&\le C\int_0^t (1+t-\tau)^{-3/2-d_i} e^{-\frac{|\vec{x}|^2}{D(t-\tau)}} \underset{\vec{x}}{\ast} \left[\left|H_p(\vec{x},\tau)\right|+(1+t-\tau)^{-1/2} \left( \left|\left(\bar{P}^1-\bar{P}^0\right)\vec{u}^1(\vec{x},\tau)\right|+\left|H_{\vec{u}}(\vec{x},\tau)\right|\right) \right] d\tau, 
	\end{aligned}
\end{equation}

\begin{equation}\label{sstari}
		\left|\vec{S}_{*}(\vec{x},t)\right|= \left|\left( \mathbb{F}_i^{0*}-\mathbb{F}_{S,i}^0\right) \!(\vec{x},t)\underset{(\vec{x},t)}{\ast} 
		\begin{pmatrix}
			\tilde{H}_p \vspace{1.5mm}\\
			H_{\vec{u}}
		\end{pmatrix}(\vec{x},t) \right| \le C\int_0^t e^{-(|\vec{x}|+t-\tau)/D} \underset{\vec{x}}{\ast} \left(\left|\tilde{H}_p(\vec{x},\tau)\right|+\left|H_{\vec{u}} (\vec{x},\tau)\right| \right) d\tau, 
\end{equation}

\begin{equation}\label{sc1}
	\begin{aligned}
		\left|\vec{S}_{c,1}(\vec{x},t)\right|&=\left|\mathbb{F}_{S,1}^0(\vec{x},t)\underset{(\vec{x},t)}{\ast} 
		\begin{pmatrix}
			\tilde{H}_p \vspace{1.5mm}\\
			H_{\vec{u}}
		\end{pmatrix}(\vec{x},t) \right|\\
		& \le C\int_0^t e^{-a(t-\tau)/4} \int_{|\vec{y}|=1} \left(\left|\tilde{H}_p\right|+\left|\nabla \tilde{H}_p\right|+\left|div H_{\vec{u}}\right|+ \left|\nabla div H_{\vec{u}}\right| \right) \!(\vec{x}+c_0 (t-\tau)\vec{y}, \tau) \, dS(\vec{y})d\tau\\
		&\quad + C\left(1+|\vec{x}|^{-1}\right)e^{-at/8}\left( e^{-|\vec{x}|/2}+e^{-a|\vec{x}|/(16c_0)}\right) \underset{(\vec{x},t)}{\ast} \left( \left| \tilde{H}_p(\vec{x},t)\right|+\left|div H_{\vec{u}}(\vec{x},t)\right|\right),
	\end{aligned}
\end{equation}
and
\begin{equation}\label{sc2}
	\begin{aligned}
		\left|\vec{S}_{c,2}(\vec{x},t)\right|&=\left|\mathbb{F}_{S,2}^0(\vec{x},t)\underset{(\vec{x},t)}{\ast} 
		\begin{pmatrix}
			\tilde{H}_p \vspace{1.5mm}\\
			H_{\vec{u}}
		\end{pmatrix}(\vec{x},t) \right|\\
	& \le \left| e^{-at} \underset{t}{\ast} H_{\vec{u}}(\vec{x},t)\right| + C\int_0^t e^{-a(t-\tau)/4} \int_{|\vec{y}|=1} \left(\left|\nabla \tilde{H}_p\right|+\left|\nabla^2 \tilde{H}_p\right|\right) \!(\vec{x}+c_0 (t-\tau)\vec{y},\tau) \, dS(\vec{y}) d\tau\\
	&\quad + C\left(1+|\vec{x}|^{-1}\right)e^{-at/8}\left( e^{-|\vec{x}|/4}+e^{-a|\vec{x}|/(16c_0)}\right) \underset{(\vec{x},t)}{\ast} \left( \left|\nabla \tilde{H}_p(\vec{x},t)\right|+\left|\nabla div H_{\vec{u}} (\vec{x},t)\right|\right).
	\end{aligned}
\end{equation}
Here, parameters $d_i (i=1,2)$ is defined in \eqref{di} and the sound speed $c_0=\displaystyle\sqrt{\frac{\bar{P}^0(R+C_v)}{C_v \bar{\rho}^0}}$. The further inverstigation about the above expressions requires the following estimates for $\tilde{H}_p$ and $H_{\vec{u}}$:
\begin{lemma}\label{lehpuc}
	The terms $\bigl(\tilde{H}_p, H_{\vec{u}}\bigr)$ defined in \eqref{pc} and \eqref{jh} satisfy that
	\begin{equation} \label{hpucbound}
		\begin{aligned}
			&\left| H_p(\vec{x},t)\right| \le C\epsilon_{01}\left(\delta+ \delta_0\right) (1+t)^{-5/4}\log(2+t)  \left((1+t)^{-3/2} e^{-|\vec{x}|^2/(K_* t)} + e^{-(|\vec{x}|+t)/K_*}\right),\\
			&\left|\nabla^i \tilde{H}_p(\vec{x},t)\right| \le C\epsilon_{01} \left(\delta+ \delta_0\right) \left((1+t)^{-2} e^{-|\vec{x}|^2/(K_* t)} + e^{-(|\vec{x}|+t)/K_*}\right), \quad \text{for } i=1,2,\\
			&\left|\nabla^j H_{\vec{u}}(\vec{x},t)\right| \le C\epsilon_{01}\left(\delta+ \delta_0\right)\left[(1+t)^{-2}\left(e^{-|\vec{x}|^2/(K_* t)} +\log(2+t) e^{-|\vec{x}|}\right)+e^{-(|\vec{x}|+t)/K_*}\right], \quad \text{for } j=0,1,2,\\
		\end{aligned}
	\end{equation}
with $\epsilon_{01}=\epsilon_0+\epsilon_1$.
Here, $\bigl| \bar{P}^0-\bar{P}^1\bigr|$, $\bigl|\bar{S}^0-\bar{S}^1\bigr|=O(1)\delta$, $\epsilon_k$ $(k=0,1)$ and $\delta_0$ are defined in \eqref{inicon}, and $K_*:=\max\{D_0,D_1\}$, with $D_0$ and $D_1$ defined in \eqref{poinconclu}.
\end{lemma}
\begin{proof}
	Theorem~\ref{pointwise}, Theorem~\ref{enlec} and Lemma~\ref{decaypuc}, together with \eqref{wemb}, \eqref{rhok} and \eqref{rhocl2} yield that
	\begin{equation}\nonumber
		\begin{aligned}
	 		\left| H_p(\vec{x},t)\right| 
	 		&\le C\left\|(p,\vec{u})(\cdot,t)\right\|_{W^{1,\infty}} \left(\left| div \vec{u}^1\right|+ \left|p^0\right|+ \left|\nabla p^1\right|+\left| \vec{u}^0\right|\right) \!(\vec{x},t)\\
	 		&\le C\epsilon_{01}\left(\delta+ \delta_0\right) (1+t)^{-5/4}\log(2+t)  \left((1+t)^{-3/2} e^{-|\vec{x}|^2/(K_* t)} + e^{-(|\vec{x}|+t)/K_*}\right),\\
	 		\left| H_{\vec{u}}(\vec{x},t) \right|&\le C\left[\left|\vec{u}\right| \cdot \left| \nabla \vec{u}^1\right| + \left|\vec{u}^0\right| \cdot  \left|\nabla \vec{u} \right| +\left|\left(p^0,s^0\right)\right| \cdot \left|\nabla p\right|\right.
	 		\left. +\left(\left|\bar{P}^0-\bar{P}^1\right|+\left|\bar{S}^0-\bar{S}^1\right|+\left| p\right| +\left| s\right| \right)\left|\nabla p^1\right|\right] \!(\vec{x},t)\\
	 		&\le C\epsilon_{01} \left(\delta+\delta_0\right) \left[  (1+t)^{-5/4}\log(2+t)  \cdot (1+t)^{-3/2} e^{-|\vec{x}|^2/(K_* t)} + (1+t)^{-2}\log(2+t) e^{-|\vec{x}|}\right.\\
	 		&\left. \quad \quad \quad \quad \quad \quad \quad+(1+t)^{-2} e^{-|\vec{x}|^2/(K_* t)}+e^{-(|\vec{x}|+t)/K_*}\right]\\
	 		&\le C\epsilon_{01}\left(\delta+ \delta_0\right)\left[(1+t)^{-2}\left(e^{-|\vec{x}|^2/(K_* t)} +\log(2+t) e^{-|\vec{x}|}\right)+e^{-(|\vec{x}|+t)/K_*}\right],
	 	\end{aligned}
	 \end{equation}
and for $i=1,2$,
\begin{equation}\nonumber
	\begin{aligned}
		\left|\nabla^i \tilde{H}_p(\vec{x},t)\right| &\le C \left[ \left|\bar{P}^1-\bar{P}^0\right| \cdot \left|\nabla^i div \vec{u}^1(\vec{x},t)\right|+\left\|\nabla^i div \vec{u}(\cdot,t)\right\|_{L^\infty} \cdot \left|p^0(\vec{x},t)\right|\right]\\
		&\quad  + C \left( \left\|p(\cdot,t)\right\|_{W^{i+1,\infty}}+\left\|\vec{u}(\cdot,t)\right\|_{W^{i,\infty}}\right)\left(\left|\vec{u}^0 \right| + \sum_{j=1}^i \left|\nabla^j \left(p^0,p^1,\vec{u}^0,\vec{u}^1 \right) \right| + \left|\nabla^{i+1} \left(p^1,\vec{u}^1 \right) \right|\right) \!(\vec{x},t)\\
		&\le C\epsilon_{01} \left(\delta+ \delta_0\right) \left((1+t)^{-2} e^{-|\vec{x}|^2/(K_* t)} + e^{-(|\vec{x}|+t)/K_*}\right),\\
		\left|\nabla^i H_{\vec{u}}(\vec{x},t)\right| & \le C \left\|\vec{u}(\cdot,t)\right\|_{W^{i+1,\infty}}\left(\left|\vec{u}^0 \right|+\sum_{j=1}^i \left|\nabla^j \left(\vec{u}^0,\vec{u}^1 \right) \right| + \left|\nabla^{i+1} \vec{u}^1 \right|\right) \!(\vec{x},t)+C\left\|\nabla p(\cdot,t)\right\|_{W^{i,\infty}} \sum_{j=0}^i \left|\nabla^j \left(p^0,s^0 \right) \!(\vec{x},t)\right|\\
		&\quad + C\left(\left|\bar{P}^0-\bar{P}^1\right|+\left|\bar{S}^0-\bar{S}^1\right|+\sum_{j=0}^i \left|\nabla^j (p,s)(\vec{x},t)\right| \right) \sum_{j=1}^{i+1} \left|\nabla^j p^1(\vec{x},t)\right|\\
		&\le C\epsilon_{01}\left(\delta+ \delta_0\right)\left[(1+t)^{-2}\left(e^{-|\vec{x}|^2/(K_* t)} +\log(2+t) e^{-|\vec{x}|}\right)+e^{-(|\vec{x}|+t)/K_*}\right].
	\end{aligned}
\end{equation}
Thus one completes the proof of \eqref{hpucbound}.
\end{proof}

Next one could combine \eqref{sri}--\eqref{sc2} and Lemma \ref{lehpuc} to obtain that
\begin{equation}\nonumber %\label{sstaricon}
	\begin{aligned}
		\left|\vec{S}_{*}(\vec{x},t)\right|&\le C\epsilon_{01}\left(\delta+ \delta_0\right) \int_0^t \left(\int_{|\vec{y}| \le |\vec{x}|/2} + \int_{|\vec{y}| > |\vec{x}|/2} \right)  e^{-(|\vec{x}-\vec{y}|+t-\tau)/D}  \left(q_2(\vec{y},\tau)+(1+\tau)^{-2}m_1(\tau) e^{-|\vec{y}|}\right) d\vec{y} d\tau\\
		&\le C\epsilon_{01}\left(\delta+ \delta_0\right)(1+t)^{-2}m_1(t) \left( e^{-\frac{|\vec{x}|}{4D}} + e^{-\frac{|\vec{x}|^2}{4K_*t}}+e^{-\frac{|\vec{x}|}{2K_*}} + e^{-\frac{|\vec{x}|}{2}} \right),
	\end{aligned}
\end{equation}
\begin{equation} \nonumber 
	\begin{aligned}
		\left|\vec{S}_{r,i}(\vec{x},t)\right| &\le C\epsilon_{01}\left(\delta+ \delta_0\right) \int_0^t \left(\int_{|\vec{y}| \le |\vec{x}|/2} + \int_{|\vec{y}| > |\vec{x}|/2} \right) (1+t-\tau)^{-3/2-d_i} e^{-\frac{|\vec{x}-\vec{y}|^2}{D(t-\tau)}}  \left[(1+\tau)^{-\frac{5}{4}}m_1(\tau)  q_1(\vec{y},\tau)\right.\\
		&\left.\quad \quad \quad \quad \quad \quad \quad \quad +(1+t-\tau)^{-1/2} \left( q_2(\vec{y},\tau)+(1+\tau)^{-2}m_1(\tau) e^{-|\vec{y}|} \right) \right] d\vec{y} d\tau\\
		&\le C\epsilon_{01}\left(\delta+ \delta_0\right)\left(e^{-\frac{|\vec{x}|^2}{8Dt}} + e^{-\frac{|\vec{x}|^2}{8K_*t}} + e^{-\frac{|\vec{x}|}{4K_*}}+e^{-\frac{|\vec{x}|}{4}}\right) \Bigg[\int_0^t (1+t-\tau)^{-2-d_i} (1+\tau)^{-2} m_1(\tau) d\tau \int_{\mathbb{R}^3} e^{-|\vec{y}|/2} d\vec{y}\\
		&\quad + \int_0^{t/2} (1+t-\tau)^{-3/2-d_i} \left( (1+\tau)^{-\frac{5}{4}}m_1(\tau) \int_{\mathbb{R}^3} q_1(\vec{y}/2,\tau) d\vec{y}+ (1+t-\tau)^{-1/2} \int_{\mathbb{R}^3} q_2(\vec{y}/2,\tau)  d\vec{y} \right) d\tau\\
		&\quad +\int_{t/2}^t  \left( \int_{\mathbb{R}^3} (1+t-\tau)^{-3/2-d_i} e^{-\frac{|\vec{x}-\vec{y}|^2}{2D(t-\tau)}}  d\vec{y} \right)  \left((1+\tau)^{-\frac{11}{4}}m_1(\tau) + (1+t-\tau)^{-1/2}(1+\tau)^{-2}\right) d\tau \Bigg]\\
		& \le C\epsilon_{01}\left(\delta+ \delta_0\right) \left[(1+t)^{-2}m_1(t)+(1+t)^{-3/2-d_i}\right]\left(e^{-\frac{|\vec{x}|^2}{8Dt}} + e^{-\frac{|\vec{x}|^2}{8K_*t}} + e^{-\frac{|\vec{x}|}{4K_*}}+e^{-\frac{|\vec{x}|}{4}}\right),
	\end{aligned}
\end{equation}
and
\begin{equation}\nonumber 
	\begin{aligned}
		\left|\vec{S}_{c,i}(\vec{x},t)\right|
		& \le C\epsilon_{01}\left(\delta+ \delta_0\right)\Bigg[e^{-at} \underset{t}{\ast} \phi(\vec{x},t) + \left( \int_{0 \le t-\tau \le |\vec{x}|/(2c_0)}+\int_{ |\vec{x}|/(2c_0) < t-\tau \le t}\right) e^{-a(t-\tau)/4} \int_{|\vec{y}|=1} \phi(\vec{x}+c_0 (t-\tau)\vec{y},\tau) \, dS(\vec{y})d\tau \\
		& \quad \quad \quad + \int_0^t \left(\int_{|\vec{y}| \le |\vec{x}|/2} + \int_{|\vec{y}| > |\vec{x}|/2} \right)  \left(\sum_{i=0}^1 |\vec{x}-\vec{y}|^{-i}\right)e^{-a(t-\tau)/8}\left( e^{-|\vec{x}-\vec{y}|/4}+e^{-a|\vec{x}-\vec{y}|/(16c_0)}\right) \phi(\vec{y},\tau) d\vec{y} d\tau\Bigg] \\
		 &\le C\epsilon_{01}\left(\delta+ \delta_0\right)\left[(1+t)^{-2} e^{-\frac{|\vec{x}|^2}{K_* t}}  + (1+t)^{-2}m_1(t) \left(e^{-|\vec{x}|/K_*}+e^{-|\vec{x}|} \right) + \left(e^{-\frac{|\vec{x}|^2}{4K_*t}}+e^{-\frac{|\vec{x}|}{2K_*}}+e^{-\frac{|\vec{x}|}{16}}+ e^{-\frac{a|\vec{x}|}{64c_0}}\right)\right.\\
		 & \left.\quad \quad \quad \times \int_0^t e^{-a(t-\tau)/8} (1+\tau)^{-2} m_1(\tau) d\tau  \left( \int_{|\vec{y}|=1} 1 dS(\vec{y})+\int_{\mathbb{R}^3} \sum_{i=0}^1 |\vec{x}-\vec{y}|^{-i} \left( e^{-|\vec{x}-\vec{y}|/8}+e^{-a|\vec{x}-\vec{y}|/(32c_0)}\right) d\vec{y} \right) \right]\\
	  	&\le C\epsilon_{01}\left(\delta+ \delta_0\right) (1+t)^{-2}m_1(t) \left(e^{-\frac{|\vec{x}|^2}{4K_*t}}+e^{-\frac{|\vec{x}|}{2K_*}}+e^{-\frac{|\vec{x}|}{16}}+ e^{-\frac{a|\vec{x}|}{64c_0}}\right), 
	\end{aligned}
\end{equation}
where one denotes 
$$q_1(\vec x,t):=(1+t)^{-3/2} e^{-|\vec{x}|^2/(K_* t)} + e^{-(|\vec{x}|+t)/K_*}, \ \ \ \ q_2(\vec{x},t):=(1+t)^{-2} e^{-|\vec{x}|^2/(K_* t)} + e^{-(|\vec{x}|+t)/K_*},$$ 
$$\phi(\vec{x},t):=q_2(\vec{x},t)+(1+t)^{-2}m_1(t) e^{-|\vec{x}|},\ \ \ \ m_1(t):=\log(2+t).$$ Then it follows that
\begin{equation}\label{sourcecon}
	\begin{aligned}
		&\left| \int_0^t \mathbb{F}_i^0(\vec{x},t-\tau) \underset{\vec{x}}{\ast} 
		\begin{pmatrix}
			\tilde{H}_p \vspace{1.5mm}\\
			H_{\vec{u}}
		\end{pmatrix}(\vec{x},\tau) \, d\tau\right| \le \left|\vec{S}_{r,i}(\vec{x},t)\right| + \left|\vec{S}_{*,i}(\vec{x},t)\right| + \left|\vec{S}_{c,i}(\vec{x},t)\right|\\
		&\le C\epsilon_{01}\left(\delta+ \delta_0\right) \left[(1+t)^{-2}m_1(t)+(1+t)^{-3/2-d_i}\right]\left(e^{-\frac{|\vec{x}|^2}{K_1t}} +  e^{-\frac{|\vec{x}|}{K_2}}\right),
	\end{aligned}
\end{equation}
for some positive constants $K_1$ and $K_2$. 

Finally, combining \eqref{jdefine}, \eqref{iniconpuc} and \eqref{sourcecon}, it holds that
\begin{equation}\nonumber
	\begin{aligned}
		\left| \vec{J}_i(\vec{x},t)\right| &\le \left|\mathbb{F}_i^0(\vec{x},t) \underset{\vec{x}}{\ast} \vec{J}(\vec{x},0)\right|+\left| \int_0^t \mathbb{F}_i^0(\vec{x},t-\tau) \underset{\vec{x}}{\ast} 
		\begin{pmatrix}
			\tilde{H}_p \vspace{1.5mm}\\
			H_{\vec{u}}
		\end{pmatrix}(\vec{x},\tau) \, d\tau\right|\\
		&\le C\epsilon_{01} \delta_0 e^{-(|\vec{x}|+t)/K_0} + C\epsilon_{01}\left(\delta+ \delta_0\right) \left[(1+t)^{-2}m_1(t)+(1+t)^{-3/2-d_i}\right]\left(e^{-\frac{|\vec{x}|^2}{K_1t}} +  e^{-\frac{|\vec{x}|^2}{K_2t}}+ e^{-\frac{|\vec{x}|+t}{2K_2}}\right)
	\end{aligned}
\end{equation}
for $i=1,2$ and $d_i$ defined in \eqref{di}. One could verify \eqref{puscpoint} by combining the above estimate with the pointwise space structure \eqref{pointsc} of the entropy function $s$, and complete the proof of Theorem~\ref{main}.

%%%%%%%%%%%%%%%%%%%%%%%%%%%%%%%%%%%%%%%%%%%%%%%%%%%%%%%%%%%%%%%%%%%%%%%%%%%%%%%%%5
%
%
%                       Appendix I
%
%
%%%%%%%%%%%%%%%%%%%%%%%%%%%%%%%%%%%%%%%%%%%%%%%%%%%%%%%%%%%%%%%%%%%%%%%%%%%%%%%%%

\setcounter{equation}{0}

\section{Appendix I: Energy Estimates for High-Order Derivatives}
\begin{proof}[\textbf{Proof of Theorem~\ref{genes}.}] 
	First, assume that the a priori assumption
	\begin{equation} \label{priori}
		\sup_{0 \le t \le T} \left\|\left(p^k, \vec{u}^k, s^k\right) \!(\cdot, t)\right\|_{H^l} \le \delta_l^k,
	\end{equation}
	holds for a sufficiently small positive constant $\delta_l^k (l\ge 3)$.
	
	For a multi-index $\alpha$ with $1\le |\alpha|\le l$, applying $\partial^{\alpha}_{\vec{x}}$ to \eqref{pk}, one has that
	\begin{equation}\label{pukt}
		\begin{cases}
			\left( \partial_{\vec{x}}^{\alpha} p^k\right)_t+\frac{\bar{P}^k\left(R+C_v\right)}{C_v} div \partial_{\vec{x}}^{\alpha} \vec u^k=-\frac{R+C_v}{C_v} p^k div \partial_{\vec{x}}^{\alpha}\vec u^k -\vec u^k\cdot \nabla \partial_{\vec{x}}^{\alpha} p^k -\frac{R+C_v}{C_v} \left[ \partial_{\vec{x}}^{\alpha}, p^k\right] div\vec u^k -\left[ \partial_{\vec{x}}^{\alpha}, \vec u^k\cdot\nabla \right] p^k=\partial^{\alpha}_{\vec x} R_p^k,\\
			\left( \partial_{\vec{x}}^{\alpha} \vec{u}^k\right)_t+\frac{1}{\bar{\rho}^k}\nabla \partial_{\vec{x}}^{\alpha} p^k+a \partial_{\vec{x}}^{\alpha} \vec{u}^k=-\vec{u}^k\cdot\nabla \partial_{\vec{x}}^{\alpha} \vec u^k +\frac{\rho^k-\bar{\rho}^k}{\bar{\rho}^k\rho^k} \nabla \partial_{\vec{x}}^{\alpha} p^k -\left[ \partial_{\vec{x}}^{\alpha}, \vec{u}^k\cdot\nabla \right]\vec u^k+\left[\partial_{\vec{x}}^{\alpha},\frac{\rho^k-\bar{\rho}^k}{\bar{\rho}^k\rho^k}\right]\nabla p^k=\partial^{\alpha}_{\vec x} R_{\vec{u}}^k,
		\end{cases}
	\end{equation}
	and
	\begin{equation}\label{skt}
		\left( \partial_{\vec{x}}^{\alpha} s^k \right)_t=-\vec u^k\cdot\nabla \partial_{\vec{x}}^{\alpha} s^k - \left[ \partial_{\vec{x}}^{\alpha}, \vec u^k\cdot\nabla\right]s^k:=\partial^{\alpha}_{\vec x} R_s^k.
	\end{equation}
	Integrate the inner product of \eqref{pukt} and $\left(\partial_{\vec{x}}^{\alpha}p^k/\bar{\rho}^k, \bar{P}^k\left(R+C_v\right)\partial_{\vec{x}}^{\alpha}\vec{u}^k/C_v\right)^t$ over $\mathbb{R}^3$ with respect to $\vec{x}$ and obtain that
	\begin{equation}\label{pukenes}
		\begin{aligned}
			&\frac{1}{2\bar{\rho}^k}\frac{d}{dt} \left\|\partial^{\alpha}_{\vec x}p^k\right\|_{L^2}^2+\frac{\bar{P}^k\left(R+C_v\right)}{2C_v}\frac{d}{dt} \left\|\partial^{\alpha}_{\vec x} \vec{u}^k\right\|_{L^2}^2 +\frac{a\bar{P}^k\left(R+C_v\right)}{C_v} \left\|\partial^{\alpha}_{\vec x}\vec{u}^k\right\|_{L^2}^2\\
			&=-\frac{R+C_v}{C_v \bar{\rho}^k} \int_{\mathbb{R}^3} \partial^{\alpha}_{\vec x}p^k \left( p^k div \partial_{\vec{x}}^{\alpha}\vec u^k\right)d\vec{x} + \frac{\bar{P}^k\left(R+C_v\right)}{C_v \bar{\rho}^k} \int_{\mathbb{R}^3}  \partial^{\alpha}_{\vec x}\vec{u}^k \cdot \left( \frac{\rho^k-\bar{\rho}^k}{\rho^k} \nabla \partial_{\vec{x}}^{\alpha} p^k\right) d\vec{x}\\
			&\quad + \frac{1}{2}  \int_{\mathbb{R}^3} div \vec{u}^k \left( \frac{1}{\bar{\rho}^k} \left( \partial^{\alpha}_{\vec x}p^k\right)^2 + \frac{\bar{P}^k\left(R+C_v\right)}{C_v}\left| \partial^{\alpha}_{\vec x}\vec{u}^k\right|^2\right) d\vec{x}\\
			& \quad -\frac{1}{\bar{\rho}^k} \int_{\mathbb{R}^3} \partial^{\alpha}_{\vec x}p^k \left( \frac{R+C_v}{C_v} \left[ \partial_{\vec{x}}^{\alpha}, p^k\right] div\vec u^k + \left[ \partial_{\vec{x}}^{\alpha}, \vec u^k\cdot\nabla \right] p^k\right)  d\vec{x} \\
			&\quad - \frac{\bar{P}^k\left(R+C_v\right)}{C_v} \int_{\mathbb{R}^3} \partial^{\alpha}_{\vec x}\vec{u}^k \cdot \left( \left[ \partial_{\vec{x}}^{\alpha}, \vec{u}^k\cdot\nabla \right]\vec u^k-\left[\partial_{\vec{x}}^{\alpha},\frac{\rho^k-\bar{\rho}^k}{\bar{\rho}^k\rho^k}\right]\nabla p^k\right) d\vec{x},
		\end{aligned}
	\end{equation}
	where
	\begin{equation} \nonumber
		\begin{aligned}
			-\int_{\mathbb{R}^3} \partial^{\alpha}_{\vec x}p^k \left( p^k div \partial_{\vec{x}}^{\alpha}\vec u^k\right)d\vec{x}&= \int_{\mathbb{R}^3} \left( p^k  \nabla \partial^{\alpha}_{\vec x}p^k +\partial^{\alpha}_{\vec x}p^k  \nabla p^k \right) \cdot \partial_{\vec{x}}^{\alpha}\vec u^k d\vec{x}\\
			&=-\frac{1}{2} \frac{d}{dt} \int_{\mathbb{R}^3} p^k \rho^k \left| \partial^{\alpha}_{\vec x}\vec{u}^k\right|^2 d\vec{x} + \frac{1}{2} \int_{\mathbb{R}^3} \left( \partial_t\left( p^k \rho^k\right)  + div\left( p^k \rho^k \vec{u}^k\right) \right) \left| \partial^{\alpha}_{\vec x}\vec{u}^k\right|^2 d\vec{x}\\
			&\quad -\int_{\mathbb{R}^3} p^k \partial^{\alpha}_{\vec x}\vec{u}^k \cdot \left(\left[ \partial_{\vec{x}}^{\alpha}, \rho^k \right]\left( \partial_t \vec{u}^k+a\vec{u}^k\right)+ \left[ \partial_{\vec{x}}^{\alpha}, \rho^k \vec{u}^k\cdot\nabla \right]\vec u^k + a \rho^k \partial_{\vec{x}}^\alpha \vec{u}^k \right) d\vec{x}\\
			&\quad +\int_{\mathbb{R}^3} \partial^{\alpha}_{\vec x}p^k \nabla p^k \cdot \partial_{\vec{x}}^{\alpha}\vec u^k d\vec{x},
		\end{aligned}
	\end{equation}
	\begin{equation} \nonumber
		\begin{aligned}
			\int_{\mathbb{R}^3}  \partial^{\alpha}_{\vec x}\vec{u}^k \cdot \left( \frac{\rho^k-\bar{\rho}^k}{\rho^k} \nabla \partial_{\vec{x}}^{\alpha} p^k\right) d\vec{x}
			&=-\frac{1}{2}\frac{d}{dt} \int_{\mathbb{R}^3} \left( \rho^k-\bar{\rho}^k\right)\left| \partial^{\alpha}_{\vec x}\vec{u}^k\right|^2 d\vec{x} + \frac{1}{2} \int_{\mathbb{R}^3} \left( \partial_t\left(\rho^k-\bar{\rho}^k\right)  + div\left( \left(\rho^k-\bar{\rho}^k\right) \vec{u}^k\right) \right) \left| \partial^{\alpha}_{\vec x}\vec{u}^k\right|^2 d\vec{x}\\
			&\quad -\int_{\mathbb{R}^3} \frac{\rho^k-\bar{\rho}^k}{\rho^k} \partial^{\alpha}_{\vec x}\vec{u}^k \cdot \left(\left[ \partial_{\vec{x}}^{\alpha}, \rho^k \right]\left( \partial_t \vec{u}^k+a\vec{u}^k\right) + \left[ \partial_{\vec{x}}^{\alpha}, \rho^k \vec{u}^k\cdot\nabla \right]\vec u^k + a \rho^k \partial_{\vec{x}}^\alpha \vec{u}^k \right)  d\vec{x}.
		\end{aligned}
	\end{equation}
	Here, one uses the identity
	$\nabla p^k=-\rho^k \left( \partial_t \vec{u}^k + \vec{u}^k \cdot \nabla \vec{u}^k+a\vec{u}^k\right)$ obtained from the second equation of \eqref{pk}. Summing \eqref{pukenes} over multi‑indices $\alpha$ satisfying $|\alpha|\in[m,l]$ with $m\ge 1$ and $l\ge 3$, and using Lemma \ref{derivaineq} along with the identities \eqref{rho} and \eqref{pp}, it holds that 
	\begin{equation}\label{umlenes}
		\begin{aligned}
			&\frac{1}{2\bar{\rho}^k}\frac{d}{dt} \left\|\nabla^m p^k\right\|_{H^{l-m}}^2+\frac{\bar{P}^k\left(R+C_v\right)}{2C_v}\frac{d}{dt} \left\|\nabla^m \vec{u}^k\right\|_{H^{l-m}}^2 + \sum_{|\alpha|=m}^l \frac{R+C_v}{2C_v \bar{\rho}^k} \frac{d}{dt} \int_{\mathbb{R}^3} \left(p^k \rho^k +  \bar{P}^k \left( \rho^k-\bar{\rho}^k\right)\right) \left| \partial^{\alpha}_{\vec x}\vec{u}^k\right|^2 d\vec{x}\\
			& \quad +\frac{a\bar{P}^k\left(R+C_v\right)}{C_v} \left\|\nabla^m \vec{u}^k\right\|_{H^{l-m}}^2\\
			& \le C\left[ \left\|\nabla^m p^k\right\|_{H^{l-m}}^2 \left(\left\|\nabla p^k\right\|_{L^\infty} + \left\|\nabla \vec{u}^k\right\|_{L^\infty} +\left\|\nabla \rho^k\right\|_{L^\infty} + \left\|\nabla \frac{1}{\rho^k}\right\|_{L^\infty}\right)+ \left( \left\|\nabla p^k\right\|_{L^\infty}^2 +\left\|\nabla \vec{u}^k\right\|_{L^\infty}^2\right) \left( \left\|\nabla^m \rho^k \right\|_{H^{l-m}}+\left\|\nabla^m \frac{1}{\rho^k}\right\|_{H^{l-m}} \right)\right] \\
			& \quad + C\left\|\nabla^m \vec{u}^k\right\|_{H^{l-m}}^2 \left( \left\|p^k\right\|_{W^{1,\infty}} + \left\|\nabla \vec{u}^k\right\|_{L^\infty} +\left\|\left(\rho^k-\bar{\rho}^k\right)\right\|_{W^{1,\infty}}+\left\|\nabla^m \rho^k \right\|_{H^{l-m}} + \left\|\nabla \frac{1}{\rho^k}\right\|_{L^\infty}+ \left\|\nabla^m \frac{1}{\rho^k}\right\|_{H^{l-m}} \right)\\
			&\quad +C\left\|\nabla^m \vec{u}^k\right\|_{H^{l-m}} \cdot \left\|\nabla \rho^k\right\|_{L^\infty} \cdot
			\begin{cases}
				0 & \text{for } m=1, \\
				\left( \left\|\nabla^{m-1} \frac{1}{\rho^k}\right\|_{L^2} \cdot \left\|\nabla p^k\right\|_{L^\infty} +  \left\|\nabla^{m-1} \vec{u}^k\right\|_{L^2} \cdot \left\|\nabla \vec{u}^k\right\|_{L^\infty} \right) & \text{for } m \ge 2
			\end{cases}\\
			& \le C\delta_l^k \left(\left\|\nabla^m p^k\right\|_{H^{l-m}}^2 + \left\|\nabla^m \vec{u}^k\right\|_{H^{l-m}}^2 + \left\|\nabla p^k\right\|_{L^\infty}^2 + \left\|\nabla \vec{u}^k \right\|_{L^\infty}^2 \right).
		\end{aligned}
	\end{equation}
	Similarly, one also has the following basic $L^2$ estimates:
	\begin{equation}\label{u0lenes}
		\begin{aligned}
			&\frac{1}{2\bar{\rho}^k}\frac{d}{dt} \left\|p^k\right\|_{L^2}^2+\frac{\bar{P}^k\left(R+C_v\right)}{2C_v}\frac{d}{dt} \left\|\vec{u}^k\right\|_{L^2}^2+\frac{a\bar{P}^k\left(R+C_v\right)}{C_v} \left\|\vec{u}^k\right\|_{L^2}^2 \\
			&\le C\left(\left\|p^k\right\|_{L^2} \cdot \left\|p^k\right\|_{L^\infty} \cdot \left\|div \vec{u}^k\right\|_{L^2} + \left\|\vec{u}^k \right\|_{L^2} \cdot \left\|\nabla p^k\right\|_{L^2} \cdot \left\|\left( \rho^k-\bar{\rho}^k\right)\right\|_{L^\infty} + \left\|div \vec{u}^k\right\|_{L^2} \cdot \left\|\vec{u}^k \right\|_{L^2} \cdot \left\|\vec{u}^k\right\|_{L^\infty}\right)\\
			& \le C\delta_3^k \left(\left\|\nabla p^k\right\|_{H^1}^2 + \left\|\vec{u}^k\right\|_{H^2}^2\right).
		\end{aligned}
	\end{equation}
The final inequalities in both \eqref{umlenes} and \eqref{u0lenes} follow from the a priori estimate \eqref{priori}. Furthermore, combining \eqref{umlenes} with $m=1$, \eqref{u0lenes}, and the Sobolev embedding $\left\|\nabla \left(p^k,\vec{u}^k\right) \right\|_{L^\infty} \le C\left\|\nabla^2 \left(p^k,\vec{u}^k\right) \right\|_{H^1}$ yields that
	\begin{equation}\label{uhlenes}
		\begin{aligned}
			&\frac{1}{2\bar{\rho}^k}\frac{d}{dt} \left\|p^k\right\|_{H^l}^2+\frac{\bar{P}^k\left(R+C_v\right)}{2C_v}\frac{d}{dt} \left\|\vec{u}^k\right\|_{H^l}^2 +\frac{a\bar{P}^k\left(R+C_v\right)}{2C_v} \left\| \vec{u}^k\right\|_{H^l}^2\\
			&+ \sum_{|\alpha|=1}^l \frac{R+C_v}{2C_v \bar{\rho}^k} \frac{d}{dt} \int_{\mathbb{R}^3} \left[p^k \rho^k +  \bar{P}^k \left( \rho^k-\bar{\rho}^k\right)\right] \left| \partial^{\alpha}_{\vec x}\vec{u}^k\right|^2 d\vec{x} 
			 \le C\delta_l^k \left\|\nabla p^k\right\|_{H^{l-1}}^2.
		\end{aligned}
	\end{equation}
	
	Next, integrating the inner product of the equations for $\partial_{\vec{x}}^{\alpha}\vec{u}^k$ in \eqref{pukt} and $\nabla \partial_{\vec{x}}^{\alpha} p^k$ over $\mathbb R^3$ yields that
	\begin{equation}\label{pterm}
		\begin{aligned}
			\frac{d}{dt} \int_{\mathbb{R}^3} \nabla \partial_{\vec{x}}^{\alpha} p^k \cdot \partial_{\vec{x}}^{\alpha} \vec{u}^k d\vec{x} + \frac{1}{\bar{\rho}^k} \left\| \nabla \partial_{\vec{x}}^{\alpha} p^k \right\|_{L^2}^2 &=\frac{\bar{P}^k\left(R+C_v\right)}{C_v} \left\| div\left( \partial_{\vec{x}}^{\alpha} \vec{u}^k \right) \right\|_{L^2}^2 -a \int_{\mathbb{R}^3} \nabla \partial_{\vec{x}}^{\alpha} p^k \cdot \partial_{\vec{x}}^{\alpha} \vec{u}^k d\vec{x}\\
			& \quad + \int_{\mathbb{R}^3} \nabla \partial_{\vec{x}}^{\alpha} p^k \cdot \partial_{\vec{x}}^{\alpha} R_{\vec{u}}^k \, d\vec{x} -\int_{\mathbb{R}^3} div \left(\partial_{\vec{x}}^{\alpha} \vec{u}^k\right) \partial_{\vec{x}}^{\alpha} R_p^k \, d\vec{x}.  
		\end{aligned}
	\end{equation}
	Upon summing \eqref{pterm} over all multi‑indices $\alpha$ satisfying $|\alpha|\in [m,l-1]$ with $m \ge 1$ and employing the following estimate
	$$a \left| \int_{\mathbb{R}^3} \nabla \partial_{\vec{x}}^{\alpha} p^k \cdot \partial_{\vec{x}}^{\alpha} \vec{u}^k d\vec{x}\right| \le \frac{1}{4 \bar{\rho}^k}  \left\| \nabla \partial_{\vec{x}}^{\alpha} p^k \right\|_{L^2}^2 + a^2 \bar{\rho}^k  \left\|\partial_{\vec{x}}^{\alpha} \vec{u}^k \right\|_{L^2}^2,$$
	it follows that
	\begin{equation}\label{pmlenes}
		\begin{aligned}
			&\sum_{|\alpha|=m}^{l-1} \frac{d}{dt} \int_{\mathbb{R}^3} \nabla \partial_{\vec{x}}^{\alpha} p^k \cdot \partial_{\vec{x}}^{\alpha} \vec{u}^k d\vec{x} + \frac{3}{4\bar{\rho}^k} \left\| \nabla^{m+1} p^k \right\|_{H^{l-m-1}}^2\\
			& \le C \left\| \nabla^m \vec{u}^k \right\|_{H^{l-m}}^2 \left(1+\left\|p^k\right\|_{W^{1,\infty}}+\left\|\vec{u}^k\right\|_{W^{1,\infty}}\right)\\
			& \quad + C\left\|\nabla^{m+1} p^k\right\|_{H^{l-m-1}}^2 \left(\left\|\vec{u}^k\right\|_{W^{1,\infty}}+ \left\|\left(\rho^k-\bar{\rho}^k \right)\right\|_{L^\infty} + \left\|\nabla^m \frac{1}{\rho^k}\right\|_{H^{l-m-1}} \right)\\
			& \quad + C\left\|\nabla^m p^k\right\|_{L^2}^2 \cdot \left\|\nabla \vec{u}^k \right\|_{L^\infty} + C\left\|\nabla p^k\right\|_{L^\infty}^2 \cdot \left\|\nabla^m \frac{1}{\rho^k}\right\|_{H^{l-m-1}}\\
			& \le C\left[ \left(1+\delta_3^k\right)\left\| \nabla^m \vec{u}^k \right\|_{H^{l-m}}^2 + \delta_{\max\{l-1,3\}}^k \left\| \nabla^m p^k \right\|_{H^{l-m}}^2 + \delta_{l-1}^k \left\| \nabla p^k \right\|_{L^\infty}^2\right],
		\end{aligned}
	\end{equation}
	and
	\begin{equation}\nonumber
		\begin{aligned}
			\frac{d}{dt} \int_{\mathbb{R}^3} \nabla p^k \cdot  \vec{u}^k d\vec{x} + \frac{3}{4\bar{\rho}^k} \left\| \nabla p^k \right\|_{L^2}^2 & \le C\left\|\vec{u}^k \right\|_{H^1}^2 \left(1+\left\|p^k\right\|_{L^\infty}+\left\|\vec{u}^k\right\|_{L^\infty}\right) + C\left\|\nabla p^k\right\|_{L^2}^2 \left(\left\|\vec{u}^k\right\|_{L^\infty}+\left\|\left(\rho^k-\bar{\rho}^k \right)\right\|_{L^\infty}\right)\\
			& \le C\left[ \left(1+\delta_2^k\right)\left\|\vec{u}^k \right\|_{H^1}^2 + \delta_2^k \left\| \nabla p^k \right\|_{L^2}^2\right],
		\end{aligned}
	\end{equation}
	which directly yields
	\begin{equation}\label{phlenes}
		\sum_{|\alpha|=0}^{l-1} \frac{d}{dt} \int_{\mathbb{R}^3} \nabla \partial_{\vec{x}}^{\alpha} p^k \cdot \partial_{\vec{x}}^{\alpha} \vec{u}^k d\vec{x} + \frac{1}{2\bar{\rho}^k} \left\| \nabla p^k \right\|_{H^{l-1}}^2 \le C\left(1+\delta_3^k\right)\left\|\vec{u}^k \right\|_{H^l}^2.
	\end{equation}
	There exists a sufficiently small positive constant $\varepsilon_0$ such that the sum \eqref{uhlenes} $+\varepsilon_0 \eqref{phlenes}$ leads to
	\begin{equation}\label{puhlenes}
		\begin{aligned}
			&\frac{1}{2\bar{\rho}^k}\frac{d}{dt} \left\|p^k\right\|_{H^l}^2+\frac{\bar{P}^k\left(R+C_v\right)}{2C_v}\frac{d}{dt} \left\|\vec{u}^k\right\|_{H^l}^2 + \sum_{|\alpha|=1}^l \frac{R+C_v}{2C_v \bar{\rho}^k} \frac{d}{dt} \int_{\mathbb{R}^3} \left(p^k \rho^k +  \bar{P}^k \left( \rho^k-\bar{\rho}^k\right)\right) \left| \partial^{\alpha}_{\vec x}\vec{u}^k\right|^2 d\vec{x}\\
			& \quad + \varepsilon_0 \sum_{|\alpha|=0}^{l-1} \frac{d}{dt} \int_{\mathbb{R}^3} \nabla \partial_{\vec{x}}^{\alpha} p^k \cdot \partial_{\vec{x}}^{\alpha} \vec{u}^k d\vec{x} + \frac{a\bar{P}^k\left(R+C_v\right)}{4C_v} \left\| \vec{u}^k\right\|_{H^l}^2+\frac{\varepsilon_0}{4\bar{\rho}^k} \left\| \nabla p^k \right\|_{H^{l-1}}^2 \le 0.
		\end{aligned}
	\end{equation}
	
	Now turn to the entropy function given by \eqref{skt}: Multiplying \eqref{skt} by $\partial_{\vec{x}}^\alpha s^k$ and integrating the product over $\mathbb{R}^3$, one has that
	\begin{equation}\nonumber
		\begin{aligned}
			\frac{1}{2}\frac{d}{dt} \left\|\partial^{\alpha}_{\vec x}s^k\right\|_{L^2}^2 &=\frac{1}{2} \int_{\mathbb{R}^3} div \vec{u}^k \left( \partial^{\alpha}_{\vec x}s^k\right)^2 d\vec{x} -\int_{\mathbb{R}^3} \partial^{\alpha}_{\vec x}s^k \left[ \partial_{\vec{x}}^{\alpha}, \vec u^k\cdot\nabla\right]s^k d\vec{x}\\
			& \le C\left\|\nabla \vec{u}^k\right\|_{L^\infty} \cdot \left\|\nabla^{|\alpha|}s^k\right\|_{L^2}^2 +
			\begin{cases}
				0 & \text{for } |\alpha|=0,\\
				C\left\|\partial^{\alpha}_{\vec x} s^k\right\|_{L^2} \cdot \left\|\nabla s^k\right\|_{L^\infty} \cdot \left\|\nabla^{|\alpha|} \vec{u}^k\right\|_{L^2} & \text{for } |\alpha| \ge 1,
			\end{cases}
		\end{aligned}
	\end{equation}
	which yields that
	\begin{equation}\nonumber
		\frac{1}{2}\frac{d}{dt} \left\|s^k\right\|_{H^l}^2 \le C \left\|\nabla \vec{u}^k\right\|_{H^{l-1}} \cdot \left\|s^k\right\|_{H^l}^2,
	\end{equation}
	and thus
	\begin{equation} \label{skhl}
		\left\|s^k(\cdot,t)\right\|_{H^l}^2 \le  \exp\left\lbrace C\int_0^t \left\|\nabla \vec{u}^k(\cdot,\tau)\right\|_{H^{l-1}} d\tau\right\rbrace  \left\|s^k_0\right\|_{H^l}^2.
	\end{equation}
	Here, one uses Gronwall's inequality. To achieve the boundedness of the entropy function $s^k$ from \eqref{skhl}, a decay estimate for $\bigl\|\nabla \vec{u}^k(\cdot, t)\bigr\|_{H^{l-1}}$ is required to guarantee its integrability with respect to $t$. 
	
	Define the following time-dependent functional
	$$ \mathcal{M}_1^k(t):= \underset{0 \le \tau \le t}{\sup} (1+\tau)^{5/2} \left( \left\|\vec{u}^k(\cdot, \tau)\right\|_{H^l}^2 + \left\|\nabla p^k (\cdot, \tau)\right\|_{H^{l-1}}^2 \right).$$ 
	Considering \eqref{umlenes} $+\varepsilon_1 \eqref{pmlenes}$, with $\varepsilon_1\in(0,1)$ taken sufficiently small, it holds that
	\begin{equation}\nonumber
		\begin{aligned}
			&\frac{1}{2\bar{\rho}^k}\frac{d}{dt} \left\|\nabla^m p^k\right\|_{H^{l-m}}^2+\frac{\bar{P}^k\left(R+C_v\right)}{2C_v}\frac{d}{dt} \left\|\nabla^m \vec{u}^k\right\|_{H^{l-m}}^2  + \sum_{|\alpha|=m}^l \frac{R+C_v}{2C_v \bar{\rho}^k} \frac{d}{dt} \int_{\mathbb{R}^3} \left(p^k \rho^k +  \bar{P}^k \left( \rho^k-\bar{\rho}^k\right)\right) \left| \partial^{\alpha}_{\vec x}\vec{u}^k\right|^2 d\vec{x}\\
			&\quad+\varepsilon_1 \sum_{|\alpha|=m}^{l-1} \frac{d}{dt} \int_{\mathbb{R}^3} \nabla \partial_{\vec{x}}^{\alpha} p^k \cdot \partial_{\vec{x}}^{\alpha} \vec{u}^k d\vec{x}+\frac{a\bar{P}^k\left(R+C_v\right)}{C_v} \left\|\nabla^m \vec{u}^k\right\|_{H^{l-m}}^2+ \frac{3\varepsilon_1}{4\bar{\rho}^k} \left\| \nabla^{m+1} p^k \right\|_{H^{l-m-1}}^2\\
			&\le C\left(\delta_l^k+\varepsilon_1\right)\left\| \nabla^m \vec{u}^k \right\|_{H^{l-m}}^2 + C\delta_l^k \left(\left\|\nabla^m p^k\right\|_{H^{l-m}}^2 + \left\|\nabla p^k\right\|_{L^\infty}^2 + \left\|\nabla \vec{u}^k \right\|_{L^\infty}^2 \right),
		\end{aligned}
	\end{equation}
	which, combined with the smallness of $\delta_l^k$ and $\varepsilon_1$ as well as the Sobolev embedding inequality $\bigl\|\nabla \bigl(p^k,\vec{u}^k\bigr) \bigr\|_{L^\infty} \le C\bigl\|\nabla^2 \bigl(p^k,\vec{u}^k\bigr) \bigr\|_{H^1}$, yields
	\begin{equation}\label{pumenes}
		\begin{aligned}
			&\frac{1}{2\bar{\rho}^k}\frac{d}{dt} \left\|\nabla^m p^k\right\|_{H^{l-m}}^2+\frac{\bar{P}^k\left(R+C_v\right)}{2C_v}\frac{d}{dt} \left\|\nabla^m \vec{u}^k\right\|_{H^{l-m}}^2+ \sum_{|\alpha|=m}^l \frac{R+C_v}{2C_v \bar{\rho}^k} \frac{d}{dt} \int_{\mathbb{R}^3} \left(p^k \rho^k +  \bar{P}^k \left( \rho^k-\bar{\rho}^k\right)\right) \left| \partial^{\alpha}_{\vec x}\vec{u}^k\right|^2 d\vec{x}\\
			&\quad +\varepsilon_1 \sum_{|\alpha|=m}^{l-1} \frac{d}{dt} \int_{\mathbb{R}^3} \nabla \partial_{\vec{x}}^{\alpha} p^k \cdot \partial_{\vec{x}}^{\alpha} \vec{u}^k d\vec{x}+\frac{a\bar{P}^k\left(R+C_v\right)}{2C_v} \left\|\nabla^m \vec{u}^k\right\|_{H^{l-m}}^2+ \frac{\varepsilon_1}{2\bar{\rho}^k} \left\| \nabla^{m+1} p^k \right\|_{H^{l-m-1}}^2\\
			&\le C\delta_l^k \left\| \nabla^m p^k \right\|_{L^2}^2 + C\delta_l^k
			\begin{cases}
				0 & \text{for } m=1,2,\\
				\left( \left\|\nabla p^k\right\|_{L^\infty}^2 + \left\|\nabla \vec{u}^k \right\|_{L^\infty}^2 \right) & \text{for } m \ge 3.
			\end{cases} 
		\end{aligned}
	\end{equation}
	
	Subsequently, for $|\alpha|\ge 1$, a standard energy estimate for $\partial_{\vec{x}}^{\alpha}\vec{u}^k$ based on the equation for $\partial_{\vec{x}}^{\alpha}\vec{u}^k$ in \eqref{pukt} yields that
	\begin{equation} \nonumber
		\begin{aligned}
			\frac{1}{2} \frac{d}{dt}\left\|\partial^{\alpha}_{\vec x} \vec{u}^k \right\|_{L^2}^2 + a\left\|\partial^{\alpha}_{\vec x} \vec{u}^k \right\|_{L^2}^2 &=\int_{\mathbb{R}^3} \partial^{\alpha}_{\vec x} \vec{u}^k \cdot \left[-\frac{1}{\bar{\rho}^k} \nabla \partial^{\alpha}_{\vec x} p^k - \partial^{\alpha}_{\vec x}\left( \vec{u}^k\cdot \nabla \vec u^k\right)  +\partial_{\vec{x}}^{\alpha} \left( \frac{\rho^k-\bar{\rho}^k}{\bar{\rho}^k\rho^k} \nabla p^k\right) \right] d\vec{x}\\
			& \le \frac{a}{4} \left\|\partial^{\alpha}_{\vec x} \vec{u}^k \right\|_{L^2}^2 + C\left\|\nabla^{|\alpha|} \vec{u}^k\right\|_{L^2}^2 \left( \left\|\vec{u}^k \right\|_{W^{1,\infty}}+\left\|\left(\rho^k-\bar{\rho}^k\right) \right\|_{L^\infty} + {{\left\|\nabla^{|\alpha|} \frac{1}{\rho^k}\right\|_{L^2}}} \right)\\
			&\quad + C \left[\left\|\nabla^{|\alpha|+1} \vec{u}^k\right\|_{L^2}^2 \cdot \left\|\vec{u}^k\right\|_{L^\infty}+\left\|\nabla^{|\alpha|+1} p^k\right\|_{L^2}^2 \left( 1+\left\|\left(\rho^k-\bar{\rho}^k\right) \right\|_{L^\infty}\right) +  \left\|\nabla p^k \right\|_{L^\infty}^2 \cdot  \left\|\nabla^{|\alpha|} \frac{1}{\rho^k}\right\|_{L^2}\right],
		\end{aligned}
	\end{equation}
and thus
	\begin{equation} \label{umlow}
		\frac{1}{2} \frac{d}{dt} \left\|\nabla^{m-1} \vec{u}^k\right\|_{L^2}^2 + \frac{a}{2} \left\|\nabla^{m-1} \vec{u}^k\right\|_{L^2}^2 \le C\left( \delta_2^k \left\|\nabla^m \vec{u}^k\right\|_{L^2}^2 + \left\|\nabla^m p^k\right\|_{L^2}^2 + \delta_{m-1}^k \left\|\nabla p^k\right\|_{L^\infty}^2\right),\ \ m\ge 2,
	\end{equation}
	\begin{equation} \label{u0low}
		\begin{aligned}
			\frac{1}{2} \frac{d}{dt} \left\|\vec{u}^k\right\|_{L^2}^2 + \frac{a}{2} \left\|\vec{u}^k\right\|_{L^2}^2 &\le C\left\|\nabla p^k\right\|_{L^2}^2 \left( 1+ \left\|\left(\rho^k -\bar{\rho}^k \right)\right\|_{L^\infty}\right) \le C\left\|\nabla p^k\right\|_{L^2}^2.
		\end{aligned}
	\end{equation}
	Adding $\|\nabla^m p^k \|_{L^2}^2$ to both sides of \eqref{pumenes}, and coupling the resulting estimate with $\bar{P}^k(R+C_v)/C_v \cdot \eqref{umlow}$, $\bar{P}^k(R+C_v)/C_v \cdot \eqref{u0low}$ together with the Sobolev embedding inequaliity, it holds that
	\begin{equation}\label{pugron}
		\begin{aligned}
			&\frac{1}{2\bar{\rho}^k}\frac{d}{dt} \left\|\nabla^m p^k\right\|_{H^{l-m}}^2+\frac{\bar{P}^k\left(R+C_v\right)}{2C_v}\frac{d}{dt} \left\|\nabla^{m-1} \vec{u}^k\right\|_{H^{l-m+1}}^2 + \sum_{|\alpha|=m}^l \frac{R+C_v}{2C_v \bar{\rho}^k} \frac{d}{dt} \int_{\mathbb{R}^3} \left(p^k \rho^k +  \bar{P}^k \left( \rho^k-\bar{\rho}^k\right)\right) \left| \partial^{\alpha}_{\vec x}\vec{u}^k\right|^2 d\vec{x}\\
			&\quad +\varepsilon_1 \sum_{|\alpha|=m}^{l-1} \frac{d}{dt} \int_{\mathbb{R}^3} \nabla \partial_{\vec{x}}^{\alpha} p^k \cdot \partial_{\vec{x}}^{\alpha} \vec{u}^k d\vec{x} +\frac{a\bar{P}^k\left(R+C_v\right)}{8C_v} \left\|\nabla^{m-1} \vec{u}^k\right\|_{H^{l-m+1}}^2+ \frac{\varepsilon_1}{4\bar{\rho}^k} \left\| \nabla^m p^k \right\|_{H^{l-m}}^2\\
			&\le C\left\| \nabla^m p^k \right\|_{L^2}^2 + C\delta_l^k
			\begin{cases}
				0 & \text{for } m=1,2,\\
				\left\|\nabla p^k\right\|_{L^\infty}^2  & \text{for } m = 3,\\
				\left( \left\|\nabla p^k\right\|_{L^\infty}^2 + \left\|\nabla \vec{u}^k \right\|_{L^\infty}^2 \right) & \text{for } m \ge 4.
			\end{cases} 
		\end{aligned}
	\end{equation}
	
	Moreover, from Theorem~\ref{greenlp} and Duhamel's principle, one has that
	\begin{equation}\nonumber
		\begin{aligned}
			\left\| \nabla p^k \right\|_{L^2} &\le C(1+t)^{-\frac{5}{4}}\left(\left\|p_0^k\right\|_{L^1 \cap H^1} + (1+t)^{-\frac{1}{2}} \left\|\vec{u}_0^k\right\|_{L^1 \cap H^1} \right)\\
			& \quad + C\int_0^t (1+t-\tau)^{-\frac{5}{4}} \left(\left\|R_p^k(\cdot,\tau)\right\|_{L^1 \cap H^1} + (1+t-\tau)^{-\frac{1}{2}} \left\|R_{\vec{u}}^k(\cdot,\tau)\right\|_{L^1 \cap H^1} \right) d\tau,
		\end{aligned}
	\end{equation}
	which together with the definitions of $R_p^k$ and $R_{\vec{u}}^k$ given in \eqref{pk} yields that
	\begin{equation}\label{rpul1}
		\begin{aligned}
			\left\|\left(R_p^k,R_{\vec{u}}^k\right) (\cdot,\tau) \right\|_{L^1} &\le C\left[ \left\|\left(p^k,\vec{u}^k\right) \!(\cdot,\tau) \right\|_{L^2} \cdot \left\|\nabla \vec{u}^k(\cdot,\tau) \right\|_{L^2} + \left( \left\|\vec{u}^k(\cdot,\tau) \right\|_{L^2} + \left\|\left( \rho^k-\bar{\rho}^k\right) \!(\cdot,\tau) \right\|_{L^2}\right) \left\|\nabla p^k(\cdot,\tau) \right\|_{L^2} \right]\\
			& \le C \delta_1^k \left( \left\|\nabla \vec{u}^k(\cdot,\tau) \right\|_{L^2}+\left\|\nabla p^k(\cdot,\tau) \right\|_{L^2} \right),
		\end{aligned} 
	\end{equation}
	\begin{equation}\label{rpul2}
		\begin{aligned}
			\left\|\left(R_p^k,R_{\vec{u}}^k\right) \! (\cdot,\tau) \right\|_{L^2} &\le C\left[ \left\|\left(p^k,\vec{u}^k\right) \!(\cdot,\tau) \right\|_{L^\infty} \cdot \left\|\nabla \vec{u}^k(\cdot,\tau) \right\|_{L^2} + \left( \left\|\vec{u}^k(\cdot,\tau) \right\|_{L^\infty} + \left\|\left( \rho^k-\bar{\rho}^k\right) \!(\cdot,\tau) \right\|_{L^\infty}\right) \left\|\nabla p^k(\cdot,\tau) \right\|_{L^2} \right]\\
			& \le C \delta_2^k \left( \left\|\nabla \vec{u}^k(\cdot,\tau) \right\|_{L^2}+\left\|\nabla p^k(\cdot,\tau) \right\|_{L^2} \right),
		\end{aligned} 
	\end{equation}
	and
	\begin{equation}\nonumber
		\begin{aligned}
			\left\|\nabla \left(R_p^k,R_{\vec{u}}^k\right) \! (\cdot,\tau) \right\|_{L^2} &\le C\left[ \left\|\left(p^k,\vec{u}^k\right) \!(\cdot,\tau) \right\|_{W^{1,\infty}}\cdot \left\|\nabla \vec{u}^k(\cdot,\tau) \right\|_{H^1} + \left( \left\|\vec{u}^k(\cdot,\tau) \right\|_{W^{1,\infty}} + \left\|\left( \rho^k-\bar{\rho}^k\right) \!(\cdot,\tau) \right\|_{W^{1,\infty}}\right) \left\|\nabla p^k(\cdot,\tau) \right\|_{H^1} \right]\\
			& \le C \delta_3^k \left( \left\|\nabla \vec{u}^k(\cdot,\tau) \right\|_{H^1}+\left\|\nabla p^k(\cdot,\tau) \right\|_{H^1} \right),
		\end{aligned} 
	\end{equation}
	and thus one obtains that
	\begin{equation}\label{nablapk}
		\begin{aligned}
			\left\| \nabla p^k \right\|_{L^2} &\le C(1+t)^{-\frac{5}{4}}\left(\left\|p_0^k\right\|_{L^1 \cap H^1} + (1+t)^{-\frac{1}{2}} \left\|\vec{u}_0^k\right\|_{L^1 \cap H^1} \right) + C\delta_3^k \int_0^t (1+t-\tau)^{-\frac{5}{4}} (1+\tau)^{-\frac{5}{4}} \sqrt{\mathcal{M}_1^k(\tau)} \, d\tau\\
			& \le C(1+t)^{-\frac{5}{4}} \left(\left\|p_0^k\right\|_{L^1 \cap H^1} + (1+t)^{-\frac{1}{2}} \left\|\vec{u}_0^k\right\|_{L^1 \cap H^1} + \delta_3^k \sqrt{\mathcal{M}_1^k(t)} \right).
		\end{aligned}
	\end{equation}
	Substituting \eqref{nablapk} into \eqref{pugron} for $m=1$ and applying Gronwall's inequality yields
	\begin{equation} \nonumber 
		\begin{aligned}
			\left\|\vec{u}^k(\cdot, t)\right\|_{H^l}^2  + \left\|\nabla p^k(\cdot, t)\right\|_{H^{l-1}}^2 & \le C e^{- t/C} \left(\left\|\vec{u}_0^k \right\|_{H^l}^2  + \left\|\nabla p_0^k\right\|_{H^{l-1}}^2 \right)\\
			& \quad + C (1+t)^{-\frac{5}{2}} \left(\left\|p_0^k\right\|_{L^1 \cap H^1}^2 +(1+t)^{-1} \left\|\vec{u}_0^k\right\|_{L^1 \cap H^1}^2 + \left( \delta_3^k \right)^2 \mathcal{M}_1^k(t)\right),
		\end{aligned}
	\end{equation}
	and hence
	\begin{equation} \nonumber 
		\mathcal{M}_1^k(t) \le  C \left(\left\|\left( p_0^k,\vec{u}_0^k\right) \right\|_{L^1 \cap H^l}^2 + \left( \delta_3^k \right)^2 \mathcal{M}_1^k(t)\right),
	\end{equation}
	which implies
	\begin{equation} \nonumber
		\mathcal{M}_1^k(t) \le C\left\|\left( p_0^k,\vec{u}_0^k\right) \right\|_{L^1 \cap H^l}^2,
	\end{equation}
	and verifies \eqref{pude}.
	
	Then, combining \eqref{pude} and \eqref{skhl}, one obtains the following estimate for $\|s^k\|_{H^l}^2$:
	\begin{equation} \nonumber
		\left\|s^k(\cdot,t)\right\|_{H^l}^2 \le  \exp\left\lbrace C \left\| \left( p_0^k, \vec{u}_0^k\right)\right\|_{L^1 \cap H^l} \right\rbrace  \left\|s^k_0\right\|_{H^l}^2,
	\end{equation}
	which together with \eqref{puhlenes} verifies the ansatz assumption \eqref{priori} and also \eqref{genenes}, and finishes the proof.
\end{proof}

\begin{proof}[\textbf{Proof of Lemma~\ref{decay}.}] The proofs of \eqref{decaygen2} and \eqref{decaygen3} are similar to the one of \eqref{pude}:
	Denote
	$$ \mathcal{M}_i^k(t):= \underset{0 \le \tau \le t}{\sup} (1+\tau)^{2m_i} \left( \left\|\nabla^{i-1} \vec{u}^k (\cdot, \tau)\right\|_{H^{l-i+1}}^2 + \left\|\nabla^i p^k (\cdot, \tau)\right\|_{H^{l-i}}^2 \right),$$
	with $m_2=7/4$ and $m_3=2$. From \eqref{pugron} and 
	\begin{equation}\label{nablampk}
		\begin{aligned}
			\left\| \nabla^i p^k \right\|_{L^2} &\le C(1+t)^{-\frac{3}{4}-\frac{i}{2}}\left(\left\|p_0^k\right\|_{L^1 \cap H^i} + (1+t)^{-\frac{1}{2}} \left\|\vec{u}_0^k\right\|_{L^1 \cap H^i} \right)\\
			& \quad + C\int_0^{t/2} (1+t-\tau)^{-\frac{3}{4}-\frac{i}{2}} \left(\left\|R_p^k(\cdot,\tau)\right\|_{L^1} + (1+t-\tau)^{-\frac{1}{2}} \left\|R_{\vec{u}}^k(\cdot,\tau)\right\|_{L^1} \right) d\tau\\
			&\quad + C\int_{t/2}^t (1+t-\tau)^{-\frac{i}{2}} \left(\left\|R_p^k(\cdot,\tau)\right\|_{L^2} + (1+t-\tau)^{-\frac{1}{2}} \left\|R_{\vec{u}}^k(\cdot,\tau)\right\|_{L^2} \right) d\tau\\
			&\quad + C\int_0^t e^{-R_0(t-\tau)} \left( \left\|\nabla^i R_p^k(\cdot,\tau)\right\|_{L^2} + \left\|\nabla^i R_{\vec{u}}^k(\cdot,\tau)\right\|_{L^2}\right) d\tau,\ \ \ \ i=2,3,
		\end{aligned} 
	\end{equation}
	 one first derives estimates for $\bigl\|\bigl( R_p^k, R_{\vec{u}}^k \bigr)(\cdot,\tau)\bigr\|_{L^1 \cap H^i}$. Combining \eqref{genenes}, \eqref{pude}, and the initial condition \eqref{initial}, $\delta_l^k$ in \eqref{rpul1} and \eqref{rpul2} taken from the a priori estimate \eqref{priori} in the proof of global existence, can be replaced by $\epsilon_k$; and one obtains that
	\begin{equation}\label{rpuil1}
		\left\|\left(R_p^k,R_{\vec{u}}^k\right) \! (\cdot,\tau) \right\|_{L^1} \le C \epsilon_k \left( \left\|\nabla \vec{u}^k(\cdot,\tau) \right\|_{L^2}+\left\|\nabla p^k(\cdot,\tau) \right\|_{L^2} \right) \le C\epsilon_k(1+\tau)^{-\frac{5}{4}} \left\|\left(p_0^k, \vec{u}_0^k \right) \right\|_{L^1 \cap H^3},
	\end{equation}
	\begin{equation}\label{rpuil2}
		\begin{aligned}
			\left\|R_p^k (\cdot,\tau) \right\|_{L^2} &\le C\left\|\nabla p^k(\cdot,\tau) \right\|_{H^1} \cdot \left\|\nabla \vec{u}^k(\cdot,\tau) \right\|_{H^1}\le C(1+\tau)^{-\frac{5}{2}} \left\|\left(p_0^k, \vec{u}_0^k \right) \right\|_{L^1 \cap H^3}^2,\\
			\left\|R_{\vec{u}}^k(\cdot,\tau) \right\|_{L^2} &\le C(1+\tau)^{-\frac{5}{2}} \left\|\left(p_0^k, \vec{u}_0^k \right) \right\|_{L^1 \cap H^3}^2 + C\epsilon_k \left\|\nabla p^k(\cdot,\tau) \right\|_{L^\infty},
		\end{aligned} 
	\end{equation}
	\begin{equation}\label{rpuhi}
		\begin{aligned}
			&\left\|R_p^k (\cdot,\tau) \right\|_{H^i} \le C \left\|\left(p^k,\vec{u}^k\right) \! (\cdot,\tau) \right\|_{W^{1,\infty}} \cdot \left\|\nabla^i \left( p^k, \vec{u}^k\right)(\cdot,\tau) \right\|_{H^1} \le C\epsilon_k \left\|\nabla^i \left( p^k, \vec{u}^k\right) \!(\cdot,\tau) \right\|_{H^1} \le C\epsilon_k (1+\tau)^{-m_i} \sqrt{\mathcal{M}_i^k(\tau)},\\
			&\left\|R_{\vec{u}}^k(\cdot,\tau) \right\|_{H^i} \le C\Bigg( \left\|\vec{u}^k(\cdot,\tau) \right\|_{W^{1,\infty}} \cdot \left\|\nabla^i \vec{u}^k(\cdot,\tau) \right\|_{H^1} +  \left\|\left( \rho^k-\bar{\rho}^k\right) \!(\cdot,\tau) \right\|_{L^\infty} \cdot \left\|\nabla^{i+1} p^k(\cdot,\tau) \right\|_{L^2}\\
			&\quad \quad \quad \quad \quad \quad \quad + \left\|\nabla p^k(\cdot,\tau) \right\|_{L^\infty} \cdot \left\|\nabla^i \left(\frac{\rho^k-\bar{\rho}^k}{\bar{\rho}^k \rho^k} \right) (\cdot,\tau) \right\|_{L^2} \Bigg)\\
			&\quad \quad \quad \quad \quad \le C\epsilon_k \left(  (1+\tau)^{-m_i}  \sqrt{\mathcal{M}_i^k(\tau)} + \left\|\nabla p^k(\cdot,\tau) \right\|_{L^\infty}\right).
		\end{aligned} 
	\end{equation}
	
One takes $m=2$ in \eqref{nablampk}, together with \eqref{rpuil1}--\eqref{rpuhi} and the estimate
\begin{equation}\nonumber
	 \left\|\nabla p^k(\cdot,\tau) \right\|_{L^\infty} \le C\left\|\nabla^2 p^k(\cdot,\tau) \right\|_{H^1} \le C(1+\tau)^{-m_2} \sqrt{\mathcal{M}_2^k(\tau)}=C(1+\tau)^{-\frac{7}{4}} \sqrt{\mathcal{M}_2^k(\tau)},
\end{equation}
and obtains
\begin{equation}\nonumber
	\begin{aligned}
		\left\| \nabla^2 p^k(\cdot,t) \right\|_{L^2} &\le C(1+t)^{-\frac{7}{4}}\left(\left\|p_0^k\right\|_{L^1 \cap H^2} + (1+t)^{-\frac{1}{2}} \left\|\vec{u}_0^k\right\|_{L^1 \cap H^2} \right)\\
		& \quad + C\epsilon_k \int_0^{t/2} (1+t-\tau)^{-\frac{7}{4}} (1+\tau)^{-\frac{5}{4}} d\tau \cdot  \left\|\left(p_0^k, \vec{u}_0^k \right) \right\|_{L^1 \cap H^3}\\
		&\quad + C\int_{t/2}^t (1+t-\tau)^{-1} \left[ (1+\tau)^{-\frac{5}{2}} \left\|\left(p_0^k, \vec{u}_0^k \right) \right\|_{L^1 \cap H^3}^2 + \epsilon_k (1+t-\tau)^{-\frac{1}{2}} (1+\tau)^{-\frac{7}{4}} \sqrt{\mathcal{M}_2^k(\tau)} \, \right] d\tau\\
		&\quad + C\epsilon_k \int_0^t e^{-R_0(t-\tau)} (1+\tau)^{-\frac{7}{4}} \sqrt{\mathcal{M}_2^k(\tau)} \, d\tau\\
		& \le C(1+t)^{-\frac{7}{4}}\left(\left\|\left(p_0^k, \vec{u}_0^k \right) \right\|_{L^1 \cap H^3}+\epsilon_k \sqrt{\mathcal{M}_2^k(t)}\right).
	\end{aligned}
\end{equation}
Combining the above estimate for $\|\nabla^2 p^k\|_{L^2}$ with \eqref{pugron} for $m=2$ yields that
\begin{equation} \nonumber 
	\begin{aligned}
		\left\|\nabla \vec{u}^k(\cdot, t)\right\|_{H^{l-1}}^2  + \left\|\nabla^2 p^k(\cdot, t)\right\|_{H^{l-2}}^2 & \le C e^{- t/C} \left(\left\|\nabla \vec{u}_0^k \right\|_{H^{l-1}}^2  + \left\|\nabla^2 p_0^k\right\|_{H^{l-2}}^2 \right) + C (1+t)^{-\frac{7}{2}} \left(\left\|\left( p_0^k,\vec{u}_0^k\right) \right\|_{L^1 \cap H^3}^2 + \left( \epsilon_k \right)^2 \mathcal{M}_2^k(t)\right).
	\end{aligned}
\end{equation}
It then follows that
\begin{equation} \nonumber 
	\mathcal{M}_2^k(t) \le  C \left(\left\|\left( p_0^k,\vec{u}_0^k\right) \right\|_{L^1 \cap H^l}^2 + \left( \epsilon^k \right)^2 \mathcal{M}_2^k(t)\right),
\end{equation}
which results in that
\begin{equation} \nonumber
	\mathcal{M}_2^k(t) \le C\left\|\left( p_0^k,\vec{u}_0^k\right) \right\|_{L^1 \cap H^l}^2,
\end{equation}
and verifies \eqref{decaygen2}.

For $i=3$, by an argument analogous to the case $i=2$, and invoking \eqref{pugron}, \eqref{nablampk}--\eqref{rpuhi} together with the ansatz \eqref{anasn} which gives $\|\nabla p^k(\cdot, t)\|_{L^\infty} \le C(1+t)^{-2} \mathcal{N}^k(t)$, one has that
	\begin{equation}\nonumber
		\begin{aligned}
			\left\| \nabla^3 p^k(\cdot,t) \right\|_{L^2} &\le C(1+t)^{-\frac{9}{4}}\left(\left\|p_0^k\right\|_{L^1 \cap H^3} + (1+t)^{-\frac{1}{2}} \left\|\vec{u}_0^k\right\|_{L^1 \cap H^3} \right)\\
			& \quad + C\epsilon_k \int_0^{t/2} (1+t-\tau)^{-\frac{9}{4}} (1+\tau)^{-\frac{5}{4}} d\tau \cdot  \left\|\left(p_0^k, \vec{u}_0^k \right) \right\|_{L^1 \cap H^3}\\
			&\quad + C\int_{t/2}^t (1+t-\tau)^{-\frac{3}{2}} \left[ (1+\tau)^{-\frac{5}{2}} \left\|\left(p_0^k, \vec{u}_0^k \right) \right\|_{L^1 \cap H^3}^2 + \epsilon_k (1+t-\tau)^{-\frac{1}{2}} (1+\tau)^{-2}\mathcal{N}^k(\tau) \, \right] d\tau\\
			&\quad + C\epsilon_k \int_0^t e^{-R_0(t-\tau)} (1+\tau)^{-2} \left(  \sqrt{\mathcal{M}_3^k(\tau)}+\mathcal{N}^k(\tau)\right)  \, d\tau\\
			& \le C(1+t)^{-\frac{9}{4}} \left\|\left(p_0^k, \vec{u}_0^k \right) \right\|_{L^1 \cap H^3}+ C\epsilon_k (1+t)^{-2} \left(  \sqrt{\mathcal{M}_3^k(t)}+\mathcal{N}^k(t)\right),
		\end{aligned}
	\end{equation}
	which combined with \eqref{pugron} for $m=3$ yields that
	\begin{equation}\nonumber
		\begin{aligned}
			\left\|\nabla^2 \vec{u}^k(\cdot, t)\right\|_{H^{l-2}}^2  + \left\|\nabla^3 p^k(\cdot, t)\right\|_{H^{l-3}}^2 & \le C (1+t)^{-\frac{9}{2}} \left\|\left( p_0^k,\vec{u}_0^k\right) \right\|_{L^1 \cap H^l}^2 + C (1+t)^{-4}\left( \left( \epsilon_k \right)^2 \mathcal{M}_3^k(t) + \left( \mathcal{N}^k(t)\right)^2\right).
		\end{aligned}
	\end{equation}
	Thus one has that
	\begin{equation} \nonumber 
		\mathcal{M}_3^k(t) \le  C \left(\left\|\left( p_0^k,\vec{u}_0^k\right) \right\|_{L^1 \cap H^l}^2 + \left( \epsilon^k \right)^2 \mathcal{M}_2^k(t) + \left( \mathcal{N}^k(t)\right)^2\right),
	\end{equation}
	and then
	\begin{equation} \nonumber
		\mathcal{M}_3^k(t) \le C\left(\left\|\left( p_0^k,\vec{u}_0^k\right) \right\|_{L^1 \cap H^l}^2 + \left( \mathcal{N}^k(t)\right)^2\right),
	\end{equation}
	which verifies \eqref{decaygen3}.
\end{proof}

%%%%%%%%%%%%%%%%%%%%%%%%%%%%%%%%%%%%%%%%%%%%%%%%%%%%%%%%%%%%%%%%%%%%%%%%%%%%%%%%%5
%
%
%                        Appendix II
%
%
%%%%%%%%%%%%%%%%%%%%%%%%%%%%%%%%%%%%%%%%%%%%%%%%%%%%%%%%%%%%%%%%%%%%%%%%%%%%%%%%%

\setcounter{equation}{0}

\section{Appendix II: Computational Lemmas and the Proofs}
In this section, we provide the proofs for all the computational lemmas used for the nonlinear wave propagations. From \eqref{fs1k}--\eqref{greenfi} and Lemma~\ref{lerpu}, the following lemmas are given to estimate $\vec{S}_{R}^k$ defined in \eqref{v}.
\begin{lemma}\label{leker}
	Let $D$ be a positive constant and $m$ be a nonnegative integer.
 	Then the following holds:
	\begin{align}
			&(1+t)^{-(3+m)/2} e^{-\frac{|\vec{x}|^2}{Dt}} \underset{(\vec{x},t)}{\ast} (1+t)^{1/2} w_3^k(\vec{x},t) \cdot w_3^k(\vec{x},t) \le C\left[ (1+t)^{-\frac{3+\min\{m,4\}}{2}} e^{-\frac{|\vec{x}|^2}{C_{0,k}t}}+(1+t)^{-\frac{3+m}{2}} e^{-(|\vec{x}|+t)/D_{0,k}}\right], \label{kercon1}\\
			&(1+t)^{-(3+m)/2} e^{-\frac{|\vec{x}|^2}{Dt}}  \underset{(\vec{x},t)}{\ast} (1+t)^{-2} e^{-|\vec{x}|} \le C(1+t)^{-\frac{3+\min\{m,1\}}{2}}\left( e^{-\frac{|\vec{x}|^2}{4Dt}} +e^{-(|\vec{x}|+t)/8}\right) \le C(1+t)^{-\frac{3+\min\{m,1\}}{2}}\left( e^{-\frac{|\vec{x}|^2}{C_{0,k}t}} +e^{-(|\vec{x}|+t)/D_{0,k}}\right), \label{kercon2}	 	
	\end{align}
	where $w_3^k(\vec{x},t)$ is defined in \eqref{w0123} and the constants $C_{0,k}$ and $D_{0,k}$ satisfy \eqref{cd0k}.
\end{lemma}

\begin{proof}
	First, for $w_3^k(\vec{x},t)$ defined in \eqref{w0123} with positive constants $C_{3,k}$ and $D_{3,k}$, where $C_{3,k}=\nu^2 C_{0,k}$ and $D_{3,k}=\nu D_{0,k}$ with $\nu=\mu^3=3\sqrt{6}/7$, one obtains that
	\begin{equation}\nonumber
		\begin{aligned}
			&\int_0^t \int_{\mathbb{R}^3} (1+t-\tau)^{-\frac{3+m}{2}} e^{-\frac{|\vec{x}-\vec{y}|^2}{D(t-\tau)}} (1+\tau)^{-\frac{7}{2}} e^{-\frac{2|\vec{y}|^2}{C_{3,k}\tau}} d\vec{y} d\tau\\
			&=\int_0^t \left(\int_{|\vec{y}| \le \nu \sqrt{6}|\vec{x}|/3} + \int_{|\vec{y}| > \nu \sqrt{6}|\vec{x}|/3}\right)  (1+t-\tau)^{-\frac{3+m}{2}} e^{-\frac{|\vec{x}-\vec{y}|^2}{D(t-\tau)}}(1+\tau)^{-\frac{7}{2}} e^{-\frac{2|\vec{y}|^2}{\nu^2 C_{0,k}\tau}} d\vec{y} d\tau,
		\end{aligned}
	\end{equation}
	which combined with the estimates $e^{-\frac{|\vec{x}-\vec{y}|^2}{2D(t-\tau)}} \le e^{-\frac{|\vec{x}|^2}{98Dt}}$ for $|\vec{y}| \le \nu \sqrt{6}|\vec{x}|/3 =6|\vec{x}|/7$, and $e^{-\frac{3|\vec{y}|^2}{2\nu^2 C_{0,k}\tau}} \le e^{-\frac{|\vec{x}|^2}{C_{0,k}t}}$	for $|\vec{y}| > \nu \sqrt{6}|\vec{x}|/3$ yields that
	\begin{equation}\label{kerw31}
		\begin{aligned}
			&\int_0^t \int_{\mathbb{R}^3} (1+t-\tau)^{-\frac{3+m}{2}} e^{-\frac{|\vec{x}-\vec{y}|^2}{D(t-\tau)}}(1+\tau)^{-\frac{7}{2}} e^{-\frac{2|\vec{y}|^2}{C_{3,k}\tau}} d\vec{y} d\tau\\
			&\le C\left( e^{-\frac{|\vec{x}|^2}{98Dt}}+e^{-\frac{|\vec{x}|^2}{C_{0,k}t}}\right) \left( \int_0^{t/2} (1+t-\tau)^{-\frac{3+m}{2}} \int_{\mathbb{R}^3} (1+\tau)^{-\frac{7}{2}} e^{-\frac{|\vec{y}|^2}{2\nu^2 C_{0,k}\tau}} d\vec{y} d\tau +\int_{t/2}^t (1+\tau)^{-\frac{7}{2}} \int_{\mathbb{R}^3} (1+t-\tau)^{-\frac{3+m}{2}} e^{-\frac{|\vec{x}-\vec{y}|^2}{2D(t-\tau)}} d\vec{y} d\tau\right)\\
			& \le C\left(e^{-\frac{|\vec{x}|^2}{98Dt}}+e^{-\frac{|\vec{x}|^2}{C_{0,k}t}}\right) \left( (1+t)^{-\frac{3+m}{2}} + (1+t)^{-\frac{7}{2}} \begin{cases}
				(1+t)^{1-\frac{m}{2}}, & m=0,1,\\
				\log(1+t), & m=2,\\
				1, & m \ge 3,
			\end{cases}\right) \le C(1+t)^{-\frac{3+\min\{m,4\}}{2}} e^{-\frac{|\vec{x}|^2}{C_{0,k}t}}
		\end{aligned}
	\end{equation}
	since $C_{0,k}\ge 98D$. Moreover, it holds that
	\begin{equation}\label{kerw32}
		\begin{aligned}
			&\int_0^t \int_{\mathbb{R}^3} (1+t-\tau)^{-\frac{3+m}{2}} e^{-\frac{|\vec{x}-\vec{y}|^2}{D(t-\tau)}}(1+\tau)^{\frac{1}{2}} e^{-2(|\vec{y}|+\tau)/D_{3,k}} d\vec{y} d\tau\\
			&\le C\int_0^t \left(\int_{|\vec{y}| \le 2 \nu |\vec{x}|/3} + \int_{|\vec{y}| > 2\nu|\vec{x}|/3}\right)  (1+t-\tau)^{-\frac{3+m}{2}} e^{-\frac{|\vec{x}-\vec{y}|^2}{D(t-\tau)}} e^{-\frac{4|\vec{y}|+3\tau}{2\nu D_{0,k}}} d\vec{y} d\tau\\
			&\le C(1+t)^{-\frac{3+m}{2}} \left( e^{-\frac{4|\vec{x}|^2}{49Dt}} + e^{-\frac{7|\vec{x}|}{6D_{0,k}}} \right) \le C(1+t)^{-\frac{3+m}{2}} \left( e^{-\frac{4|\vec{x}|^2}{49Dt}} + e^{-\frac{7|\vec{x}|^2}{36D_{0,k}t}} + e^{-\frac{|\vec{x}|+t}{D_{0,k}}} \right),
		\end{aligned}
	\end{equation}
	where one uses $|\vec{x}-\vec{y}| \le 2|\vec{x}|/7$ for $|\vec{y}| \le 2 \nu |\vec{x}|/3 = 2\sqrt{6}|\vec{x}|/7 \le 5|\vec{x}|/7$. One could combine \eqref{kerw31}--\eqref{kerw32} with $C_{0,k} \ge 6D_{0,k}$ to verify \eqref{kercon1}.

	For \eqref{kercon2}, one has that
	\begin{equation}\nonumber
		\begin{aligned}
			&\int_0^t \left(\int_{|\vec{y}| \le |\vec{x}|/2} + \int_{|\vec{y}| > |\vec{x}|/2}\right)  (1+t-\tau)^{-\frac{3+m}{2}} e^{-\frac{|\vec{x}-\vec{y}|^2}{D(t-\tau)}} (1+\tau)^{-2} e^{-|\vec{y}|} d\vec{y} d\tau\\
			&\le C\left( e^{-\frac{|\vec{x}|^2}{4Dt}} + e^{-|\vec{x}|/4} \right) \int_0^t (1+t-\tau)^{-\frac{3+m}{2}} (1+\tau)^{-2} d\tau \int_{\mathbb{R}^3} e^{-|\vec{y}|/2} d\vec{y}\\
			&\le C(1+t)^{-\frac{3+\min\{m,1\}}{2}}\left( e^{-\frac{|\vec{x}|^2}{4Dt}} + e^{-\frac{|\vec{x}|^2}{4t}} +e^{-(|\vec{x}|+t)/8}\right),
		\end{aligned}
	\end{equation}
	which verifies \eqref{kercon2}. Here, the final inequality relies on the two cases $|\vec{x}| \le t$ and $|\vec{x}|>t$.
\end{proof}

The next lemma is provided to estimate convolutions involving exponential functions:
\begin{lemma} \label{leexp}
	Let $L_1, L_2$ be positive constants and an integer $d\in\{0,1,2\}$. Then it holds that
	\begin{align}
		&|\vec{x}|^{-d} e^{-|\vec{x}|/L_1-t/L_2} \underset{(\vec{x},t)}{\ast} (1+t)^{1/2} w_3^k(\vec{x},t) \cdot w_3^k(\vec{x},t) \le C\left[(1+t)^{-7/2} \left(e^{-|\vec{x}|/(14L_1)}+e^{-\frac{|\vec{x}|^2}{C_{0,k}t}}\right) +e^{-|\vec{x}|/D_{0,k}} \left( e^{-2t/(7L_2)} + e^{-t/D_{0,k}} \right)\right],\label{expw3}\\
		&|\vec{x}|^{-d} e^{-|\vec{x}|/L_1-t/L_2} \underset{(\vec{x},t)}{\ast} (1+t)^{-2} e^{-|\vec{x}|} \le C(1+t)^{-2} \left( e^{-|\vec{x}|/(4L_1)} + e^{-|\vec{x}|/2} \right),\label{expst}
	\end{align}
	and for $i=0,1,2$,
	\begin{equation}\label{expwi}
		|\vec{x}|^{-d} e^{-|\vec{x}|/L_1-t/L_2} \underset{(\vec{x},t)}{\ast} (1+t)^{-3/2} w_i^k(\vec{x},t) \le C\left[ (1+t)^{-7/2}\left( e^{-(1-\mu^{-1})|\vec{x}|/(2L_1)}+e^{-\frac{|\vec{x}|^2}{C_{i+1,k}t}}\right) + e^{-|\vec{x}|/D_{i+1,k}} \left(e^{-(1-\mu^{-1}) t/L_2}+e^{-t/D_{i+1,k}}\right)\right],
	\end{equation}
	with $w_{0}^k(\vec{x},t)$, $w_1^k(\vec{x},t)$, $w_2^k(\vec{x},t)$, $w_3^k(\vec{x},t)$ given by \eqref{w0123}.
\end{lemma}

\begin{proof}
	First, by arguments analogous to \eqref{kerw31}--\eqref{kerw32}, one has that
	\begin{equation}\nonumber
		\begin{aligned}
			&\int_0^t \left(\int_{|\vec{y}| \le \nu \sqrt{2}|\vec{x}|/2} + \int_{|\vec{y}| > \nu \sqrt{2}|\vec{x}|/2}\right)  |\vec{x}-\vec{y}|^{-d} e^{-|\vec{x}-\vec{y}|/L_1-(t-\tau)/L_2} (1+\tau)^{-\frac{7}{2}} e^{-\frac{2|\vec{y}|^2}{C_{3,k}\tau}} d\vec{y} d\tau\\
			&\le C\left(e^{-|\vec{x}|/(14L_1)}+ e^{-\frac{|\vec{x}|^2}{C_{0,k}t}} \right) \int_0^t e^{-(t-\tau)/L_2} (1+\tau)^{-\frac{7}{2}} d\tau \int_{\mathbb{R}^3} |\vec{x}-\vec{y}|^{-d} e^{-|\vec{x}-\vec{y}|/(2L_1)} d\vec{y}\\
			& \le C(1+t)^{-\frac{7}{2}} \left(e^{-|\vec{x}|/(14L_1)}+e^{-\frac{|\vec{x}|^2}{C_{0,k}t}}\right),
		\end{aligned}
	\end{equation}
	and
	\begin{equation}\nonumber
		\begin{aligned}
			&\left( \int_0^{2\nu t/3} + \int_{2 \nu t/3}^t\right) \left(\int_{|\vec{y}| \le \nu |\vec{x}|/2} + \int_{|\vec{y}| > \nu|\vec{x}|/2}\right)  |\vec{x}-\vec{y}|^{-d} e^{-|\vec{x}-\vec{y}|/L_1-(t-\tau)/L_2} (1+\tau)^{\frac{1}{2}} e^{-2(|\vec{y}|+\tau)/D_{3,k}} d\vec{y} d\tau\\
			&\le  C\left( e^{-|\vec{x}|/(7L_1)} + e^{-|\vec{x}|/D_{0,k}} \right) \left( e^{-2t/(7L_2)} + e^{-t/D_{0,k}}\right),
		\end{aligned}
	\end{equation}
	which relies on the inequalities $\nu \sqrt{2}/2 \le 6/7$ and $\nu /2 \le 2\nu/3 \le 5/7$, and the above estimates yield \eqref{expw3}.
	
	Next, a direct computation yields that
	\begin{equation}\nonumber
		\begin{aligned}
			&\int_0^t \left(\int_{|\vec{y}| \le |\vec{x}|/2} + \int_{|\vec{y}| > |\vec{x}|/2}\right)  |\vec{x}-\vec{y}|^{-d} e^{-|\vec{x}-\vec{y}|/L_1-(t-\tau)/L_2}  (1+\tau)^{-2} e^{-|\vec{y}|} d\vec{y} d\tau\\
			&\le \left( e^{-\frac{|\vec{x}|}{4L_1}} + e^{-\frac{|\vec{x}|}{2}} \right) \int_0^t e^{-(t-\tau)/L_2} (1+\tau)^{-2} d\tau \int_{\mathbb{R}^3} |\vec{x}-\vec{y}|^{-d} e^{-\frac{|\vec{x}-\vec{y}|}{2L_1}} d\vec{y} \le C(1+t)^{-2} \left( e^{-\frac{|\vec{x}|}{4L_1}} + e^{-\frac{|\vec{x}|}{2}} \right),
		\end{aligned}
	\end{equation}
	which implies \eqref{expst}.
	
	Finally, since $C_{i+1,k}=\mu^2 C_{i,k}$ and $D_{i+1,k}=\mu D_{i,k}$ for $i=0,1,2$, with $\mu=\bigl(3\sqrt{6}/7\bigr)^{1/3}>1$, one derives
	\begin{equation}\label{expwical}
		\begin{aligned}
			&\int_0^t \int_{\mathbb{R}^3} |\vec{x}-\vec{y}|^{-d} e^{-|\vec{x}-\vec{y}|/L_1-(t-\tau)/L_2}(1+\tau)^{-3/2} \left( (1+\tau)^{-2} e^{-\frac{|\vec{y}|^2}{C_{i,k}\tau}} + e^{-(|\vec{y}|+\tau)/D_{i,k}} \right) d\vec{y} d\tau\\
			&=\int_0^t \left( \int_{|\vec{y}| \le \mu^{-1}|\vec{x}|}+ \int_{|\vec{y}| > \mu^{-1}|\vec{x}|} \right) |\vec{x}-\vec{y}|^{-d} e^{-|\vec{x}-\vec{y}|/L_1-(t-\tau)/L_2}(1+\tau)^{-3/2} \left( (1+\tau)^{-2} e^{-\frac{\mu^2|\vec{y}|^2}{C_{i+1,k}\tau}} + e^{-\mu(|\vec{y}|+\tau)/D_{i+1,k}} \right) d\vec{y} d\tau\\
			&\le C\left[(1+t)^{-7/2}\left( e^{-(1-\mu^{-1})|\vec{x}|/(2L_1)}+e^{-\frac{|\vec{x}|^2}{C_{i+1,k}t}}\right) + e^{-|\vec{x}|/D_{i+1,k}} \left(e^{-(1-\mu^{-1}) t/L_2}+e^{-t/D_{i+1,k}}\right)\right]
		\end{aligned}
	\end{equation}
	and verifies \eqref{expwi}.
\end{proof}

The last lemma concerns the convolution involving $\mathbb{S}^k$ and $\mathbb{C}^k$, which are defined in \eqref{skck}. The proof relies on Kirchhoff's formulas \eqref{sg}--\eqref{cg}:
\begin{lemma}\label{leskck}
	Let $a$ and $c$ be positive constants. Then one has that
	\begin{equation}\label{skckw3}
		\begin{aligned}
			&\int_0^t e^{-a(t-\tau)/4} \int_{|\vec{y}|=1} (1+\tau)^{1/2} \left(w_3^k\right)^2(\vec{x}+c(t-\tau)\vec{y}, \tau) \, dS(\vec{y}) d\tau\\ &\le C\left[(1+t)^{-7/2} \left( e^{-a|\vec{x}|/(56c)}+e^{-\frac{|\vec{x}|^2}{C_{0,k}t}}\right)+e^{-|\vec{x}|/D_{0,k}} \left(e^{-at/28} + e^{-t/D_{0,k}}\right)\right], 
		\end{aligned}
	\end{equation}
	\begin{equation} \label{skckst}
		\int_0^t e^{-a(t-\tau)/4} \int_{|\vec{y}|=1} (1+\tau)^{-2} e^{-|\vec{x}+c(t-\tau)\vec{y}|} \, dS(\vec{y}) d\tau \le C(1+t)^{-2} \left(e^{-a|\vec{x}|/(16c)} +e^{-|\vec{x}|/2} \right),
	\end{equation}
	and for $i=0,1,2$,
	\begin{equation}\label{skckwi}
		\begin{aligned}
			&\int_0^t e^{-a(t-\tau)/4} \int_{|\vec{y}|=1} (1+\tau)^{-3/2} w_i^k(\vec{x}+c(t-\tau)\vec{y}, \tau) \, dS(\vec{y}) d\tau\\
			&\le C\left[ (1+t)^{-7/2} \left( e^{-a(1-\mu^{-1})|\vec{x}|/(8c)}+e^{-\frac{|\vec{x}|^2}{C_{i+1,k}t}}\right)+e^{-|\vec{x}|/D_{i+1,k}}\left(e^{-a(1-\mu^{-1}) t/8}+e^{-t/D_{i+1,k}}\right)\right].
		\end{aligned}
	\end{equation}
\end{lemma}

\begin{proof}
	First, for $c(t-\tau) \le \bigl(1-\sqrt{2}\nu/2\bigr)|\vec{x}|$ and $|\vec{y}|=1$, it holds that
	$|\vec{x}+c(t-\tau)\vec{y}| \ge |\vec{x}|-c(t-\tau) \ge \sqrt{2}\nu |\vec{x}|/2 > \nu |\vec{x}|/2$;
	and for $c(t-\tau) > (1-\sqrt{2}\nu/2)|\vec{x}| > |\vec{x}|/7$, one has that
	$e^{-a(t-\tau)/8} \le e^{-a|\vec{x}|/(56c)}$; thus it follows that
	\begin{equation}\nonumber
		\begin{aligned}
			\int_0^t e^{-a(t-\tau)/4} \int_{|\vec{y}|=1}  (1+\tau)^{-\frac{7}{2}} e^{-\frac{2|\vec{x}+c(t-\tau)\vec{y}|^2}{C_{3,k}\tau}} dS(\vec{y}) d\tau
			&\le C\left( e^{-a|\vec{x}|/(56c)}+e^{-\frac{|\vec{x}|^2}{C_{0,k}t}}\right)
			\int_0^t e^{-a(t-\tau)/8}(1+\tau)^{-\frac{7}{2}} d\tau \int_{|\vec{y}|=1} 1 dS(\vec{y})\\
			&\le C(1+t)^{-\frac{7}{2}} \left( e^{-a|\vec{x}|/(56c)}+e^{-\frac{|\vec{x}|^2}{C_{0,k}t}}\right),
		\end{aligned}
	\end{equation}
	and
	\begin{equation}\nonumber
		\begin{aligned}
			&\int_0^t e^{-a(t-\tau)/4} \int_{|\vec{y}|=1}  (1+\tau)^{\frac{1}{2}} e^{-2(|\vec{x}+c(t-\tau)\vec{y}|+\tau)/D_{3,k}} dS(\vec{y}) d\tau\\
			&\le C\left( e^{-a|\vec{x}|/(56c)}+e^{-|\vec{x}|/D_{0,k}}\right)
			\left( \int_0^{2\nu t/3} + \int_{2\nu t/3}^t \right)  e^{-a(t-\tau)/8} e^{-\frac{3\tau}{2\nu D_{0,k}}} d\tau \int_{|\vec{y}|=1} 1 dS(\vec{y})\\
			&\le C\left( e^{-a|\vec{x}|/(56c)}+e^{-|\vec{x}|/D_{0,k}}\right)\left(e^{-at/28} + e^{-t/D_{0,k}}\right),
		\end{aligned}
	\end{equation}
	where the inequality $2\nu/3=2\sqrt{6}/7 <5/7$ holds, and \eqref{skckw3} is verified.
	
	Furthermore, from $|\vec{x}+c(t-\tau)\vec{y}| \ge |\vec{x}|-c(t-\tau) \ge |\vec{x}|/2$ for $c(t-\tau) \le |\vec{x}|/2$, and $(t-\tau)/8 >|\vec{x}|/(16c)$ for $c(t-\tau) > |\vec{x}|/2$, one derives \eqref{skckst} as follows:
	\begin{equation}\nonumber
		\begin{aligned}
			\int_0^t e^{-a(t-\tau)/4} \int_{|\vec{y}|=1}  (1+\tau)^{-2} e^{-|\vec{x}+c(t-\tau)\vec{y}|} dS(\vec{y}) d\tau &\le C\left(e^{-a|\vec{x}|/(16c)} +e^{-|\vec{x}|/2} \right) \int_0^t e^{-a(t-\tau)/8} (1+\tau)^{-2} d\tau \int_{|\vec{y}|=1} 1 dS(\vec{y})\\
			&\le C(1+t)^{-2} \left(e^{-a|\vec{x}|/(16c)} +e^{-|\vec{x}|/2} \right).
		\end{aligned}
	\end{equation}
	
	Finally, one splits the time integral over the regions $c(t-\tau) \le (1-\mu^{-1})|\vec{x}|$ and $c(t-\tau) > (1-\mu^{-1})|\vec{x}|$, follows arguments similar to those for \eqref{expwical}, and makes use of $C_{i+1,k}=\mu^2 C_{i,k}, D_{i+1,k}=\mu D_{i,k}$ to obtain that
	\begin{equation}\nonumber
		\begin{aligned}
			&\int_0^t e^{-a(t-\tau)/4} \int_{|\vec{y}|=1} (1+\tau)^{-3/2} \left( (1+\tau)^{-2} e^{-\frac{|\vec{x}+c(t-\tau)\vec{y}|^2}{C_{i,k}\tau}} + e^{-(|\vec{x}+c(t-\tau)\vec{y}|+\tau)/D_{i,k}} \right) dS(\vec{y}) d\tau\\
			&\le C\left[ \left( e^{-a(1-\mu^{-1})|\vec{x}|/(8c)}+e^{-\frac{|\vec{x}|^2}{C_{i+1,k}t}}\right) \int_0^t e^{-a(t-\tau)/8} (1+\tau)^{-7/2} d\tau +e^{-|\vec{x}|/D_{i+1,k}} \int_0^t e^{-a(t-\tau)/8} e^{-\mu \tau/D_{i+1,k}} d\tau \right]\\
			&\le C\left[(1+t)^{-7/2} \left( e^{-a(1-\mu^{-1})|\vec{x}|/(8c)}+e^{-\frac{|\vec{x}|^2}{C_{i+1,k}t}}\right)+e^{-|\vec{x}|/D_{i+1,k}}\left(e^{-a(1-\mu^{-1}) t/8}+e^{-t/D_{i+1,k}}\right)\right], 
		\end{aligned}
	\end{equation}
	which directly yields \eqref{skckwi}.
\end{proof}

\section*{Acknowledgement}
The authors are supported by National Key R\&D Program of China (No. 2022YFA1007300).

\end{document}